\documentclass[11pt]{article}

\usepackage[T1]{fontenc}
\usepackage[latin1]{inputenc}
\usepackage{amsfonts,amsmath,amssymb,amsthm,color,enumerate,graphicx}
\usepackage{dsfont}
\usepackage[colorlinks,urlcolor=blue]{hyperref}
\usepackage{bbm}
\usepackage{fullpage}
\usepackage[numbers]{natbib}
\newcommand{\abs}[1]{\lvert #1 \rvert}

\renewcommand{\vector}[1]{\mathbf{#1}}

\newcommand{\PP}{\mathcal{P}}

\newcommand{\N}{\mathbb{N}}

\newcommand{\Branch}{\mathbf{B}}
\newcommand{\Graphseq}{\mathbf{G}}
\newcommand{\Grid}{\mathbf{Grid}}
\newcommand{\Complete}{\mathbf{K}}
\newcommand{\Path}{\mathbf{P}}
\newcommand{\Hypercube}{\mathbf{Q}}

\DeclareMathOperator{\ex}{ex}

\DeclareMathOperator{\Binom}{Binom}
\DeclareMathOperator{\Ent}{Ent}
\DeclareMathOperator{\Forb}{Forb}

\numberwithin{equation}{section}

\newtheorem{theorem}{Theorem}[section]
\newtheorem{lemma}[theorem]{Lemma}
\newtheorem{corollary}[theorem]{Corollary}

\newtheorem{proposition}[theorem]{Proposition}

\newtheorem{observation}[theorem]{Observation}

\theoremstyle{definition}
\newtheorem{definition}[theorem]{Definition}
\newtheorem{example}[theorem]{Example}

\newtheorem{problem}[theorem]{Problem}

\theoremstyle{remark}
\newtheorem{remark}[theorem]{Remark}

\title{Multicolour containers, extremal entropy and counting}

\author{Victor Falgas-Ravry \\
\small Ume{\aa} Universitet\\
\small \tt victor.falgas-ravry@umu.se\\
\and
Kelly O'Connell \\
\small Vanderbilt University\\ 
\small \tt kelly.m.oconnell@vanderbilt.edu
\and 
Andrew Uzzell\\
\small University of Nebraska--Lincoln\\ 
\small \tt andrew.uzzell@unl.edu}

\begin{document}
\maketitle
	\begin{abstract}
		In breakthrough results, Saxton--Thomason and Balogh--Morris--Samotij developed powerful theories of hypergraph containers. 

In this paper, we explore some consequences of these theories. We use a simple container theorem of Saxton--Thomason and an entropy-based framework to deduce container and counting theorems  for hereditary properties of $k$-colourings of very general objects, which include both vertex-- and edge--colourings of general hypergraph sequences 
as special cases.

In the case of sequences of complete graphs, we further derive characterisation and transference results for hereditary properties in terms of their stability families and extremal entropy. This covers within a unified framework a great variety of combinatorial structures, some of which had not previously been studied via containers: directed graphs, oriented graphs, tournaments, multigraphs with bounded multiplicity and multicoloured graphs amongst others.

Similar results were recently and independently obtained by Terry.
	\end{abstract}

{
  \hypersetup{linkcolor=black}
  \tableofcontents
}
	\section{Introduction}
	\subsection{Notation and basic definitions}
	Given a natural number~$r$, we write $A^{(r)}$ for the collection of all subsets of $A$ of size~$r$. We denote the powerset of~$A$ by $2^{A}$ and the collection of nonempty subsets of $A$ by $2^A\setminus\{\emptyset\}$.
An \emph{$r$-uniform hypergraph}, or \emph{$r$-graph}, is a pair~$G=(V,E)$, where $V=V(G)$ is a set of \emph{vertices} and $E=E(G)\subseteq V^{(r)}$ is a set of \emph{$r$-edges}. We write `graph' for `$2$-graph' and, when there is no risk of confusion, `edge' for `$r$-edge'. We let $e(G):=\vert E(G)\vert$ denote the \emph{size} of~$G$ and  $v(G):=\vert V(G)\vert$ denote its \emph{order}.

	A \emph{subgraph} of an $r$-graph~$G$ is an $r$-graph~$H$ with $V(H)\subseteq V(G)$ and $E(H)\subseteq E(G)$. We use $H\leq G$ to denote that $H$ is a subgraph of $G$. Given a set of vertices~$A\subseteq V(G)$, the subgraph of~$G$ \emph{induced} by $A$ is $G[A]:=(A, E(G)\cap A^{(r)})$. A set of vertices~$A$ is \emph{independent} in $G$ if the subgraph it induces contains no edges. The \emph{degree} of a vertex $v\in V(G)$ is the number of $r$-edges of $G$ containing $v$.
Finally an \emph{isomorphism} between $r$-graphs $G_1$ and~$G_2$ is a bijection~$\phi: \ V(G_1)\rightarrow V(G_2)$ which sends edges to edges and non-edges to non-edges.

Let $[n]:=\{1,2,\ldots, n\}$. A property $\mathcal{P}$ of (labelled) $r$-graphs is a sequence $(\mathcal{P}_n)_{n \in \N}$, where $\mathcal{P}_n$ is a collection of $r$-graphs on the \emph{labelled} vertex set $[n]$.  Hereafter, we do not distinguish between a property~$\PP$ and the class of all $r$-graphs belonging to $\PP_n$ for some $n\in \mathbb{N}$.  An $r$-graph property is \emph{symmetric} if it is closed under relabelling of the vertices, i.e.\ under permutations of the vertex set $[n]$.  A symmetric $r$-graph property is \emph{monotone (decreasing)} if for every $r$-graph~$G\in \mathcal{P}$, every subgraph~$H$ of~$G$ is isomorphic to an element of~$\mathcal{P}$. A symmetric $r$-graph property is \emph{hereditary} if for every $r$-graph~$G\in \mathcal{P}$ every \emph{induced} subgraph~$H$ of~$G$ is isomorphic to an element of $\mathcal{P}$. Note that every monotone property is hereditary, but that the converse is not true. For example, the property of not containing a $4$-cycle as an induced subgraph is hereditary but not monotone.

In order to encode certain combinatorial objects of interest, such as directed graphs, we will consider a weaker notion of symmetry.
\begin{definition}
Let $m$,~$n\in\N$ with $m\leq n$. An \emph{order-preserving} map from $[m]$ to $[n]$ is a function~$\phi: \ [m]\rightarrow [n]$ such that $\phi(i)\leq\phi(j)$ whenever $i\leq j$. Given $e\in [m]^{(r)}$, we write $\phi(e)$ for the set~$\phi(e)=\{\phi(v): \ v\in e\}$.

Given $r$-graphs $G_1$ on $[m]$ and~$G_2$ on $[n]$, we say that $G_2$ contains $G_1$ as an \emph{ordered subgraph} if there is an order-preserving isomorphism from $G_1$ to an $m$-vertex subgraph~$H$ of~$G_2$. We further say that $G_2$ contains $G_1$ as an \emph{induced} orderered subgraph if the $m$-vertex subgraph~$H$ in question is an induced subgraph of~$G_2$.

An $r$-graph property $\mathcal{P}$ is said to be \emph{order-hereditary} if for every $G\in \mathcal{P}_n$ and every order-preserving injection $\phi: \ [m]\rightarrow [n]$, the graph~$G_{\vert \phi}=([m], \{e: \ \phi(e)\in E(G)\})$ is a member of~$\mathcal{P}_m$. 	
\end{definition}
Clearly, every hereditary property is order-hereditary, but the converse is not true. As an example, consider the property~$\mathcal{P}$ of not containing an increasing path of length $2$, that is, the collection of graphs on $[n]$ ($n\in\N$) not containing vertices $i<j<k$ such that $ij$ and $jk$ are both edges. This is order-hereditary, but not symmetric --- and, as we shall see in Section~\ref{section: examples}, is much larger than the symmetric monotone property of not containing a path of length~$2$.

We use standard Landau notation throughout this paper, which we recall here. Given functions $f$,~$g:\N\to\mathbb{R}$, we have $f=O(g)$ if there exists a constant $C>0$ such that $\limsup_{n\to\infty}{f(n)}/{g(n)}\leq C$. If $\lim_{n\to\infty}{f(n)}/{g(n)}=0$, then we write $f=o(g)$. We write $f=\Omega (g)$ and $f=\omega(g)$ to denote $g=O(f)$ and $g=o(f)$ respectively. If we have both $f=O(g)$ and $f=\Omega(g)$, we say that $f$ and $g$ are of the same order and denote this by $f=\theta(g)$. We also use $f\ll g$ and $f \gg g$ as alternatives to $f=o(g)$ and $f=\omega(g)$, respectively. Finally, we say that a sequence of events~$A_n$ occurs \emph{with high probability (whp)} if $\lim_{n\rightarrow \infty} \mathbb{P}(A_n)=1$.

	\subsection{Background: speeds of hereditary graph properties}\label{section: background hereditary props}
The problem of counting and characterising graphs in a given hereditary property $\mathcal{P}$ has a long and distinguished history. The \emph{speed} $n\mapsto \vert \mathcal{P}_n \vert$ of a graph property was introduced in 1976 by Erd{\H o}s, Kleitman and Rothschild~\cite{ErdosKleitmanRothschild76}. Together with the structural properties of a `typical' element of $\mathcal{P}_n$, it has received extensive attention from the research community.

Early work focused on the case where $\mathcal{P}=\Forb(F)$, the monotone decreasing property of not containing a fixed graph $F$ as a subgraph. We refer to the graphs in $\Forb(F)$ as \emph{$F$-free} graphs. The \emph{Tur\'an number} of~$F$, a function of $n$ denoted by $\ex(n, F)$, is the maximum number of edges in an $F$-free graph on $n$ vertices. Clearly, any subgraph of an $F$-free graph is also $F$-free. This gives the following lower bound on the number of $F$-free graphs on $n$ labelled vertices:
\[\Forb(F)_n \geq 2^{\ex(n, F)}.\]
Erd{\H o}s, Kleitman and R{\"o}dl~\cite{ErdosKleitmanRothschild76} showed that if $F=K_t$, the complete graph on $t$ vertices, then the exponent in this lower bound is asymptotically tight:
\[\Forb(K_t)_n \leq 2^{\bigl(1+o(1)\bigr)\ex(n, K_t)}.\]
Their work was generalised by Erd{\H o}s, Frankl and R{\"o}dl~\cite{ErdosFranklRodl86} to the case of arbitrary forbidden subgraphs~$F$ and by Pr\"omel and Steger~\cite{PS92}, who considered the property $\Forb^{*}(F)$ of not containing $F$ as an induced subgraph.  Finally, Alekseev~\cite{Alekseev93} and Bollob\'as--Thomason~\cite{BT} independently determined the asymptotics of the logarithm of the speed for any hereditary property in terms of its \emph{colouring number}, which we now define.
\begin{definition}\label{definition: colouring number}
For each $r\in \N$ and $\vector{v}\in \{0,1\}^r$, let  $\mathcal{H}(r,\vector{v})$ be the collection of all graphs $G$ such that $V(G)$ may be partitioned into $r$ disjoint sets $A_1$, \dots,~$A_r$ such that for each $i$, $G[A_i]$ is an empty graph if $v_i=0$ and a complete graph if $v_i=1$. The \emph{colouring number} $\chi_c(\mathcal{P})$ of a hereditary property is defined to be
\[\chi_c(\mathcal{P}):=\sup\bigl\{r\in \N \, : \, \mathcal{H}(r,\vector{v})\subseteq \mathcal{P} \textrm{ for some }\vector{v} \in \{0,1\}^r\bigr\}.\]
\end{definition}
\begin{theorem}[Alekseev--Bollob\'as--Thomason Theorem]\label{theorem: alekseevbollobasthomason}
If\/ $\mathcal{P}$ is a hereditary property of graphs with $\chi_c(\mathcal{P})=r$, then
\[\lim_{n\rightarrow \infty} \dfrac{\log_2\vert\mathcal{P}_n\vert}{\binom{n}{2}}= 1-\dfrac{1}{r}.\]
\end{theorem}
Subsequently, the rate of convergence of $\log_2 \vert \mathcal{P}_n\vert/\binom{n}{2}$ and the structure of typical graphs were investigated by Balogh, Bollob\'as and Simonovits~\cite{BBS04, BBS09} for monotone properties, and by Alon, Balogh, Bollob\'as and Morris~\cite{ABBM11} for hereditary properties.

There has also been interest in the speed of monotone properties in other discrete structures. Kohayakawa, Nagle and R\"odl~\cite{KohayakawaNagleRodl03}, Ishigami~\cite{Ishigami07}, Dotson and Nagle~\cite{DotsonNagle09} and Nagle, R\"odl and Schacht~\cite{NagleRodlSchacht06} investigated the speed of hypergraph properties, while in a series of papers Balogh, Bollob\'as and Morris~\cite{BaloghBollobasMorris06, BaloghBollobasMorris07b, BaloghBollobasMorris07} studied the speed of properties of ordered graphs, oriented graphs and tournaments. Many of these results relied on the use of graph and hypergraph regularity lemmas. See the survey of Bollob\'as~\cite{Bollobas07} for an overview of the state of the area \emph{before} the breakthroughs discussed in the next subsection.
\subsection{Background: transference and containers}
Recently, there has been great interest in \emph{transference} theorems, in which central results of extremal combinatorics are shown to also hold in `sparse random' settings.   These results are motivated by, \emph{inter alia}, the celebrated Green--Tao theorem on arithmetic progressions in the primes~\cite{GreenTao08} and the K{\L}R conjecture of Kohayawa, {\L}uczak and R\"odl~\cite{KLR} and its applications.

Very roughly, the K{\L}R conjecture says the following: let $H$ be an arbitrary graph, and let $p=p(n)$ be such that with high probability the Erd{\H o}s--R\'enyi random graph~$G_{n,p}$ contains many more copies of $H$ than edges (so that in particular it cannot be made $H$-free by deleting a negligible proportion of the edges). Then all but an exponentially small proportion of the graphs obtained by replacing each vertex $x\in V(H)$ by an $n$-set $V_x$ and each edge $xy \in E(H)$ by a `sparse' $(\varepsilon, p)$-regular bipartite graph between $V_x$ and $V_y$ contain a `canonical' copy of $H$ (a copy taking exactly one vertex from each of the parts $(V_x)_{x\in V(H)}$). This implies in particular that with high probability the partitions obtained from applications of a sparse regularity lemma to a sufficiently dense subgraph of an Erd{\H o}s--R\'enyi random graph satisfy `sparse' analogues of the embedding lemmas for Szemer\'edi regularity partitions. See~\cite{ConlonGowersSamotijSchacht14} for a more rigorous statement and discussion of the conjecture and its applications.


In major breakthroughs a little over five years ago, Conlon and Gowers~\cite{ConlonGowers10} and independently 
Schacht~\cite{Schacht2009} proved very general transference results, yielding important corollaries of the K{\L}R conjecture. Their work was soon followed by another dramatic breakthrough: Balogh, Morris and Samotij~\cite{BaloghMorrisSamotij15} and independently Saxton and Thomason~\cite{SaxtonThomason15}, building on work of Kleitman--Winston~\cite{KleitmanWinston82} and of Sapozhenko~\cite{Sapozhenko87,Sapozhenko01} for graphs, developed powerful theories of hypergraph containers.

These container theories imply that hereditary properties of graphs and hypergraphs can be `covered' by `small' families of `containers', which are themselves `almost in the property'. We discuss containers with more precision and details in Section~\ref{section: containers}. As an application of their theories, Balogh--Morris--Samotij and Saxton--Thomason gave both new proofs of known counting/characterisation results and many new counting/characterisation results for hereditary properties, and in addition a spate of transference results.
In particular Balogh--Morris--Samotij and Saxton--Thomason settled the K{\L}R conjecture in full generality --- 
see the excellent ICM survey of Conlon~\cite{Conlon14} for an in-depth discussion of some of the recent groundbreaking progress made by researchers in the area.

\subsection{Background: entropy and graph limits}
A parallel but separate development at the intersection of extremal combinatorics and discrete probability has been the rise of theories of limit objects for sequences of discrete structures. One of several approaches to limits of sequences of graphs (see~\cite{Austin08} for a description of other approaches and the links between them) is the theory of left convergence and graphons, which is developed at length in the monograph of Lov\'asz~\cite{LovaszBook}.

Recently, Hatami, Janson and Szegedy~\cite{HJS} defined and studied the entropy of a graphon.  They used this notion to recover Theorem~\ref{theorem: alekseevbollobasthomason} and to describe the typical structure of a graph in a hereditary property. In a separate paper~\cite{FalgasRavryStrombergUzzell17} (see also~\cite{FalgasRavryOConnellStrombergUzzell16}), Str\"omberg and a subset of the authors of the present article follow the Hatami--Janson--Szegedy approach to recover some of the main results in this paper using the entropy of decorated graphons rather than multicolour containers as their main tool. We discuss this briefly in Section~\ref{subsection: discussion of containers/limits}.

Let us however note here that the Hatami--Janson--Szegedy notion of entropy can be viewed as a graphon analogue of the classical notion of the entropy of a discrete random variable, which first appeared in Shannon's foundational paper~\cite{Shannon48}. Using entropy to count objects is an old and celebrated technique in discrete probability --- see for example Galvin~\cite{Galvin14} for an exposition of the applications of entropy to counting. This provides a natural motivation for the arguments in this paper.

	\subsection{Contributions of this paper}\label{subsection: results}
In this paper, we explore consequences of the container theories of Balogh--Morris--Samotij and Saxton--Thomason. We use existing container theorems, together with ideas of Saxton--Thomason and Balogh--Wagner and an entropy-based framework, to deduce container and counting theorems for general hereditary properties of $k$-colourings of very general objects (set-sequences equipped with embeddings, or \emph{ssee}, see Definition~\ref{definition: possee}). As special cases relevant in many applications, our results cover vertex-- and edge-colourings of graph and hypergraph sequences; examples of such sequences include hypercube graphs, multipartite graphs and grid graphs amongst others.

In the case of sequences of complete graphs, we further derive characterisation and transference results for order-hereditary properties in terms of their stability families and extremal entropy.   Amongst other structures of interest, these latter results cover $k$-coloured graphs, directed graphs, oriented graphs, tournaments and multipartite graphs with bounded multiplicity.

As we restrict ourselves to the study of `dense' properties, our container statements and their corollaries are (we believe) simple and easy to apply (albeit weaker than the full strength of the Balogh--Morris--Samotij and Saxton--Thomason container theorems), which we hope may be useful to other researchers. In particular the corollaries (counting/characterisation/transference) are very general `assumption-free' statements, which do not require checking any codegree condition or even any knowledge of container theory.

We  give a number of examples and applications.  First, we give a very short proof of the Alekseev--Bollob\'as--Thomason theorem.  Second, we solve a problem of K\"uhn, Osthus, Townsend and Zhao~\cite{KuhnOsthusTownsendZhao14} on $H$-free digraphs.  Third, we prove a counting result for multigraphs in which no triple of vertices supports more than four edges (this is a special case of recent and much more general results of Mubayi and Terry~\cite{MubayiTerry16a, MubayiTerry16b}).  Fourth, we prove counting and stability results for $3$-coloured graphs with no rainbow triangle, which is a special case of a problem of Erd{\H o}s and Rothschild~\cite{Erdos74}.  Fifth, we determine the asymptotic number of induced subgraphs of the hypercube graph containing no $4$-cycle. 


Our main tools are a container theorem of Saxton--Thomason for linear hypergraphs and the adoption of an entropy-based framework. We should like to emphasise here once more the intellectual debt this paper owes to the pioneering work of Balogh--Morris--Samotij and Saxton--Thomason: our work relies on theirs in a crucial way, and many of our ideas exist already in their papers in an embryonic form, which we explore further. The usefulness of our exploration is vindicated by the fact that some of the applications of containers to other discrete structures which we treat are new, and were not well understood by the mathematical community at the time of writing.

For example, finding a suitable container theorem for digraphs was a problem raised by K\"uhn, Osthus, Townsend and Zhao~\cite{KuhnOsthusTownsendZhao14}, which we resolve in the present paper. The `twist' in our approach is that, following the earlier approaches of Saxton and Thomason~\cite{SaxtonThomason15} and Balogh and Wagner~\cite{BaloghWagner16}, our `containers' are, in essence, collection of random digraph models, rather than the collections of digraphs as had been used previously.  Explicitly, given a digraph property $\mathcal{P}$, we derive the existence of a collection $\mathcal{T}$ of random digraph models  in which the state of each pair of vertices is independent of the rest and such that (i) every digraph in $\mathcal{P}$ occurs with strictly positive probability as the outcome  of some random digraph model from $\mathcal{T}$ (ii) outcomes of random graph models from $\mathcal{T}$ are either digraphs in $\mathcal{P}$ or digraphs close to $\mathcal{P}$ in edit distance, and (iii) $\vert \mathcal{T}\vert$ is small. As we show, the maximum entropy of a random digraph model satisfying (ii) then determines the speed of $\mathcal{P}$. This connection with random graphs is explored in greater detail in~\cite{FalgasRavryStrombergUzzell17}, where containers and graph limits are studied in parallel.


There is a significant overlap between the container, enumeration, and stability results presented here and those obtained independently by Terry~\cite{Terry16} (see the discussion below), although the emphases and arguments of our two papers are quite different. 
Since the first version of this paper was written, Terry has further explored interesting related questions on the different possible speeds in general multicolour properties~\cite{Terry17}, going beyond the case of `dense' properties we consider in the present work.
		
	\subsection{Structure of the paper}\label{subsection: structure of the paper}
Section~\ref{section: containers} gathers together our main results on multicolour containers for colourings of $K_n$. Section~\ref{subsection: key definitions templates and entropy} contains our key definitions of templates and entropy.  In Section~\ref{subsection: containers proof}, we state and derive our first multicolour container theorem (Theorem~\ref{theorem: multi-colour container}), and in Section~\ref{subsection: entropy density, supersaturation} we introduce entropy density and prove a supersaturation result that is key to several of our applications.  In Section~\ref{subsection: approximation of arbitrary hereditary properties}, we use these tools to obtain container theorems for general order-hereditary properties (Corollary~\ref{corollary: containers for arbitrary hereditary properties}) and prove a general counting result (Corollary~\ref{corollary: speed of arbitrary hereditary properties}).  Finally in Sections \ref{subsection: stability} and~\ref{subsection: transference} we obtain general characterisation and transference results (Theorems \ref{theorem: strong stability and containers} and~\ref{theorem: transference}, respectively).

As indicated above, the results of Sections \ref{subsection: containers proof}--\ref{subsection: stability} are very similar to those recently obtained by Terry~\cite{Terry16}.  In particular, Terry's Theorems 2, 3, 6 and~7 correspond to our Proposition~\ref{proposition: entropy density}, Corollary~\ref{corollary: speed of arbitrary hereditary properties}, Theorem~\ref{theorem: multi-colour container} and Lemma~\ref{lemma: supersaturation}, respectively, while Terry's Theorem~5 is very similar to our Theorem~\ref{theorem: strong stability and containers}.  Both Terry's results and our own hold for $r$-uniform hypergraphs (see Section~\ref{subsection: hypergraph sequences}). Terry's results extend even further to any finite relational language, which we do not cover in this paper.

In Section~\ref{section: other structures}, we extend our main results to a number of other discrete structures. 
 Section~\ref{subsection: oriented} describes how our theorems apply to tournaments, oriented graphs and directed graphs; as mentioned earlier, this addresses an issue raised in~\cite{KuhnOsthusTownsendZhao14}. In Section~\ref{subsection: possee} we extend our main results to obtain container theorems for colourings of very general objects, namely set-sequences equipped with embeddings (or \emph{ssee}, see Definition~\ref{definition: possee}). This class of objects include many examples of interest, including both edge-- and vertex--colourings of sequences of graphs and hypergraphs, as well as other structures such as sequences of posets or groups. We restate the main results on containers for ssee-s in the more familiar terms of graph and hypergraph sequences in Section~\ref{subsection: hypergraph sequences}, and then illustrate their implications in Section~\ref{subsection: hypercubes} by deriving general counting results for hereditary properties of vertex-- and edge--colourings of hypercube graphs.

Section~\ref{section: examples} is dedicated to applications of our results to a variety of concrete examples. 
Amongst other things, we give a short proof of the Alekseev--Bollob\'as--Thomason theorem and prove counting and/or characterisation results for some hereditary properties of directed graphs, multigraphs with bounded multiplicitys, $3$-coloured graphs and hypercube graphs.

We end this paper in Section~\ref{section: concluding remarks} with an open problem on the possible structure of entropy maximisers in the multicolour setting and a brief discussion of the links between the container and the graph limit approaches to counting, characterisation and transference.
	\section{Multicolour containers}\label{section: containers}
\subsection{Key definitions: templates and entropy}\label{subsection: key definitions templates and entropy}
Let $K_n$ denote the complete graph $([n], [n]^{(2)})$. We study $k$-colourings of (the edges of) $K_n$, that is to say, we work with the set of colouring functions~$c: \ E(K_n)\rightarrow [k]$. Denote by $[k]^{K_n}$ the set of all $k$-colourings of $K_n$.  Such colourings are of interest as they allow us to encode many important combinatorial structures. Note that each colour~$i$ induces a graph~$c^i$ on $[n]$, where $c^i=([n], c^{-1}(i))$. An ordinary graph~$G$ may thus be viewed as a $2$-colouring of~$E(K_n)$, with $G=c^1$ and its complement~$\overline{G}=c^2$. Similarly, an oriented graph~$\vec{G}$ may be viewed as a $3$-colouring of~$E(K_n)$, in which each edge~$ij$ with $i<j$ is coloured $2$ if $\vec{ij}\in D$, $3$ if $\vec{ji} \in D$ and $1$ otherwise. See  Section~\ref{subsection: oriented} for more examples in this vein.

Write $\binom{K_n}{K_m}$ for the collection of order-preserving injections $\phi: \ [m]\to [n]$. Given $\phi\in \binom{K_n}{K_m}$ and a colouring $c\in [k]^{K_n}$, we write $c_{\vert \phi}$ for the \emph{subcolouring of $c$ induced by $\phi$}, defined by $c_{\vert\phi}(ij) =c(\phi(i)\phi(j))$ for all $ij\in [m]^{(2)}$. Further, we say $c'\in [k]^{K_m}$ is a \emph{subcolouring} of $c\in [k]^{K_n}$ if there exists $\phi \in \binom{K_n}{K_m}$ with $c'=c_{\vert \phi}$.
Our main object of study in this section will be \emph{order-hereditary} properties of~$[k]^{K_n}$.
\begin{definition}\label{definition: order hereditary}
An \emph{order-hereditary property} of $k$-colourings is a sequence~$\mathcal{P}= \left(\mathcal{P}_n\right)_{n\in \N}$, such that:
\begin{enumerate}[(i)]
	\item $\mathcal{P}_n$ is a family of $k$-colourings of~$K_n$,
	\item for every $m\leq n$, $c\in \mathcal{P}_n$ and $\phi\in \binom{K_n}{K_m}$, $c_{\vert \phi}\in \mathcal{P}_{m}$.
\end{enumerate} 		
\end{definition}
A key tool in extending container theory to $k$-coloured graphs will be the notion of a \emph{template}. This was first introduced in the context of container theory by Saxton and Thomason in~\cite[Section~2.4]{SaxtonThomason15} (in the case $k=2$, under the name of ``$2$-coloured multigraphs''), and later by Balogh and Wagner in~\cite[Section 4]{BaloghWagner16} (in the case of general $k$, and simply called ``containers'' in that paper).
\begin{definition}[Template]\label{definition: template}
A \emph{template} for a $k$-colouring of~$K_n$ is a function 
\[t: \ E(K_n)\rightarrow	2^{[k]}\setminus \{\emptyset\},\]
 associating to each edge $e$ of~$K_n$ a non-empty list of colours $t(e)\subseteq [k]$; we refer to $t(e)$ as the \emph{palette} available at $e$.  We write $\left(2^{[k]}\setminus \{\emptyset\} \right)^{K_n}$ for the family of all $k$-colouring templates of $K_n$.

Given a template~$t\in \left(2^{[k]}\setminus \{\emptyset \}\right)^{K_n}$, we say that a $k$-colouring~$c\in[k]^{K_n}$ \emph{realises} $t$ if $c(e)\in t(e)$ for every edge~$e\in E(K_n)$.  We write $\langle t \rangle$ for the collection of realisations of~$t$ and $c\in \langle t\rangle$  as a shorthand for `$c$ is a realisation of $t$'.
\end{definition}
 In other words, a template~$t$ gives, for each edge of~$K_n$, a palette of permitted colours, and $\langle t \rangle$ is the set of $k$-colourings of~$K_n$ that respect the template. We observe that a $k$-colouring of~$K_n$ may itself be regarded as a template, albeit with only~one colour allowed at each edge. We extend our notion of subcolouring to templates in the natural way.
 \begin{definition}\label{definition: subtemplate}
 Let $t\in\left(2^{[k]}\setminus \{\emptyset \}\right)^{K_n}$ and let $\phi \in \binom{K_n}{K_m}$. The \emph{subtemplate of $t$ induced by $\phi$} is the template $t_{\vert \phi} \in \left(2^{[k]}\setminus \{\emptyset \}\right)^{K_m}$ defined by $t_{\vert \phi}(ij)=t(\phi(i)\phi(j))$.

 Given $m\leq n$ and  $k$-colouring templates $t$,~$t'$ for $K_m$,~$K_n$ respectively, we say that $t$ is a \emph{subtemplate} of~$t'$, which we denote by $t\leq t'$, if there exists $\phi \in \binom{K_n}{K_m}$ such that $t(e)\subseteq t'_{\phi}(e)$ for every $e\in E(K_m)$. 
\end{definition}
Our notion of subtemplates can be viewed as the template analogue of the notion of an order-preserving subgraph for graphs on a linearly ordered vertex set.

 Templates enable us to generalise the notion of containers to the $k$-coloured setting.
 \begin{definition}
	Given a family of $k$-colourings $\mathcal{F}$ of $E(K_n)$, a \emph{container family} for $\mathcal{F}$ is a collection~$\mathcal{T}=\{t_1, t_2, \ldots, t_m\}$ of $k$-colouring templates such that for every template $t\in (2^{[k]}\setminus\{\emptyset\} )^{K_n}$ with $\langle t \rangle \subseteq \mathcal{F}$, there exists $t_i\in \mathcal{T}$ with $t\leq t_i$.
\end{definition}
In particular, if $\mathcal{T}$ is a container family for $\mathcal{F}$ then every colouring from $\mathcal{F}$ is a realisation of some template from $\mathcal{T}$ (so the set of realisations from $\mathcal{T}$ `contains' $\mathcal{F}$).

Next we define the key notion of the \emph{entropy} of a template.
\begin{definition}\label{def: entropy}
The \emph{entropy} of a $k$-colouring template $t$ is 
\[\Ent(t):= \log_k\prod_{e\in E(K_n)} \vert t(e)\vert.\]
\end{definition}
For any template~$t$ we have $0\leq \Ent(t)\leq \binom{n}{2}$, and the number of distinct realisations of~$t$ is exactly~$\vert\langle t \rangle\vert= k^{\Ent(t)}$. Observe also that zero-entropy templates correspond to $k$-colourings of~$K_n$, and that if $t$ is a subtemplate of~$t'$ then $\Ent(t)\leq \Ent(t')$. There is a direct correspondence between our notion of entropy and that of Shannon entropy in discrete probability: given a template $t$, we can define a \emph{$t$-random} colouring~$\mathbf{c}_t$, by choosing for each $e\in E(K_n)$ a colour~$\mathbf{c}_t(e)$ uniformly at random from $t(e)$. The entropy of~$t$ defined above is precisely the $k$-ary Shannon entropy of the discrete random variable~$\mathbf{c}_t$.

The notion of $t$-random colouring allows us to view our templates as, in essence, random graph models, and their realizations as random graph outcomes. Using templates/the associated random colourings as containers is key to making multicolour containers work. This idea, due to Saxton and Thomason~\cite{SaxtonThomason15} and Balogh and Wagner~\cite{BaloghWagner16}, will allow us to overcome the obstacles to a container theorem for a particular digraph problem of K\"uhn, Osthus, Townsend and Zhao~\cite{KuhnOsthusTownsendZhao14} (see Section~\ref{subsection: digraph examples}).
\subsection{Container families}\label{subsection: containers proof}

Let $N\in \N$ be fixed and let $\mathcal{F}$ be a nonempty collection of $k$-colourings of~$E(K_N)$. Let $\Forb(\mathcal{F})$ be the collection of all $k$-colourings~$c$ of~$K_n$, $n\in \N$, such that for all $\phi \in \binom{K_n}{K_N}$,  $c_{\vert \phi}\notin \mathcal{F}$. More succinctly, $\Forb(\mathcal{F})$ is the collection of all $k$-colourings avoiding $\mathcal{F}$, which, clearly, is an order-hereditary property of $k$-colourings.

\begin{theorem}\label{theorem: multi-colour container}
Let $N\in \N$ be fixed and let $\mathcal{F}$ be a nonempty collection of $k$-colourings of~$E(K_N)$.  For any $\varepsilon>0$, there exist constants $C_0$,~$n_0>0$, depending only on $(\varepsilon, k, N)$,  such that for any $n\geq n_0$ there exists a collection~$\mathcal{T}_n$ of $k$-colouring templates for $K_n$ satisfying:
\begin{enumerate}[(i)]
\item $\mathcal{T}_n$ is a container family for $\left(\Forb(\mathcal{F})\right)_n$;
\item for each template $t\in\mathcal{T}_n$, there are at most $\varepsilon \binom{n}{N}$ pairs $(\phi,c)$ with $\phi \in \binom{K_n}{K_N}$, $c\in \mathcal{F}$ and $c\in \langle t_{\vert \phi}\rangle$;
\item $\log_k \vert \mathcal{T}_n\vert\leq \frac{C_0}{n^{1/\bigl(2\binom{N}{2}-1\bigr)}} \binom{n}{2}$.
\end{enumerate}
\end{theorem}
In other words, the theorem says that we can find a \emph{small} (property (iii)) collection of templates, that together \emph{cover} $\Forb(\mathcal{F})_n$ (property (i)), and whose realisations are \emph{close} to lying in $\Forb(\mathcal{F})_n$ (property (ii)). 

We shall deduce Theorem~\ref{theorem: multi-colour container} from a hypergraph container theorem of Saxton and Thomason. Say that an $r$-graph~$H$ is \emph{linear} if each pair of distinct $r$-edges of~$H$ meets in at most~$1$ vertex. Saxton and Thomason proved the following:
\begin{theorem}[Saxton--Thomason (Theorem 1.2 in~\cite{SaxtonThomason16})]\label{theorem: saxton thomason simple containers}
	Let $r \geq 2$ and let $0<\delta<1$. There exists $d_0=d_0(r, \delta)$ such that if $G$ is a linear $r$-graph of average degree~$d\geq d_0$, then there exists a collection~$\mathcal{C}$ of subsets of~$V(G)$ satisfying:
	\begin{enumerate}
		\item if $I\subseteq V(G)$ is an independent set, then there exists $C\in \mathcal{C}$ with $I\subseteq C$;
		\item $e(G[C])< \delta e(G)$ for every $C\in \mathcal{C}$;
		\item  $\vert \mathcal{C} \vert \leq 2^{\beta v(G)}$, where $\beta=(1/d)^{1/(2r-1)}$.
	\end{enumerate}  	
\end{theorem}

In the proof of Theorem~\ref{theorem: multi-colour container} and elsewhere, we shall use the following standard Chernoff bound: if $X\sim\Binom(n,p)$, then for any $\delta\in[0,1]$,
\begin{equation}\label{equation: Chernoff} \mathbb{P}\bigl(\vert X-np\vert \geq \delta np\bigr) \leq 2e^{-\frac{\delta^2np}{4}}.\end{equation}

\begin{proof}[Proof of Theorem~\ref{theorem: multi-colour container}]
If $N=2$, then $\mathcal{F}$ just gives us a list of forbidden colours, say $F\subseteq [k]$. Then $\left(\Forb(\mathcal{F})\right)_n$ is exactly the collection of all realisations of the template~$t$ assigning to each edge~$e$ of~$K_n$ the collection~$[k]\setminus F$ of colours not forbidden by $\mathcal{F}$. Thus in this case our result trivially holds, and we may therefore assume $N\ge3$ in the rest of the proof.

We define a hypergraph~$H$ from $\mathcal{F}$ and $K_n$ as follows. Set $r=\binom{N}{2}$. We let the vertex set of $H$ consist of~$k$ disjoint copies of~$E(K_n)$, one for each of our $k$ colours: $V(H)=E(K_n)\times [k]$; this idea, allowing us to apply Theorem~\ref{theorem: saxton thomason simple containers}, first appeared in work of Saxton and Thomason~\cite[Section 2.4]{SaxtonThomason15} (in the $2$-colour case) and of Balogh and Wagner~\cite[Section 4]{BaloghWagner16} (in the $k$-colour case).

For every order-preserving embedding $\phi \in  \binom{K_n}{K_N}$ and every  $k$-colouring~$c \in \mathcal{F}$, we add to $H$ an $r$-edge~$e_{\phi,c}$, where
\[e_{\phi,c}=\bigl\{\bigl(\phi(i)\phi(j), c(ij)\bigr) \, : \, ij\in [N]^{(2)} \bigr\}.\]
This gives us an $r$-graph~$H$. Let us give bounds on its average degree. 
Since $\mathcal{F}$ is nonempty, for every $N$-set $A\subseteq [n]$, there are at least~$1$ and at most~$k^{\binom{N}{2}}$ colourings~$c$ of $A^{(2)}$ which are order-isomorphic to an element of~$\mathcal{F}$. It follows that:
\begin{equation}\label{equation: bounds on e(H)}
\frac{n^N}{N^N} \leq  \binom{n}{N} \leq e(H)\leq k^{\binom{N}{2}}\binom{n}{N}\leq k^{\binom{N}{2}}\left(\frac{en}{N}\right)^N .
\end{equation}
Thus $e(H)$ is of order~$n^N$ and the average degree in $H$ is of order~$n^{N-2}$ (since $v(H)=k\binom{n}{2}$), which tends to infinity as $n\rightarrow \infty$. We are almost in a position to apply Theorem~\ref{theorem: saxton thomason simple containers}, with one caveat: the hypergraph~$H$ we have defined is in no way linear. Following Saxton--Thomason~\cite{SaxtonThomason16}, we circumvent this difficulty by considering a \emph{random sparsification} of~$H$. We note that other approaches are possible, by using the original container theorems of~\cite{BaloghMorrisSamotij15, SaxtonThomason15} and computing co-degrees; this would avoid the need for sparsification and potentially give better bounds on $\vert \mathcal{T}_n$, but make other aspects of our proof --- and our later generalisations --- less transparent. Since we are focussing on `dense' properties in this paper and easily applicable general statements, we do not pursue this here.

Let $\varepsilon_1 \in(0,1)$ be a constant to be specified later and let
\begin{equation}\label{eq: sparsification probability}
p=\varepsilon_1 \big/\left(12k^{2\binom{N}{2}-3}\binom{N}{3}\binom{n-3}{N-3}\right).
\end{equation}
Note that by (\ref{equation: bounds on e(H)}) we have
\begin{equation}\label{equation: bound on pe(H)}
pe(H)\geq p\binom{n}{N}=\Omega(n^3).
\end{equation}
We shall keep each $r$-edge of~$H$ independently with probability~$p$, and delete it otherwise, to obtain a random subgraph~$H'$ of~$H$. Standard probabilistic estimates will then show that with positive probability the $r$-graph~$H'$ is almost linear, has large average degree and respects the density of~$H$. More precisely, we show:
\begin{lemma}\label{lemma: random sparsification}
Let $p$ be as in~\eqref{eq: sparsification probability}, let $H'$ be the random subgraph of~$H$ defined above and consider the following events:
\begin{itemize}
	\item the event $F_1$ that $e(H')\geq \frac{pe(H)}{2}$; 
	\item the event $F_2$ that $H'$ has at most 
	$\frac{\varepsilon_1}{4}p\binom{n}{N}$ pairs of edges~$(e,e')$ with $\vert e\cap e'\vert \geq 2$;
	\item the event $F_3$ that for all $S\subseteq V(H)$ with $e(H[S])\geq \varepsilon_1 e(H)$, we have $e(H'[S])\geq \frac{\varepsilon_1}{2} e(H')$.
\end{itemize}
There exists $n_1=n_1(\varepsilon_1, k, N)\in\N$ such that for all $n\geq n_1$, $F_1 \cap F_2 \cap F_3$ occurs with strictly positive probability.
\end{lemma}
The proof of our lemma follows that of~\cite[Lemma~3.3]{SaxtonThomason16} with minor modifications.
\begin{proof}
By (\ref{equation: bounds on e(H)}), we have $\mathbb{E}e(H')=pe(H)\geq p\binom{n}{N}$. Applying the Chernoff bound (\ref{equation: Chernoff}) with $\delta=1/2$, we get that the probability that $F_1$ fails in $H'$ is at most
\begin{equation*}\label{equation: bound on prob 1. fails}
\mathbb{P}\biggl(e(H')< \frac{1}{2}pe(H)\biggr)\leq 2e^{-\frac{p\binom{n}{N}}{16}}= e^{-\Omega(n^3)},
\end{equation*}
where the last equality follows from (\ref{equation: bound on pe(H)}).
Next consider the pairs of $r$-edges $(e,e')$ in $H$ with $\vert e\cap e'\vert \geq 2$, which we refer to hereafter as \emph{overlapping pairs}.  Let $Y_H$ and $Y_{H'}$ denote the number of overlapping pairs in $H$ and $H'$ respectively. Since one needs at least $3$ vertices to support $2$ distinct edges in $K_n$, $Y_H$ is certainly bounded above by the number of ways of choosing an $N$-set~$A\subseteq [n]$, a $3$-set~$B$ from $A$ and an $(N-3)$-set~$A'$ from $[n]\setminus B$ (thereby making an overlapping pair of $N$-sets~$(A, A'\cup B)$) and assigning an arbitrary $k$-colouring to the edges in $A^{(2)}\cup ({A'}\cup B)^{(2)}$. Thus,
\begin{equation*} Y_H \leq \binom{n}{N}\binom{N}{3} \binom{n-3}{N-3}k^{2\binom{N}{2}-3}
\end{equation*}
and
\begin{equation*}
\mathbb{E}(Y_{H'}) = p^2 Y_H \leq p^2\binom{n}{N}\binom{N}{3} \binom{n-3}{N-3}k^{2\binom{N}{2}-3} =\frac{\varepsilon_1}{12} p\binom{n}{N}.
\end{equation*} 
Applying Markov's inequality, we have with probability at least~$\frac{2}{3}$ that $Y_{H'} \leq \frac{\varepsilon_1}{4} p\binom{n}{N}$ (and thus $F_2$) holds.

Finally, consider a set $S\subseteq V(H)$ with $e(H[S])\geq \varepsilon_1 e(H)$. Applying the Chernoff bound~\eqref{equation: Chernoff} with $\delta=1-1/\sqrt{2}$ and the lower bound~\eqref{equation: bound on pe(H)} for $pe(H)$, we get
\begin{equation}\label{eq: too few induced edges}
\mathbb{P}\left(e\bigl(H'[S]\bigr)\leq\frac{1}{\sqrt {2}}\mathbb{E}e\bigl(H'[S]\bigr)\right) \leq 2e^{-(1-\frac{1}{\sqrt{2}})^2\frac{\mathbb{E} e(H'[S])}{4}}= e^{-\Omega( p\varepsilon_1 e(H))}= e^{-\Omega(n^3)}.
\end{equation}
Moreover, by \eqref{equation: Chernoff} with $\delta = \sqrt{2}-1$ and~\eqref{equation: bound on pe(H)} again,
\begin{equation}\label{eq: too many induced edges}
\mathbb{P}\left(e(H')\ge\sqrt{2}\mathbb{E} e(H')\right)\leq 2e^{-\frac{(\sqrt{2}-1)^2p e(H)}{4}}= e^{-\Omega(n^3)}.
\end{equation}
Say that a nonempty set~$S\subseteq V(H)$ is \emph{bad} if $e(H[S])\geq \varepsilon_1 e(H)$ and $e(H'[S])\leq \frac{\varepsilon_1}{2} e(H')$.   By \eqref{eq: too few induced edges},~\eqref{eq: too many induced edges} and the union bound, the probability that $F_3$ fails, i.e., that there exists some bad $S\subseteq V(H)$, is at most
\begin{equation*}
\mathbb{P}\bigl(\exists \textrm{ bad }S\bigr) \leq \mathbb{P}\bigl(e(H')\geq \sqrt{2}\mathbb{E} e(H')\bigr)+\sum_S \mathbb{P}\biggl(e(H'[S]\leq\frac{1}{\sqrt {2}}\mathbb{E}e\bigl(H'[S]\bigr)\biggr)\leq 2^{k\binom{n}{2}} e^{-\Omega(n^3)}=e^{-\Omega(n^3)}.
\end{equation*}

Therefore with probability at least~$2/3-o(1)$ the events $F_1$, $F_2$ and~$F_3$ all occur, and in particular they must occur simultaneously with strictly positive probability for all $n\geq n_1=n_1(\varepsilon_1, k, N)$. 
\end{proof}
By Lemma~\ref{lemma: random sparsification}, for any $\varepsilon>0$ and $\varepsilon_1=k^{-\binom{N}{2}}\varepsilon$ fixed and any $n\geq n_1(\varepsilon_1, k, N)$, there exists a sparsification~$H'$ of~$H$ for which the events $F_1$, $F_2$ and~$F_3$ from the lemma all hold.  Deleting one $r$-edge from each overlapping pair in $H'$, we obtain a linear $r$-graph~$H''$ with average degree~$d$ satisfying
\begin{equation}\label{eq: linear hypergraph average degree}
d=\frac{re(H'')}{v(H'')} \geq \frac{\binom{N}{2}}{k\binom{n}{2}}\left(e(H') - Y_{H'}\right)
\geq \frac{\binom{N}{2}}{k\binom{n}{2}}\biggl(\frac{1}{2} -\frac{\varepsilon_1}{4}\biggr)p\binom{n}{N}=\Omega(\frac{n^{N}}{n^2n^{N-3}} )=\Omega(n).
\end{equation}
We are now in a position to apply the container theorem for linear $r$-graphs, Theorem~\ref{theorem: saxton thomason simple containers}, to $H''$. Let $\delta=\delta(\varepsilon_1)$ satisfy $0<\delta <\varepsilon_1/4$ and let $d_0=d_0(\delta, r)$ be the constant in Theorem~\ref{theorem: saxton thomason simple containers}. For $n\geq n_2(k, N, \delta)$ sufficiently large, we have $d\geq d_0$. Thus there exists a collection~$\mathcal{C}$ of subsets of~$V(H'')=V(H)$ satisfying conclusions 1.--3.\ of Theorem~\ref{theorem: saxton thomason simple containers}.

For each $C\in \mathcal{C}$, we obtain a template~$t=t(C)$ for a \emph{partial} $k$-colouring of~$K_n$ as follows: for each edge $e$, we are given a palette $t(e)=\{i\in[k]: \ (e,i)\in C\}$ of available colours (note that some edges may have the empty palette). 
Set
\[\mathcal{T}:=\{t(C): \ C\in \mathcal{C}, \ t(e)\neq \emptyset  \textrm{ for all }e\in E(K_n)  \}\]
to be the family of templates from $\left(2^{[k]}\setminus \{\emptyset\}\right)^{K_n}$ which can be constructed in this way.
We claim that the template family~$\mathcal{T}$ satisfies conclusions (i)--(iii) of Theorem~\ref{theorem: multi-colour container}.

Indeed, by definition of~$H$, any template $t'$ with $\langle t\rangle\subseteq \mathcal{P}_n$ gives rise to an independent set $I$ in the $r$-graph~$H$ and hence its subgraph~$H''$, namely $I=\{(e,i): \ i\in [k], \ e\in E(K_n), \ i\in t(e)\}$. Thus there exist $C\in \mathcal{C}$ with $I\subseteq C$, giving rise to a proper template~$t \in \mathcal{T}$ with $t'\leq t$. Conclusion~(i) is therefore satisfied by $\mathcal{T}$.

Further for each $C\in \mathcal{C}$, conclusion 2.\ of Theorem~\ref{theorem: saxton thomason simple containers} and the event~$F_2$ together imply
\begin{equation*}
e\bigl(H'[C]\bigr)\leq e\bigl(H''[C]\bigr) + \left(e(H')-e(H'')\right)\leq \delta e(H'') + \frac{\varepsilon_1}{4}e(H')< \frac{\varepsilon_1}{2}e(H').
\end{equation*}
Together with the fact that $F_3$ holds, this implies $e(H[C])<\varepsilon_1 e(H)$, which by (\ref{equation: bounds on e(H)}) and our choice of $\varepsilon_1$ is at most $\varepsilon_1 k^{\binom{N}{2}} \binom{n}{N} =\varepsilon \binom{n}{N}$. In particular, by construction of~$H$, we have that for each $t=t(C)\in \mathcal{T}$ there are at most~$\varepsilon \binom{n}{N}$ pairs~$(\phi, c)$ with $\phi \in \binom{K_n}{K_m}$, $c\in \mathcal{F}$ and $c\in  \langle t_{\vert \phi}\rangle$. This establishes (ii).

Finally by conclusion 3.\ of Theorem~\ref{theorem: saxton thomason simple containers} and our bound (\ref{eq: linear hypergraph average degree}) on the average degree~$d$ in $H''$, we have
\begin{equation*}
\vert \mathcal{T}\vert \leq \vert \mathcal{C}\vert \leq 2^{\beta(d) k\binom{n}{2}}=k^{O\left( n^{-1/(2r-1)}\right)\binom{n}{2}},
\end{equation*}
so that there exist constants $C_0$,~$n_3>0$ such that for all $n\geq n_3$~sufficiently large, 
\[\log_k\vert \mathcal{T}\vert \leq \frac{C_0}{n^{1/\bigl(2\binom{N}{2}-1\bigr)}} \binom{n}{2}\]
 and (iii) is satisfied. This establishes the statement of Theorem~\ref{theorem: multi-colour container} for $n\geq n_0=\max(n_1, n_2, n_3)$.
\end{proof}


\subsection{Extremal entropy and supersaturation}\label{subsection: entropy density, supersaturation}
In this section we derive the two ingredients needed in virtually all applications of containers, namely the existence of a limiting `entropy density' and a supersaturation result.

\begin{definition}
Let $\mathcal{P}$ be an order-hereditary property of $k$-colourings with $\mathcal{P}_n \neq \emptyset$ for every $n\in \N$. For every $n\in \N$, we define the \emph{extremal entropy} of $\mathcal{P}$ to be 
\[\ex(n, \mathcal{P})=\max\left\{\Ent(t)\, : \,  t \textrm{ is a $k$-colouring template for }K_n \textrm{ with } \langle t \rangle\subseteq \mathcal{P}_n\right\}.\]
\end{definition}
Note that this definition generalises the concept of the Tur\'an number: if $k = 2$, $F$ is a graph and $\mathcal{P} = \Forb(F)$, then $\ex(n, \mathcal{P}) = \ex(n, F)$.
\begin{proposition}\label{proposition: entropy density}
If\/ $\mathcal{P}$ is an order-hereditary property of $k$-colourings with\/ $\mathcal{P}_n \neq \emptyset$ for every $n\in \N$, then the sequence $\bigl(\ex(n, \mathcal{P})/\binom{n}{2}\bigr)_{n\in \N}$ is nonincreasing and tends to a limit $\pi(\mathcal{P})\in[0,1]$ as $n\rightarrow \infty$.
\end{proposition}
\begin{proof}
This is similar to the classical proof of the existence of the Tur\'an density. As observed after Definition~\ref{def: entropy}, $0\leq \Ent(t)\leq \binom{n}{2}$ for any $k$-colouring template $t$ of~$K_n$, so that $\ex(n, \mathcal{P})/\binom{n}{2}\in[0,1]$. It is therefore enough to show that $\left(\ex(n, \mathcal{P})/\binom{n}{2}\right)_{n\in \N}$ is nonincreasing. Let $t$ be any $k$-colouring template for $K_{n+1}$ with $\langle t \rangle\subseteq \mathcal{P}_{n+1}$. For any $\phi \in \binom{K_{n+1}}{K_n}$,  consider $t_{\vert \phi}$. Since $\mathcal{P}$ is order-hereditary, $\langle t \rangle\subseteq \mathcal{P}_{n+1}$ implies $\langle t_{\vert \phi} \rangle \subseteq \mathcal{P}_n$. By averaging over all choices of $\phi$, we have:
\begin{align*}
\frac{\Ent(t)}{\binom{n+1}{2}} 
=\frac{1}{\binom{n+1}{2}} \log_k \left(\prod_{e\in [n+1]^{(2)}} \bigl\lvert t(e)\bigr\rvert\right)
&= \frac{1}{\binom{n+1}{2}}\log_k\left(\prod_{\phi \in \binom{K_{n+1}}{K_n}} \prod_{e\in [n]^{(2)}} \bigl\lvert t_{\vert \phi}(e)\bigr\rvert \right)^{1/n-1}\\
& =\frac{1}{n+1} \frac{1}{\binom{n}{2}}\sum_{\phi \in \binom{K_{n+1}}{K_n}} \Ent(t_{\vert \phi})\\
&\leq \frac{1}{n+1} \frac{1}{\binom{n}{2}}(n+1) \ex(n, \mathcal{P}).
\end{align*}
Thus $\ex(n+1, \mathcal{P})/\binom{n+1}{2}\leq \ex(n, \mathcal{P})/\binom{n}{2}$ as required and we are done.
\end{proof}
We call the limit $\pi(\mathcal{P})$ the \emph{entropy density} of $\mathcal{P}$. 
 Observe that the entropy density gives a lower bound on the \emph{speed} $\vert\mathcal{P}_n\vert$ of the property $\mathcal{P}$: for all $n\in\N$,
\begin{equation}\label{eq: entropy bound on P_n}
k^{\pi(\mathcal{P})\binom{n}{2}}\leq k^{\ex(n, \mathcal{P})}\leq \vert \mathcal{P}_n\vert.
\end{equation}
We shall show (Theorem~\ref{theorem: counting result for Forb(F), F finite hereditary families}) that the exponent in this lower bound is asymptotically tight.


\begin{lemma}[Supersaturation]\label{lemma: supersaturation}
Let $N \in \N$ be fixed and let $\mathcal{F}$ be a nonempty collection of $k$-colourings of~$K_N$. Set\/ $\mathcal{P}=\Forb(\mathcal{F})$.  For every $\varepsilon$ with $0<\varepsilon<1$, there exist constants $n_0\in\N$ and~$C_{0}>0$ such that for all $n\geq n_0$ and every template $t\in \left(2^{[k]}\setminus \{\emptyset \}\right)^{K_n}$ with
\[\Ent(t)> \left(\pi(\mathcal{P})+\varepsilon\right)\binom{n}{2},\]
 there are at least~
 \[C_{0} \varepsilon\binom{n}{N}\] 
 pairs~$(\phi, c)$ with $\phi \in \binom{K_{n}}{K_N}$, $c \in \mathcal{F}$ and  $c\in \langle t_{\vert \phi} \rangle$.
\end{lemma}
\begin{proof}
Given a template   $t'\in \left(2^{[k]}\setminus \{\emptyset\}\right)^{K_m}$, for some $m\geq N$, let  $B(t')$ denote the collection of pairs~$(\phi, c)$ with $\phi \in \binom{K_{m}}{K_N}$, $c\in \mathcal{F}$ and  $c\in \langle {t'}_{\vert \phi} \rangle$.

By Proposition~\ref{proposition: entropy density}, there exists $n_0\geq N$ such that for all $t'\in\left(2^{[k]}\setminus \{\emptyset\}\right)^{K_{n_0}}$ with $\Ent(t')>\left(\pi(\mathcal{P}) +\frac{\varepsilon}{2}\right)\binom{n_0}{2}$, we must have $\vert B(t')\vert\geq 1$. Let $t\in \left(2^{[k]}\setminus\{\emptyset\}\right)^{K_n}$ for some $n\geq n_0$, and suppose $\Ent(t)>\left(\pi(\mathcal{P})+\varepsilon\right)\binom{n}{2}$.  Let $X$ denote the number of sets $\phi\in \binom{K_n}{K_{n_0}}$ such that $\Ent(t_{\vert \phi})> \left(\pi(\mathcal{P}) +\frac{\varepsilon}{2}\right)\binom{n_0}{2}$. By summing $\Ent(t_{\vert \phi})$ over all $\phi\in \binom{K_n}{K_{n_0}}$, we have
\begin{align*}
\left(\pi(\mathcal{P})+\varepsilon\right)\binom{n}{2} \binom{n-2}{n_0-2}<\Ent(t)\binom{n-2}{n_0-2}&=\sum_{\phi} \Ent(t_{\vert \phi})\leq
\binom{n}{n_0}\left(\pi(\mathcal{P})+\frac{\varepsilon}{2}\right)\binom{n_0}{2} +X \binom{n_0}{2},
\end{align*}
implying  $X> \frac{\varepsilon}{2}\binom{n}{n_0}$.
On the other hand, summing $\vert B(t_{\vert \phi})\vert $ over all $\phi\in \binom{K_n}{K_{n_0}}$, we have
\begin{align*}
\vert B(t)\vert \binom{n-N}{n_0-N} = \sum_{\phi}\vert B(t_{\vert \phi})\vert \geq X>\frac{\varepsilon}{2}\binom{n}{n_0},
\end{align*}
so that 
\begin{align*}
\vert B(t)\vert > \frac{1}{2\binom{n_0}{N}} \varepsilon\binom{n}{N}.
\end{align*}
This proves the lemma with $C_0= \bigl(2\binom{n_0}{N}\bigr)^{-1}$. 
\end{proof}
\subsection{Speed of order-hereditary properties}\label{subsection: approximation of arbitrary hereditary properties}
In this section, we relate the speed of an order-hereditary property to its extremal entropy density and obtain container and counting theorems for arbitrary order-hereditary properties (i.e., properties defined by a possibly \emph{infinite} set of forbidden colourings).
\begin{theorem}\label{theorem: counting result for Forb(F), F finite hereditary families}
Let $N \in \N$ be fixed and let $\mathcal{F}$ be a nonempty collection of $k$-colourings of~$E(K_N)$. Set\/ $\mathcal{P}=\Forb(\mathcal{F})$.  For all $\varepsilon>0$, there exists $n_0$ such that for all $n\geq n_0$ we have
\[k^{\pi(\mathcal{P})\binom{n}{2}}\leq \vert \mathcal{P}_n \vert \leq k^{\left(\pi(\mathcal{P})+\varepsilon\right) \binom{n}{2}}.\]
\end{theorem}
\begin{proof}
Inequality~(\ref{eq: entropy bound on P_n}) already established the lower bound on the speed $n \mapsto \vert \mathcal{P}_n\vert$. 
For the upper bound, we apply multicolour containers. By Theorem~\ref{theorem: multi-colour container} for any $\eta>0$ there exists $n_1\in \N$ such that for all $n\geq n_1$, there exists a collection of templates ~$\mathcal{T}_n$ from $(2^{[k]}\setminus \emptyset )^{K_n}$ such that (i) $\mathcal{T}_n$ is a container family for $\mathcal{P}_n$, (ii) for each $t\in \mathcal{T}_n$ there are at most $\eta \binom{n}{N}$ pairs $(\phi,c)$ with $\phi \in \binom{K_n}{K_N}$, $c\in \mathcal{F}$ and $c\in \langle t_{\vert \phi}\rangle$, and (iii) $\vert \mathcal{T}_n\vert \leq k^{\eta\binom{n}{2}}$

Provided we pick $\eta>0$ sufficiently small, Lemma~\ref{lemma: supersaturation} implies there exists $n_2\in \N$ such that for all $n\geq n_2$, if $t\in \left(2^{[k]}\setminus \{\emptyset \}\right)^{K_n}$ and there are at most $\eta \binom{n}{N}$ pairs $(\phi,c)$ with $\phi \in \binom{K_n}{K_N}$, $c\in \mathcal{F}$ and $c\in \langle t_{\vert \phi}\rangle$, then $\Ent(t)\leq \left(\pi(\mathcal{P})+\frac{\varepsilon}{2}\right)\binom{n}{2}$.

Thus choosing $\eta=\eta(\varepsilon)>0$ sufficiently small (in particular less than $\varepsilon/2$) and $n_0\geq \max(n_1, n_2)$, we have that for $n\geq n_0$ every template $t\in  \mathcal{T}_n$ has entropy at most $\left(\pi(\mathcal{P})+\frac{\varepsilon}{2}\right)\binom{n}{2}$, whence we may at last bound above the number of realisations of templates from $\mathcal{T}_n$, and hence the speed of $\mathcal{P}$: for $n\geq n_0$,
\begin{equation*}
\vert \mathcal{P}_n\vert \leq \vert \mathcal{T}_n\vert k^{\max \{\Ent(t): \ t\in \mathcal{T}_n\}}\leq k^{\eta\binom{n}{2} + \left(\pi(\mathcal{P}) + \frac{\varepsilon}{2} \right)\binom{n}{2}}\leq k^{ \pi(\mathcal{P})\binom{n}{2} + \varepsilon \binom{n}{2}}. \qedhere
\end{equation*}
\end{proof}

\begin{theorem}[Approximation of general order-hereditary properties]\label{theorem: approximation of general hereditary properties}
Let\/ $\mathcal{P}$ be an order-hereditary property of $k$-colourings with\/ $\mathcal{P}_n\neq\emptyset$ for every $n\in \N$. Let $\varepsilon>0$ be fixed. There exist constants $N$ and~$n_0\in\N$ and a family~$\mathcal{F}$ of $k$-colourings of $E(K_N)$ such that for all $n\geq n_0$, we have
\begin{enumerate}[(i)]
\item $\mathcal{P}_n\subseteq \Forb(\mathcal{F})_n$, and 
\item $\vert\Forb(\mathcal{F})_n \vert \leq \lvert \mathcal{P}_n\rvert k^{\varepsilon \binom{n}{2}}$.
\end{enumerate}
\end{theorem}

\begin{proof}
For every $n\in\N$, let $\mathcal{F}_n$ denote the collection of $k$-colourings of $E(K_n)$ which are not in $\mathcal{P}_n$ (and thus, as $\mathcal{P}$ is order-hereditary, do not appear as subcolourings of an elements of $\mathcal{P}_{n'}$ for any $n'\geq n$). Set $Q^n=\Forb(\bigcup_{m\leq n}\mathcal{F}_m)$ to be the order-hereditary property of $k$-colourings which avoids exactly the same $k$-colourings on at most $n$ vertices as $\mathcal{P}$. By construction we have
a chain of inclusions
\[\mathcal{Q}^1\supseteq \mathcal{Q}^2\supseteq \cdots \supseteq \mathcal{Q}^n \supseteq \cdots \supseteq \mathcal{P}.\] 
Consequently, the sequence of entropy densities $\left(\pi(\mathcal{Q}^n)\right)_{n\in \N}$ is nonincreasing and bounded below by $\pi(\mathcal{P})$. We claim that $\lim_{n\rightarrow \infty} \pi(\mathcal{Q}^n)=\pi(\mathcal{P})$. Indeed, suppose this was not the case. Then there exists $\eta>0$ such that $\pi(\mathcal{Q}^n)>\pi(\mathcal{P})+\eta$ for all $n\in\N$. For each $n\in \mathbb{N}$, let $t^n$ be a template such that $\langle t^n\rangle\subseteq (\mathcal{Q}^n)_n$ and $\Ent(t^n)=\ex(n, \mathcal{Q}^n)=\ex(n, \mathcal{P})$. Since the sequence $\left(\ex(m, \mathcal{Q}^n)/\binom{m}{2}\right)_{m\in \N}$ is nonincreasing (by Proposition~\ref{proposition: entropy density}), we have that for every $n \in \mathbb{N}$, 
\[\pi(\mathcal{P})+\eta<\pi(\mathcal{Q}^n)\leq \Ent(t^n)/\binom{n}{2}= \ex(n, \mathcal{P})/\binom{n}{2},\]
contradicting 
Proposition~\ref{proposition: entropy density}. Thus  $\lim_{n\rightarrow \infty} \pi(\mathcal{Q}^n)=\pi(\mathcal{P})$, as claimed. In particular there exists $N \in \N$ for which $\pi(\mathcal{Q}^N)<\pi(\mathcal{P})+\varepsilon/2$.

By (\ref{eq: entropy bound on P_n}), $k^{\pi(\mathcal{F}) \binom{n}{2}}\leq \vert \mathcal{P}_n\vert$. On the other hand, by Theorem~\ref{theorem: counting result for Forb(F), F finite hereditary families} (applied to the property $\mathcal{Q}^N$ with parameter $\varepsilon/2$) there exists $n_0\in \N$ with $n_0\geq N$ such that for all $n\geq n_0$ we have:
\begin{equation*}
\bigl\lvert (\mathcal{Q}^N)_n\bigr\rvert\leq k^{\left(\pi(\mathcal{Q}^N)+\frac{\varepsilon}{2} \right)\binom{n}{2}}< k^{\left(\pi(\mathcal{P})+\varepsilon\right)\binom{n}{2}}\leq \vert \mathcal{P}_n\vert k^{\varepsilon{\binom{n}{2}}}. 
\end{equation*}
Observing that for $n\geq n_0$ we have $\left(\mathcal{Q}^N\right)_n = \Forb(\mathcal{F}_N)_n\supseteq \mathcal{P}_n$ we see that the triple $(N, n_0, \mathcal{F}_N)$ satisfies the conclusion of the theorem. 
\end{proof}

\begin{corollary}[Containers for order-hereditary properties]\label{corollary: containers for arbitrary hereditary properties}
Let\/ $\mathcal{P}$ be an order-hereditary property of $k$-colourings with\/ $\mathcal{P}_n\neq\emptyset$ for every $n\in \N$, and let $\varepsilon>0$,~$m\in\N$ be fixed. There exists $n_0$ such that for any $n\geq n_0$ there exists a collection~$\mathcal{T}_n\subseteq \left(2^{[k]}\setminus \{\emptyset\} \right)^{K_n}$ satisfying:
\begin{enumerate}[(i)]
\item $\mathcal{T}_n$ is a container family for\/ $\mathcal{P}_n$;
\item for each template~$t\in \mathcal{T}_n$, $\Ent(t)\leq \left(\pi(\mathcal{P})+\varepsilon \right)\binom{n}{2}$;
\item for each template~$t\in \mathcal{T}_n$, there are at most $\varepsilon \binom{n}{m}$ pairs $(\phi,c)$ where $\phi \in \binom{K_n}{K_m}$, $c\notin \mathcal{P}_m$ and $c\in  \langle t_{\vert \phi}\rangle $;
\item $\vert \mathcal{T}_n\vert\leq k^{\varepsilon \binom{n}{2}}$.
\end{enumerate}
\end{corollary}
\begin{proof}
Fix $\varepsilon>0$. Let $\mathcal{F}_n$ and $\mathcal{Q}^n$ be defined as in the proof of Theorem~\ref{theorem: approximation of general hereditary properties}. As shown in that proof, there exists some $N$ such that $\mathcal{Q}^N\supseteq \mathcal{P}$ and  $\pi(\mathcal{Q}^N)< \pi(\mathcal{P})+\frac{\varepsilon}{2}$. Without loss of generality we may take $N>m$. Note $(\mathcal{Q}^N)_n =\Forb(\mathcal{F}_N)_n$ for all $n\geq N$.


Pick $\delta>0$ sufficiently small. By Theorem~\ref{theorem: multi-colour container} applied to $\Forb(\mathcal{F}_N)$, there exists $n_1\geq N$ such that, for all $n\geq n_1$, there is a collection of templates $\mathcal{T}_n\subseteq \left(2^{[k]}\setminus \{\emptyset\}\right)^{K_n}$ satisfying (a) $\mathcal{T}_n$ is a container family for $(Q^N)_n$,  (b) for each $t\in \mathcal{T}_n$ there are most $\delta \binom{n}{N}$ pairs $(\phi,c)$ with $\phi \in \binom{K_n}{K_N}$, $c\in \mathcal{F}_N$ and $c\in \langle t_{\vert \phi}\rangle$ and (c) $\vert \mathcal{T}_n\vert \leq k^{\delta \binom{n}{2}}$.

Property (a) implies that $\mathcal{T}_n$ is a container family for $\mathcal{P}_n$, establishing part (i)  of the corollary. Property (c) implies part (iv), provided we pick $\delta<\varepsilon$. For part (iii), let $t\in \mathcal{T}_n$ and consider a pair $(\psi, c)$ with $\psi\in \binom{K_n}{K_m}$, $c\in \mathcal{F}_m$ and $c\in \langle t_{\vert \psi}\rangle $. Since $\mathcal{P}$ is order-hereditary, for every $\phi \in \binom{K_n}{K_N}$ extending $\psi$ (i.e. with $\phi(i)=\psi(i)$ for all $i \in [m]$), there exists $c'\in \mathcal{F}_N$ with $c'\in \langle t_{\vert \phi}\rangle $. As we know by property (c) that there are at most $\delta \binom{n}{N}$ such pairs $(\phi, c')$ for any $t\in \mathcal{T}_n$, we have that 
\[ \binom{n-m}{N-m} \bigl\vert \{ (\psi,c): \ \psi\in \binom{K_n}{K_m}, \ c\in \mathcal{F}_m, \ c\in \langle t_{\vert \psi}\rangle \}\bigr\vert \leq \binom{N}{m}\delta \binom{n}{N},\]
implying part (iii), provided we pick $\delta<\varepsilon$. Finally for part (ii)  we use supersaturation: there exists $n_0\geq n_1$ such that if $\delta$ is sufficiently small, $t\in \left(2^{[k]}\setminus \{\emptyset\}\right)^{K_n}$ with $n\geq n_0$ and there are at most~$\delta\binom{n}{N}$ pairs $(\phi,c)\in \binom{K_n}{K_N}\times \mathcal{F}_N$ with $c\in \langle t_{\vert \phi}\rangle $, then 
\[\Ent(t)\leq \left(\pi(\Forb\bigl(\mathcal{F}_N)\bigr)+\frac{\varepsilon}{2}\right)\binom{n}{2}=\left(\pi(\mathcal{Q}^N)+\frac{\varepsilon}{2}\right)\binom{n}{2}<\left(\pi(\mathcal{P})+\varepsilon\right)\binom{n}{2}.\] 
Thus for $n\geq n_0$ and $\delta>0$ sufficiently small, property (c) implies part (ii) and we are done.
\end{proof}

\begin{corollary}[Speed of order-hereditary properties]\label{corollary: speed of arbitrary hereditary properties}
Let\/ $\mathcal{P}$ be an order-hereditary property of $k$-colourings with\/ $\mathcal{P}_n\neq\emptyset$ for every $n\in \N$ and let $\varepsilon>0$ be fixed. There exists $n_0\in \N$ such that for all $n\geq n_0$,
\[k^{\pi(\mathcal{P})\binom{n}{2}}\leq \vert \mathcal{P}_n\vert \leq k^{\left(\pi(\mathcal{P})+\varepsilon \binom{n}{2}\right)}.\]
\end{corollary}
\begin{proof}
The lower bound is given by inequality~(\ref{eq: entropy bound on P_n}). For the upper bound, we apply Corollary~\ref{corollary: containers for arbitrary hereditary properties} (with parameter~$\varepsilon/2$) to obtain for all $n\geq n_0$ a family of templates~$\mathcal{T}_n$ satisfying properties (i), (ii) and~(iv). Thus for $n\geq n_0$,
\[\vert\mathcal{P}_n\vert \leq \vert\mathcal{T}_n\vert k^{\max\{ \Ent(t):\ t\in \mathcal{T}_n\} } \leq k^{\frac{\varepsilon}{2}\binom{n}{2}} k^{\pi(\mathcal{P})\binom{n}{2}+\frac{\varepsilon}{2}\binom{n}{2}}= k^{\left(\pi(\mathcal{P})+\varepsilon \binom{n}{2}\right)}. \qedhere\]
\end{proof}
\subsection{Stability and characterisation of typical colourings}\label{subsection: stability}

\begin{definition}[Edit distance]\label{definition: edit distance}
	The \emph{edit distance}~$\rho(t,t')$ between two $k$-colouring templates $t$,~$t'$ of~$K_n$ is the number of edges~$e\in E(K_n)$ on which $t(e)\neq t'(e)$. Further, if $\mathcal{S}$ is a family of $k$-colouring templates and $t$ is a $k$-colouring template of $K_n$, we define the \emph{edit distance} between them to be
	\[\rho(\mathcal{S}, t):=\min\{\rho(s,t):\ s\in \mathcal{S}\cap (2^{[k]}\setminus \{\emptyset\})^{K_n} \}\] 
if $\mathcal{S}$ contains at least one element of $(2^{[k]}\setminus \{\emptyset\})^{K_n} $, and $\binom{n}{2}$ otherwise.
	\end{definition}
Given a family of $k$-colouring templates $\mathcal{S}$, we let $\langle \mathcal{S}\rangle$ denote the collection of all realisations from $\mathcal{S}$. We further let $\rho(\langle \mathcal{S}\rangle, c)$ denote the edit distance between $\langle \mathcal{S}\rangle$ and a $k$-colouring $c$.
\begin{definition}[Stability]\label{definition: strong stability family}
Let $\mathcal{P}$ be an order-hereditary property of $k$-colourings of~$K_n$. A family~$\mathcal{S}$ of $k$-colouring templates is said to be a \emph{stability family} for $\mathcal{P}$ if the following holds:\\
for all $\varepsilon>0$, there exist $\delta>0$ and $m$,~$n_0\in \N$ such that for all $n\geq n_0$, every 
$t\in (2^{[k]}\setminus \{\emptyset\})^{K_n}$ with
	\begin{enumerate}[(i)]
		\item (almost extremality) $\Ent(t)\geq (\pi(\mathcal{P})-\delta)\binom{n}{2}$, and
		\item (almost locally  in $\mathcal{P}$) at most~$\delta \binom{n}{m}$ pairs~$(\phi, c)$ with $\phi\in \binom{K_n}{K_m}$, $c\notin \mathcal{P}_m$  and $c\in \langle t_{\vert_\phi}\rangle$ 
	\end{enumerate}
	must lie within edit distance at most $\varepsilon \binom{n}{2}$ of some template $s\in\mathcal{S}$.

If $\mathcal{P}$ has a stability family, then it is said to be \emph{stable}.
\end{definition}

\begin{theorem}[Characterisation of typical colourings in stable order-hereditary properties]\label{theorem: strong stability and containers}
\	\\
Let\/ $\mathcal{P}$ be an order-hereditary property of $k$-colourings and suppose $\mathcal{S}$ is a stability family for\/ $\mathcal{P}$. Then typical elements of\/ $\mathcal{P}$ are close in the edit distance to realisations from\/ $\mathcal{S}$.

Explicitly, for all $\varepsilon>0$, there exists $n_0\in \N$ such that for all $n\geq n_0$ there are at most~$\varepsilon \vert \mathcal{P}_n\vert$ colourings~$c\in\mathcal{P}_n$ with $\rho(\langle \mathcal{S} \rangle, c)>\varepsilon \binom{n}{2}$.
\end{theorem}

\begin{proof}

Let $\varepsilon>0$, and let $\delta, m, n_0$ be as given by Definition~\ref{definition: strong stability family} applied to $\mathcal{S}$ and $\mathcal{P}$ with parameter~$\varepsilon$. Apply Corollary \ref{corollary: containers for arbitrary hereditary properties} to $\mathcal{P}$ with parameters $\varepsilon_1>0, m\in \mathbb{N}$, with $\varepsilon_1<\delta$, to get (for all $n$~sufficiently large) a container family $\mathcal{T}_n$ for $\mathcal{P}_n$.

Now remove from $\mathcal{T}_n$ any template $t$ with $\text{Ent}(t)<(\pi(\mathcal{P})-\delta){\binom{n}{2}}$ and let $\mathcal{T}_n'\subseteq\mathcal{T}_n$ denote the resulting subfamily.  By Corollary~\ref{corollary: containers for arbitrary hereditary properties} parts (i) and (iv), Proposition~\ref{proposition: entropy density} and~\eqref{eq: entropy bound on P_n}, the number of elements of $\mathcal{P}_n$ which are not realisable from a template~$t\in \mathcal{T}_n'$ is then at most
\[\vert \mathcal{T}_n\vert k^{(\pi(\mathcal{P})-\delta){\binom{n}{2}}} \le k^{(\pi(\mathcal{P})+\varepsilon_1-\delta){\binom{n}{2}}} \leq k^{\ex(n, \mathcal{P})+(\varepsilon_1-\delta){\binom{n}{2}}}\leq\vert \mathcal{P}_n\vert k^{(\varepsilon_1-\delta){\binom{n}{2}}}.\]
Since $\varepsilon_1<\delta$, the right hand side is less than~$\varepsilon \vert \mathcal{P}_n\vert$ for all $n$ sufficiently large.

Now let $c$ be a member of $\mathcal{P}_n$ which \textit{is} realisable from a template $t\in\mathcal{T}_n'$. Since $t\in\mathcal{T}_n'$, we have $\text{Ent}(t)\ge (\pi(\mathcal{P})-\delta){\binom{n}{2}}$. Also since $\varepsilon_1<\delta$, Corollary~\ref{corollary: containers for arbitrary hereditary properties} part (iii) implies that $t$ satisfies condition~(ii) from Definition~\ref{definition: strong stability family}, and so there is a template~$s\in \mathcal{S}$ such that $\rho(s,t)<\varepsilon{\binom{n}{2}}$. Since $c$ realises $t$, this readily implies that $\rho(\langle \mathcal{S}\rangle, c)<\varepsilon{\binom{n}{2}}$.
\end{proof}

\subsection{Transference}\label{subsection: transference}
\begin{definition}[Multicolour monotonicity]
An order-hereditary property $\mathcal{P}$ of $k$-colourings is \emph{monotone} with respect to colour $i\in[k]$, or \emph{$i$-monotone}, if whenever $c$ is a $k$-colouring of~$K_n$ which lies in $\mathcal{P}$ and $e$ is any edge of~$K_n$, the colouring~$\tilde{c}$ obtained from $c$ by changing the colour of~$e$ to $i$ also lies in $\mathcal{P}$.
\end{definition}
\begin{definition}[Meet of two templates]
Given two $k$-colouring templates $t$,~$t'$ of~$K_n$ which have at least~one realisation in common, we denote by $t\wedge t'$ the template with $(t\wedge t')(e) =t(e)\cap t'(e)$ for every $e\in E(K_n)$; we refer to $t\wedge t'$ as the \emph{meet} of $t$~and~$t'$.  More generally, given a set~$\mathcal{S}$ of $k$-colouring templates of~$K_n$ and
$t\in (2^{[k]}\setminus\{\emptyset\} )^{K_n}$ we denote by $\mathcal{S}\wedge t$ the collection~$\{s\wedge t: \ s\in\mathcal{S}\}$. 
\end{definition}

\begin{definition}[Complete, random and constant templates]
Let $T_n=[k]^{K_n}$ denote the \emph{complete $k$-colouring template} for $K_n$, that is, the unique template allowing all $k$ colours on every edge. Given a fixed colour $i\in[k]$ and $p\in[0,1]$, we define the \emph{$p$-random template} $T_{n,p}=T_{n,p}(i)$ to be the random template for a $k$-colouring of~$K_n$ obtained by letting
\[T_{n,p}(e)=\left\{\begin{array}{ll}
[k] & \textrm{with probability $p$,}\\
\{i\} & \textrm{otherwise,}\end{array} \right.\]
independently for each edge~$e\in E(K_n)$. Finally, we define $E_n=E_n(i)$ to be the \emph{$i$-monotone} template with $E_n(e)=\{i\}$ for each $e\in E(K_n)$.
\end{definition}
The $p$-random template~$T_{n,p}$ is a $k$-colouring analogue of the celebrated Erd{\H o}s--R\'enyi binomial random graph~$G_{n,p}$, while the zero-entropy template $E_n$ is a $k$-colouring analogue of the empty graph. Just as extremal theorems for the graph~$K_n$ can be reproved in the sparse random setting of~$G_{n,p}$, so also extremal entropy results for $i$-monotone properties in $T_n$ can be transferred to the $T_{n,p}(i)$ setting.


\begin{definition}[Relative entropy]
Let $\mathcal{P}$ be a property of $k$-colourings of complete graphs 
and let $t \in (2^{[k]}\setminus \{\emptyset\} )^{K_n}$ with $\langle t\rangle \cap \mathcal{P}_n \neq \emptyset$.
We define the \emph{extremal entropy of $\mathcal{P}$ relative to $t$} to be:
\[\ex(t,\mathcal{P}):=\max\bigl\{\Ent(t') \, : \, v(t')=n,\ t'\leq t, \ \langle t' \rangle \subseteq \mathcal{P}_n\bigr\}.\]
\end{definition}
Note that this extends the notion of extremal entropy introduced in Definition~\ref{def: entropy}: $\ex(n, \mathcal{P})=\ex(T_n,\mathcal{P})$. Our next theorem states that for $p$ not too small, with high probability the extremal entropy of an $i$-monotone property $\mathcal{P}$ relative to a $p$-random template $T_{n,p}(i)$ is $p\left(\ex(n, \mathcal{P})+o(n^2)\right)$. 
\begin{theorem}[Transference]\label{theorem: transference}
Let $i\in [k]$ be fixed. Let\/ $\mathcal{P}$ be an order-hereditary, $i$-monotone property of $k$-colourings of complete graphs defined by forbidden colourings of $E(K_N)$ for some $N\geq 2$. Let $p=p(n)\gg n^{-1/(2\binom{N}{2}-1)}$ and let $T$ denote an instance of the $p$-random template $T_{n,p}(i)$. For any fixed $\varepsilon>0$, with high probability
\begin{equation}\label{eq:TransferenceEntropy}
p\left(\ex(n,\mathcal{P})-\varepsilon n^2 \right) \leq \ex(T,\mathcal{P})\leq p\left(\ex(n,\mathcal{P})+2\varepsilon n^2 \right).
\end{equation}
\end{theorem}
\begin{proof}
Let $\varepsilon>0$ be fixed, and let $C_0^{\ref{theorem: multi-colour container}}$ and $C_0^{\ref{lemma: supersaturation}}$ denote the constants in Theorem~\ref{theorem: multi-colour container} and Lemma~\ref{lemma: supersaturation} respectively.

Applying Theorem~\ref{theorem: multi-colour container} with parameter $\varepsilon_1= C_0^{\ref{lemma: supersaturation}} \varepsilon$ followed by Lemma~\ref{lemma: supersaturation} gives for all $n\geq n_0=n_0(\varepsilon)$ sufficiently large a container family~$\mathcal{T}_n$ for $\mathcal{P}_n$ such that $\log_k \vert \mathcal{T}_n \vert \leq 
C_0^{\ref{theorem: multi-colour container}}n^{-1/(2\binom{N}{2}-1)} \binom{n}{2}$ and every template $t\in \mathcal{T}_n$  has entropy at most $\ex(n, \mathcal{P})+\varepsilon \binom{n}{2}$ (property (ii) of Theorem~\ref{theorem: multi-colour container} together with Lemma~\ref{lemma: supersaturation}).

Now let $T$ be an instance of the random $k$-colouring template $T_{n,p}$. As $\mathcal{T}_n$ is a container family for $\mathcal{P}_n$,  $\mathcal{T}_n\wedge T$ is a container family for $\mathcal{P}_n\wedge T$. In particular,
\begin{align}\label{equation: bound on the extremal entropy in the sparse random setting}
\ex(T, \mathcal{P})\leq \max\{\mathrm{Ent}(t\wedge T): \ t\in \mathcal{T}_n\}.\end{align}
We claim that whp the right-hand side is at most $p\left(\ex(n,\mathcal{P})+2\varepsilon n^2 \right)$. Indeed for each $t \in \mathcal{T}_n$ we have 
\[\Ent(t\wedge T)=\sum_{e\in E(K_n)} \log_k \bigl\lvert t\wedge T(e)\bigr\rvert=\sum_{e: \ T(e)=[k]} \log_k\bigl\lvert t(e)\bigr\rvert.\]
Partition the edges of $K_n$ into $k$ sets, $A_1$, $A_2$, \ldots $A_k$, where $A_i=\{e:\ \vert t(e)\vert = i\}$. 
Set $A_{i, T}:= A_i\cap \{e:\ T(e)=[k]\}$. By the equation above, each edge in $A_{i,T}$ contributes $\log_k i \leq 1$ to the entropy of $t\wedge T$. A simple union bound together with an application of the Chernoff bound~(\ref{equation: Chernoff}) then yields:
\begin{equation}
\mathbb{P}\left(\Ent(t\wedge T)> p \Ent(t) + p\varepsilon n^2\right) \leq \mathbb{P}\left(\exists i: \ \vert A_{i,T}\vert > p \vert A_i \vert +\frac{2\varepsilon}{k}p\frac{n^2}{2}  \right) \leq 2ke^{-\frac{\varepsilon^2 n^2 p}{2k^2}}.\label{equation: probability of a large deviation in the sparsified entropy of t}
\end{equation}
Bringing together (\ref{equation: bound on the extremal entropy in the sparse random setting}), our bound  on the entropy of templates from $\mathcal{T}_n$ and the inequality (\ref{equation: probability of a large deviation in the sparsified entropy of t}), we get
\[
\mathbb{P}\left(\ex(T, \mathcal{P})> p(\ex(n, \mathcal{P})+2\varepsilon n^2)\right)\leq \mathbb{P}\left(\exists t\in \mathcal{T}_n: \ \Ent(t\wedge T)> p\Ent(t)+ p \varepsilon n^2\right)\leq \vert \mathcal{T}_n\vert 2k e^{-\frac{\varepsilon^2 n^2 p}{2k^2}}.
\]
Since $\log_k \vert \mathcal{T}_n\vert =O\Bigl(n^{2-\frac{1}{2\bigl(\binom{N}{2}-1\bigr) }}\Bigr)$, the expression above is of order $o(1)$ if $p \gg n^{-1/\bigl(2\binom{N}{2}-1\bigr)}$. Thus for such values of $p$, whp $\ex(T, \mathcal{P})\leq p(\ex(n, \mathcal{P})  +2\varepsilon n^2)$. This establishes the upper bound in~\eqref{eq:TransferenceEntropy}.

For the lower bound, consider a maximum entropy template $t_ {\star}$ for $\mathcal{P}_n$. 
Applying the Chernoff bound~\eqref{equation: Chernoff}, we have
\[\mathbb{P}\Bigl(\Ent(t_{\star}\wedge T)<p\bigl(\ex(n,\mathcal{P})-\varepsilon n^2\bigr)\Bigr) \leq 2 e^{-\frac{\varepsilon^2 n^2 p}{k^2}},\]
which is $o(1)$ for $\varepsilon >0$ fixed and $p\gg n^{-2}$. Thus certainly for $p\gg n^{-\frac{1}{2\bigl(\binom{N}{2}-1\bigr)}}$ we have whp 
\[\ex(T, \mathcal{P})\geq \Ent(t_{\star}\wedge T)\geq p (\ex(n, \mathcal{P}) -\varepsilon n^2). \qedhere\]
\end{proof}

\begin{remark}
	The bound on $p$ required in Theorem~\ref{theorem: transference} is not best possible in general (see~\cite{Conlon14}). 
	This bound can be improved by using the more powerful container theorems of~\cite{BaloghMorrisSamotij15} and~\cite{SaxtonThomason15} rather than the simple hypergraph container theorem, Theorem~\ref{theorem: saxton thomason simple containers}. However
	we do not pursue such improvements further here.
\end{remark}
We note Theorem~\ref{theorem: transference} can be extended to cover general order-hereditary properties.
\begin{corollary}[Transference for general order-hereditary properties]\label{corollary: transference for general properties}
Let\/ $\mathcal{P}$ be an order-hereditary, $i$-monotone property of $k$-colourings of complete graphs. Let $p=p(n)$ be a sequence of probabilities satisfying $\log (1/p)=o(\log n)$. 
For any fixed $\varepsilon>0$, with high probability
\[p\left(\ex(n,\mathcal{P})-\varepsilon n^2\right)\leq \ex(T_{n,p}(i),\mathcal{P})\leq p\left(\ex(n,\mathcal{P})+4\varepsilon n^2\right).\]
\end{corollary}
\begin{proof}
Let $\varepsilon>0$ be fixed. As in Theorem~\ref{theorem: approximation of general hereditary properties}, approximate $\mathcal{P}$ from above by some property~$\mathcal{Q}$ defined by forbidden colourings on at most $N$ vertices and satisfying $\mathcal{P}\subseteq\mathcal{Q}$ and $\pi(\mathcal{Q}) \leq \pi(\mathcal{P}) +\varepsilon$. For $n$~sufficiently large, 
\[\ex(n,\mathcal{Q})\leq \pi(\mathcal{Q})\binom{n}{2}+\varepsilon \binom{n}{2}\leq \pi(\mathcal{P})\binom{n}{2}+\varepsilon n^2\leq ex(n, \mathcal{P})+2\varepsilon n^2.\] Applying Theorem~\ref{theorem: transference} to $\mathcal{Q}$, and noting that our condition on $p$ ensures $p\gg n^{-\bigl(1/\bigl(2\binom{N}{2}-1\bigr)\bigr)}$ for all $N \in \N$, we obtain the desired upper bound: whp
\[\ex(T_{n,p}(i), \mathcal{P}) \leq \ex(T_{n,p}(i), \mathcal{Q})\leq p\left(\ex(n, \mathcal{Q}) +2\varepsilon n^2 \right)\leq p \left(\ex(n, \mathcal{P}) +4\varepsilon n^2\right).\] 
For the lower bound, let $t_{\star}$ be a maximum entropy template for $\mathcal{P}_n$. 
Applying the Chernoff bound as in the  proof of Theorem~\ref{theorem: transference}, we have that whp \[\ex(T_{n,p}(i), \mathcal{P})\geq \Ent(t_{\star}\wedge T_{n,p}(i))\geq p\left(\ex(n, \mathcal{P})-\varepsilon n^2\right). \qedhere\]
\end{proof}

\section{Containers for other discrete structures}\label{section: other structures}
Our container results so far allow us to compute the speed of (dense) order-hereditary properties of $k$-colourings of~$K_n$, as well as to characterise typical colourings and (in the $i$-monotone case) to transfer extremal entropy results to the sparse random setting. However, the container theory of Saxton--Thomason and Balogh--Morris--Samotij is robust enough to cover $k$-colourings of many other interesting discrete structures, which is what we explore in this section.

As we show,  all that is required for (the existence of) a container theorem is, in essence, a sufficiently rich notion of substructure: provided we have a sequence of $r$-graphs $(H_n)_{n\in\mathbb{N}}$  such that $e(H_n)\rightarrow \infty$ as $n\rightarrow \infty$ and for each $N\in \mathbb{N}$ we have `many' embeddings of $H_N$ into $H_n$, we can derive a container theorem. In this regard, container theory is somewhat reminiscent of the versatile theory of flag algebras developed by Razborov~\cite{Razborov07}, which can treat any class of discrete structures with a sufficiently rich notion of substructure.

This section is organised as follows: first in Section~\ref{subsection: oriented}, we outline how our $k$-colouring extensions of the container theorems of Balogh--Morris--Samotij and Saxton--Thomason can be applied to tournaments, oriented graphs and directed graphs; next in Section~\ref{subsection: possee} we derive container theorems for very general discrete structure, namely set-sequences equipped with embeddings (or \emph{ssee}, see Definition~\ref{definition: possee}); in Section~\ref{subsection: hypergraph sequences}, we record the consequences of our results for sequences of graphs and hypergraphs, which are the special cases most relevant in applications; in Section~\ref{subsection: hypercubes}, we illustrate our results by obtaining general counting theorems for $k$-colouring properties of hypercube graphs.

\subsection{Tournaments, oriented graphs and directed graphs}\label{subsection: oriented}
Tournaments, oriented graphs and directed graphs are important objects of study in discrete mathematics and computer science, with a number of applications both to other branches of mathematics and to real-world problems. As we show below, we can encode these structures within our framework of order-hereditary properties for $k$-colourings of~$K_n$, which immediately gives container, supersaturation, counting, characterisation and transference theorems for these objects. We note containers had not been successfully applied to the directed graph setting before (see Section~\ref{subsection: digraph examples} for a discussion, or the remark after Corollary~3.4 in K\"uhn, Osthus, Townsend and Zhao~\cite{KuhnOsthusTownsendZhao14}).

Formally, a \emph{directed graph}, or \emph{digraph}, is a pair $D=(V,E)$, where $V=V(D)$ is a set of vertices and $E=E(D)\subseteq V\times V$ is a collection of \emph{ordered} pairs from $V$. By convention, we write $\vec{ij}$ to denote $(i,j)\in E$. Note that we could have both $\vec{ij}\in E(D)$ and $\vec{ji}\in E(D)$, in which case we say that $ij$ is a \emph{double edge} of $D$.

An \emph{oriented graph}, or \emph{orgraph}, is a digraph $\vec{G}$ in which for each pair $ij\in V(\vec{G})^{(2)}$ at most one of $\vec{ij}$ and $\vec{ji}$ lies in $E(\vec{G})$. A \emph{tournament} $\vec{T}$ is a digraph in which for each pair $ij\in V(\vec{T})^{(2)}$ exactly one $\vec{ij}$ and $\vec{ji}$ lies in $E(\vec{T})$; alternatively, a tournament can be viewed as an orientation of the edges of the complete graph.

A \emph{monotone (decreasing)} property of digraphs/orgraphs is a property of digraphs/orgraphs which is closed with respect to taking subgraphs (i.e.\ closed under the deletion of vertices and oriented edges). A \emph{hereditary} property of digraphs/orgraphs/tournaments is a property of digraphs/orgraphs/tournaments which is closed with respect to taking induced subgraphs.

\begin{observation}\label{observation: encoding of digraphs}
Tournaments, oriented graphs and directed graphs on the labelled vertex set $[n]$  can be encoded as $2$--, $3$-- and $4$-colourings of~$K_n$. Moreover, under this encoding, hereditary properties of tournaments, oriented graphs and directed graphs correspond to order-hereditary properties of $2$--, $3$-- and $4$-colourings of~$K_n$ respectively.
\end{observation}	
\begin{proof}
	Given a directed graph $D$ on $[n]$, we define a colouring $c$ of $E(K_n)$ by setting for each pair~$ij\in [n]^{(2)}$ with $i<j$
	\[c(ij):=\left\{\begin{array}{ll}1 & \textrm{if neither of $\vec{ij}$,~$\vec{ji}$ lies in $E(D)$,}\\
	2 & \textrm{if $\vec{ij}\in E(D)$, $\vec{ji}\notin E(D)$,}\\
	3 & \textrm{if $\vec{ij}\notin E(D)$, $\vec{ji}\in E(D)$,}\\
	4 & \textrm{if both of $\vec{ij}$,~$\vec{ji}$ lie in $E(D)$.}
	\end{array}
	 \right.\]
	The digraph part of Observation~\ref{observation: encoding of digraphs} is immediate from this colouring and our definition of order-hereditary properties. Tournaments for their part correspond to colourings of $E(K_n)$ with the palette ~$\{2,3\}$ and oriented graphs to colourings with the palette~$\{1,2,3\}$. 
\end{proof}	
\begin{remark}
Monotone properties of digraphs/orgraphs give rise to $1$-monotone order-hereditary properties of $4$--/$3$-colourings of~$K_n$, and so are covered by our transference results. However there are some theoretical subtleties to bear in mind in the digraph case: a monotone digraph property has  monotonicity `away from colour $4$' as well as monotonicity `towards colour $1$'. Thus there could be interesting and natural alternative models to the $p$-random template $T=T_{n,p}$ to study in connection with transference for digraph properties.  Instead of letting $T(e)=[4]$ with probability~$p$ and $\{1\}$ with probability~$1-p$ for each edge $e$, one could consider more general distributions on the collection  of subsets of $[4]$ containing the colour~$1$.

 In addition, we should make it clear that there are examples of $1$-monotone properties of $4$-colourings of $K_n$ which do \emph{not} correspond to to monotone properties of digraphs. A nice example of such a property suggested  by one of the referees is the property of having every pair in colour either $1$ or $4$. This is $1$-monotone, but does not give rise to a monotone property of digraphs in our encoding.	
\end{remark}
\begin{corollary}
If\/ $\mathcal{P}$ is a hereditary property of digraphs/orgraphs/tournaments defined by forbidden configurations on at most~$N$ vertices and  $k=4/3/2$ is the corresponding number of colours from Observation~\ref{observation: encoding of digraphs}, then the conclusions of Theorem~\ref{theorem: multi-colour container}, Lemma~\ref{lemma: supersaturation}, Theorems~\ref{theorem: counting result for Forb(F), F finite hereditary families}, \ref{theorem: strong stability and containers}, and~\ref{theorem: transference} hold for\/ $\mathcal{P}$.\qed
\end{corollary}
\begin{corollary}
If\/ $\mathcal{P}$ is a hereditary property of digraphs/orgraphs/tournaments and $k=4/3/2$ is the corresponding number of colours from Observation~\ref{observation: encoding of digraphs}, then the conclusions of Corollaries~\ref{corollary: containers for arbitrary hereditary properties} and~\ref{corollary: speed of arbitrary hereditary properties}, Theorem~\ref{theorem: strong stability and containers} and Corollary~\ref{corollary: transference for general properties} hold for\/ $\mathcal{P}$.\qed
\end{corollary}

In particular, we have general container, counting, stability and transference results for hereditary properties of digraphs, orgraphs and tournaments. As mentioned earlier, this overcomes an obstruction to the extension of containers to the digraph setting.

\subsection{Other host structures: set-sequences equipped with embeddings (ssee-s)}\label{subsection: possee}
The work in Section~\ref{section: containers} was concerned with $k$-colouring properties of the sequence of complete graphs~$\Complete=(K_n)_{n\in \N}$. However, there are many other interesting natural graph sequences we might wish to study. Examples of such sequences include:
\begin{itemize}
	\item $\Path=(P_n)_{n\in \N}$, the sequence of paths on $[n]$, $P_n=([n], \{i(i+1): \ 1\leq i \leq n-1\})$;
	\item $\Grid= (P_n\times P_n)_{n \in \N}$, the sequence of $n\times n$ grids~$P_n\times P_n$ obtained by taking the Cartesian product of $P_n$ with itself, or, more generally for $(a,b)\in \N^2$ the sequence of rectangular grids~$\Grid(a,b)=(P_{an}\times P_{bn})_{n\in \N}$ with vertex-set $\{(x,y): \ 1\leq x \leq an,  \ 1\leq y \leq bn\}$ and edge-set $\{(x,y)(x',y'): \ \vert x-x'\vert + \vert y-y'\vert =1\}$;
	\item $\Branch_b=(B_{b,n})_{n \in \N}$, the sequence of $b$-branching trees with $n$ generations from a single root;
	\item $\Complete_q=(K_q(n))_{n \in \N}$, the sequence of complete balanced $q$-partite graphs on $qn$ vertices;
	\item $\Hypercube=(Q_n)_{n\in \N}$, the sequence of $n$-dimensional discrete hypercube graphs~$Q_n=(\{0,1\}^n, \{\vector{x}\vector{y}: \ \vector{x}_i =\vector{y}_i \textrm{ for all but exactly~one index }i\}$.
\end{itemize}
Outside of extremal combinatorics, the sequences $\Hypercube$ and~$\Branch_2$ are of central importance in theoretical computer science and discrete probability (they represent $n$-bit sequences and binary search trees respectively), while the sequence~$\Grid$ has been extensively studied in the context of percolation theory, in particular with respect to crossing probabilities.

Each of the graph sequences above comes equipped with a natural notion of `substructure' --- subpaths of a path, subgrids of a grid, subtrees of a branching tree, $q$-partite subgraphs of a $q$-partite graph, subcubes of a hypercube --- and of `embeddings' of earlier terms of the sequence into later ones, which leads to a natural notion of an (order-) hereditary property.

As we shall see in this subsection, the container theory of Balogh--Morris--Samotij and Saxton--Thomason is powerful enough to cover the case of $k$-colourings of any graph sequence~$\Graphseq$ with a `sufficiently rich' notion of substructure. More generally, we shall derive container theorems for $k$-colouring properties of some very general structures (\emph{good ssee}, defined below)  which cover $k$-colourings of vertices and $k$-colouring of edges of `good' hypergraph sequences (and many other structures besides) as special cases.  Roughly speaking, a hypergraph sequence $\Graphseq$ is (edge-) good if it is rich in embeddings --- for any $N$ fixed and $n\geq N$ there must be many almost disjoint ways of embedding $G_N$ into $G_n$ relative to the number of edges.

Our main results in this subsection are a container theorem for hereditary $k$-colourings of `good' set-sequences (Theorem~\ref{theorem: ossee container}), and, modulo some easily checkable technical conditions, the accompanying counting results (Theorems~\ref{theorem: counting for possee} and \ref{theorem: counting for arbitrary possee}). Cases of interest covered by our result include $k$-colourings of $\Complete_q$, $\Grid$ and both vertex- and edge-$k$-colourings of $\Hypercube$.  A final observation to make before we give our definitions and results is that, as we shall show in Section~\ref{subsection: nonexample containers fail for paths}, some form of our `goodness' assumption is necessary --- the sequence $\Path$, for instance, has too few embeddings to be `good', and  we give an example of a hereditary $k$-colouring property for $\Path$ for which the statement of Theorem~\ref{theorem: ossee container} fails.

\begin{definition}[Ssee, embeddings]\label{definition: possee}
	A \emph{set-sequence equipped with embeddings}, or \emph{ssee}, is a sequence~$\mathbf{V}=(V_n)_{n\in \mathbb{N}}$ of sets $V_n$, together with for every $N\leq n$ a collection $\binom{V_n}{V_N}$ of injections~$\phi: \ V_N \to V_n$.

	We refer to the members of $\binom{V_n}{V_N}$ as \emph{embeddings} of $V_N$ into $V_n$.
\end{definition}
Ssee may seem rather abstract, so let us immediately give some examples.
\begin{example}\label{example: poset ssee}
Let $\mathbf{V}$ denote a sequence $(V_n)_{n\in \mathbb{N}}$ of partially ordered sets, with embeddings $\binom{V_n}{V_N}$ consisting of all order-preserving injections from $V_N$ to $V_n$. Then $\mathbf{V}$ is an ssee.	
\end{example}
\begin{example}\label{example: graph vertex ssee}
Consider a sequence of graphs $\Graphseq=(G_n)_{n\in \mathbb{N}}$ on linearly-ordered vertex sets. We can obtain an ssee from $\Graphseq$ by taking as our set-sequences the vertex-sets $V_n=V(G_n)$ and setting $\binom{V_n}{V_N}$ to be the collection of all order-preserving injections $\phi: \ V_N \to V_n$ such that $\phi(x)\phi(y) \in E(G_n)$ if and only if $xy \in E(G_N)$ --- in other words, the collection of all order-preserving embeddings of $G_N$ into $G_n$.
\end{example}
 \begin{example}\label{example: graph edge ssee}
Consider again a sequence of graphs $\Graphseq=(G_n)_{n\in \mathbb{N}}$ on linearly-ordered vertex-sets. 
We can obtain another ssee from $\Graphseq$ by taking the sets of our sequence to be the edge-sets $E(G_n)$ and setting $\binom{E(G_n)}{E(G_N)}$ to be the collection of injections $\psi: \ E_N\to E_n$ arising from order-preserving embeddings $\phi: V(G_N)\to V(G_n)$ (i.e. such that $\psi(e)=\{\phi(x): \ x \in e\}$).
 \end{example}
\begin{example}\label{example: permutation ssee}
	Consider a sequence of permutations $\sigma_n \in S_n$. Let $V_n=[n]$, and let $\binom{V_n}{V_N}$ denote the collection of order-preserving injections $\phi: \ [N]\to [n]$ such that $\sigma_n (\phi(i))< \sigma_n (\phi(j))$ whenever $\sigma_N (i)< \sigma_N (j)$. This constitutes an ssee.
\end{example}
\begin{example}\label{example: group ssee}
	Consider a sequence of groups  $((\Gamma_n, +_n))_{n\in \mathbb{N}}$. For every $n$, let $V_n$ be a nonempty subset of $\Gamma_n$, and let $\binom{V_n}{V_N}$ denote the collection of 
	injections $\phi: \ V_N \rightarrow V_n$ which preserve the group actions, i.e. such that $\phi(x+_Ny)=\phi(x)+_n\phi(y)$ for all $x,y \in V_N$ with $x+_Ny \in V_N$. This constitutes an ssee.  
\end{example}
A more concrete form of the last example, which has already been extensively studied using containers, is that of subsets of $(\mathbb{Z}, +)$, which are connected to many problems in additive combinatorics --- see the original papers of Balogh, Morris and Samotij~\cite{BaloghMorrisSamotij15} and of Saxton and Thomason~\cite{SaxtonThomason15}.

Having thus set the scene with some motivational examples of ssee-s, we now turn to the main business of this section, namely generalising Theorem~\ref{theorem: multi-colour container} to the ssee setting. To do this, we need notions of colourings, templates and extremal entropy \emph{relative} to a set.
\begin{definition}[Colourings, templates and entropy relative to a set]\label{definition: template/entropy relative to possee}
	Let $V$ be a set. A $k$-colouring  \emph{template} of~$V$ is a function~$t: \ V\to 2^{[k]}\setminus \{\emptyset\}$, while a $k$-colouring of $V$ is a function~$c: \ V\to [k]$. We denote the set of all $k$-colouring templates of~$V$ and the set of all $k$-colourings of~$V$ by $\left(2^{[k]}\setminus\{\emptyset\} \right)^{V}$ and $[k]^V$ respectively.

	Given a template $t\in \left(2^{[k]}\setminus\{\emptyset\} \right)^{V}$, we write $\langle t \rangle$ for the collection of \emph{realisations} of $t$, that is, the collection of $k$-colourings~$c\in [k]^V$ such that $c(e)\in t(e)$ for every~$e\in V$. The \emph{entropy} of a $k$-colouring template~$t$ of~$V$ is 
	\[\Ent(t):= \sum_{e\in V}\log_k\vert t(e)\vert.\]
\end{definition}
Observe that $0\leq \Ent(t)\leq \vert V\vert $ and $\vert \langle t \rangle\vert =k^{\Ent(t)}$.
\begin{definition}[Extremal entropy relative to an ssee]\label{definition: extremal entropy/relative to possee}
	Let $\mathbf{V}=(V_n)_{n \in \N}$ be an ssee. A \emph{$k$-colouring property of $\mathbf{V}$} is a sequence $\mathcal{P}=(\mathcal{P}_n)_{n\in \N}$, where $\mathcal{P}_n$ is a collection of $k$-colourings of $V_n$. The \emph{extremal entropy} of $\mathcal{P}$ \emph{relative to $\mathbf{V}$} is
	\[\ex(\mathbf{V}, \mathcal{P})_n=\ex(V_n, \mathcal{P}_n):=\max\left\{\Ent(t) \,: \, t\in \left(2^{[k]}\setminus\{\emptyset\} \right)^{V_n}, \ \langle t \rangle\subseteq \mathcal{P}_n\right\}.\]
\end{definition}

\begin{definition}[Hereditary properties for an ssee]
	Let $\mathbf{V}=(V_n)_{n \in \N}$ be an ssee. Given an embedding~$\phi\in \binom{V_n}{V_N}$ and  a template $t\in \left(2^{[k]} \setminus \{\emptyset\}\right)^{V_n}$, we denote by $t_{\vert \phi}$ the $k$-colouring template for $V_N$ induced by $\phi$, 
	\[t_{\vert \phi} (x)= t(\phi(x)) \qquad \forall x \in V_N.\]
	
	A \emph{hereditary $k$-colouring property} for an ssee $\mathbf{V}$ is a $k$-colouring property $\mathcal{P}=\left(\mathcal{P}_n\right)_{n\in \mathbb{N}}$ such that for all $n\geq N$, $c\in \mathcal{P}_n$ and $\phi \in \binom{V_n}{V_N}$, we have $c_{\vert \phi} \in \mathcal{P}_N$.
\end{definition}
\begin{remark}
This is a common generalisation of the notion of hereditary and order-hereditary properties for graphs: by choosing one's embeddings appropriately when building an ssee from a graph sequence, we can encode either kind of property as an ssee hereditary property.
\end{remark}

We are now in a position to state what a `sufficiently rich' notion of substructure means.
\begin{definition}[Intersecting embeddings]
	Let $\mathbf{V}=(V_n)_{n \in \N}$ be an ssee. Let $N_1$,~$N_2 \leq n$.  An \emph{$i$-intersecting embedding} of~$(V_{N_1}, V_{N_2})$ into $V_n$ is a function $\phi: \ V_{N_1}\sqcup V_{N_2}\rightarrow V_n$ such that:
	\begin{enumerate}[(i)]
		\item the restriction of $\phi$ to  $V_{N_1}$ lies in $\binom{V_n}{V_{N_1}}$, and the restriction of $\phi$ to $V_{N_2}$ lies in $\binom{V_n}{V_{N_2}}$;
		\item $\vert \phi(V_{N_1})\cap \phi(V_{N_2})\vert=i$.
	\end{enumerate}
	We denote by $I_i\bigl( (V_{N_1}, V_{N_2}), V_n\bigr)$ the number of $i$-intersecting embeddings of $(V_{N_1}, V_{N_2})$ into $V_n$, and set
	\[I(N, n):=\sum_{1<i<\vert V_N\vert} I_i\bigl( (V_{N}, V_{N}), V_n\bigr).\]		
\end{definition}
\begin{definition}[Good ssee]\label{definition: good ossee}
	A ssee  $\mathbf{V}$ is \emph{good} if it satisfies the following conditions:
	\begin{enumerate}[(i)]
		\item $\vert V_n\vert \rightarrow \infty$ (`the sets in the sequence become large');
		\item for all $N \in \N$ with $\vert V_N\vert >1$,  $\bigl\vert \binom{V_n}{V_N} \bigr\vert \gg \vert V_n\vert $ (`on average, vertices in $V_n$ are contained in many embedded copies of $V_N$');
		\item for all $N \in \N$ with $\vert V_N\vert >1$, $\Bigl(\vert V_n\vert I(N,n)\Bigr)\Big/ \bigl\vert \binom{V_n}{V_N} \bigr\vert ^2 \rightarrow 0$ as $n\rightarrow \infty$ (`most pairs of embeddings of~$V_N$ into $V_n$ share at most one vertex').
	\end{enumerate}	
\end{definition}
\begin{remark}
	Condition (iii) can be interpreted as an `average co-degree condition' in a certain hypergraph, namely $H$ in the proof of Theorem~\ref{theorem: ossee container} below. Thus our `goodness' condition is related to the more usual `co-degree conditions' found in the the container theorems of Balogh--Morris--Samotij~\cite{BaloghMorrisSamotij15} and Saxton--Thomason~\cite{SaxtonThomason15}.
\end{remark}

Let $\mathbf{V}$ be an ssee. Given a collection $\mathcal{F}$ of $k$-colourings of $V_N$, denote by $\Forb_{\mathbf{V}}(\mathcal{F})$ the order-hereditary property of $k$-colourings of $\mathbf{V}$ not containing an embedding of a colouring in $\mathcal{F}$, i.e.
\[\Forb_{\mathbf{V}}(\mathcal{F})_n=\Bigl\{c\in [k]^{V_n}: \ \forall \phi \in \binom{V_n}{V_N}, \ c_{\vert \phi}\notin \mathcal{F}\Bigr\}.\]
Our main result in this subsection is that if $\mathbf{V}$ is a good ssee, then we have a container theorem for $\Forb_{\mathbf{V}}(\mathcal{F})$. As before, we say that a template family $\mathcal{T}_n$ is a \emph{container family} for a family of colourings $\mathcal{P}_n$ if for every $c\in \mathcal{P}_n$ there is $t\in \mathcal{T}_n$ with $c\in \langle t\rangle$.
\begin{theorem}\label{theorem: ossee container}
	Let $\mathbf{V}$ be a good ssee, and let $k$,~$N \in \N$. Let $\mathcal{F}$ be a nonempty collection of $k$-colourings of~$V_N$ and let\/ $\mathcal{P}=\Forb_{\mathbf{V}}(\mathcal{F})$.  For any $\varepsilon>0$, there exists $n_0>0$ such that for any $n\geq n_0$ there exists a collection\/  $\mathcal{T}_n$ of $k$-colouring templates for $V_n$ satisfying:
	\begin{enumerate}[(i)]
		\item $\mathcal{T}_n$ is a container family for\/ $\mathcal{P}_n$;
		\item for each template $t\in\mathcal{T}_n$, there are at most $\varepsilon \bigl\vert \binom{V_n}{V_N} \bigr\vert$ pairs $(\phi,c)$ with $\phi \in \binom{V_n}{V_N}$, $c\in \mathcal{F}$ and $c\in \langle t_{\vert \phi}\rangle$;
		\item $\vert \mathcal{T}_n\vert\leq  k^{\varepsilon \vert V_n\vert}$.
	\end{enumerate}	
\end{theorem}

\begin{proof}
	We follow in the main the proof of Theorem~\ref{theorem: multi-colour container}. Let $\mathbf{V}$, $k$, $N$, $\mathcal{F}$, and~$\mathcal{P}$ be as above. Fix $\varepsilon>0$.  We may assume without loss of generality that $\vert V_N\vert >1$, for otherwise $\mathcal{F}$ just gives us a list~$F$ of forbidden colours and the single template $t=\left([k]\setminus F\right)^{V_n}$ is a container for $\mathcal{P}_n$ lying entirely inside $\mathcal{P}_n$.

	First we modify the construction of the hypergraph~$H=H(\mathcal{F}, n)$ in the proof of Theorem~\ref{theorem: multi-colour container} as follows:
	\begin{itemize}
		\item we set $r=\vert V_N\vert$ (rather than $\binom{N}{2}$);
		\item we let $V(H)= V_n\times [k]$ (rather than $E(K_n)\times [k]$);
		\item for every  $\phi\in \binom{V_n}{V_N}$, and every colouring $c\in \mathcal{F}$, we add to $E(H)$ the $r$-edge
		\[e_{\phi, c}=\bigl\{\bigl(\phi(v), c(v)\bigr) \, : \, v\in V_N\bigr\}.\] 
	\end{itemize}
	As before, we bound $e(H)$; since $\mathcal{F}$ is nonempty, we have the following analogue of (\ref{equation: bounds on e(H)}):
	\begin{equation}\label{equation: bound on e(H), ossee}
	\biggl\vert \binom{V_n}{V_N} \biggr\vert \leq e(H) \leq k^{\vert V_N\vert } \biggl\vert \binom{V_n}{V_N} \biggr\vert.
	\end{equation}
	Just as before, our problem is that $H$ may be far from linear, so that we cannot apply Theorem~\ref{theorem: saxton thomason simple containers} directly. Here unlike in Theorem~\ref{theorem: multi-colour container} we have two cases to consider.

	Observe that $I(N,n)$ is exactly the number  of $r$-edges in $H$ which meet in at least two vertices, i.e. the `bad' pairs that make $H$ non-linear, henceforth referred to as overlapping pairs. If $I(N,n)\leq \varepsilon \bigl\vert \binom{V_n}{V_N} \bigr\vert/2$, then we can delete at most $\varepsilon e(H)/2$ $r$-edges from $H$ to make $H$ linear. This leaves us with a  linear $r$-graph with average degree 
	\[d\geq \vert V_N\vert \left(\biggl\vert \dbinom{V_n}{V_m}\biggr\vert-I(N,n)\right)/\vert V_n\vert \geq \vert V_N\vert (1-\varepsilon/2)\frac{\bigl\vert \binom{V_n}{V_N} \bigr\vert}{\vert V_n\vert},\]
	which by the goodness condition (ii) tends to infinity as $n\rightarrow \infty$. From there, Theorem~\ref{theorem: ossee container} follows easily from Theorem~\ref{theorem: saxton thomason simple containers} applied with parameter $\delta= \varepsilon\big/ 2k^{\vert V_N\vert}$: we obtain a collection $\mathcal{C}$ of sets which together cover all the independent sets in $H$ (property 1. of Theorem~\ref{theorem: saxton thomason simple containers}), each containing at most 
	\[\varepsilon \biggl\vert \binom{V_n}{V_N} \biggr\vert/2 +\delta e(H)\leq \varepsilon \biggl\vert \binom{V_n}{V_N} \biggr\vert\]
	$r$-edges (by property 2., our choice of $\delta$ and inequality (\ref{equation: bound on e(H), ossee})), with $\log_k \vert \mathcal{C}\vert \leq O\Bigl({\vert V_n\vert } d^{-\frac{1}{2\vert V_N\vert-1}}\Bigr)$ (property 3.). For $n$ sufficiently large, $\log_k \vert \mathcal{C}\vert$ is less than $\varepsilon \vert V_n \vert$ (since $d\gg1$). The family $\mathcal{C}$ then gives us our desired family of templates $\mathcal{T}_n$ here just as it did in the proof of Theorem~\ref{theorem: multi-colour container}.

	We therefore consider the more interesting case where $I(N,n)>\varepsilon \bigl\vert \binom{V_n}{V_N} \bigr\vert/2$ --- this is the case where on average embeddings of $V_N$ in $V_n$ are involved in $\Omega(1)$ overlapping pairs. Here we need a $\mathbf{V}$-analogue of our random sparsification lemma, Lemma~\ref{lemma: random sparsification}. The goodness of $\mathbf{V}$ is exactly what is needed for the proof to go through as before.

	Pick $\varepsilon_1>0$ sufficiently small so that 
	\begin{align}\label{equation: epsilon-good choice of epsilon1 in osse container thm}
24 \varepsilon_1 k^{\vert V_N\vert}<\varepsilon, \qquad \varepsilon_1<1/6
	\end{align}
	 and
	\begin{equation}\label{eq: general sparsification probability}
	p=\varepsilon_1\frac{\bigl\vert \binom{V_n}{V_N} \bigr\vert }{ I(N,n)}<1.
	\end{equation}
Keep each $r$-edge of~$H$ independently with probability~$p$, and delete it otherwise, to obtain a random subgraph~$H'$ of~$H$.
	\begin{lemma}\label{lemma: sparsification lemma, osse}
		Let $p$ be as in~\eqref{eq: general sparsification probability}, let $H'$ be the random subgraph of~$H$ defined above, and consider the following events:
		\begin{itemize}
			\item the event $F_1$ that $e(H')\geq \frac{p\bigl\vert \binom{V_n}{V_N}\bigr\vert}{2}$;
			\item the event $F_2$ that $H'$ has at most $3p^2 I(N,n)=3\varepsilon_1 p\bigl\vert \binom{V_n}{V_N} \bigr\vert$ pairs of $r$-edges $(e,e')$ with $\vert e\cap e'\vert \geq 2$;
			\item the event $F_3$ that for all $S\subseteq V(H)$ with $e(H[S])\geq 24 \varepsilon_1 e(H)$, we have $e(H'[S])\geq 12\varepsilon_1 e(H')$.
		\end{itemize}
		There exists $n_1\in\N$ such that for all $n\geq n_1$, $F_1 \cap F_2 \cap F_3$ occurs with strictly positive probability.
	\end{lemma}
	\begin{proof}
		We follow the proof of Lemma~\ref{lemma: random sparsification}. By (\ref{equation: bound on e(H), ossee}), the definition of $p$ and the goodness conditions (iii) and (i) for $\mathbf{V}$ we have
		\begin{equation}\label{equation: bounds on p (Vn choose VN)}
		pe(H)\geq p\biggl\vert \binom{V_n}{V_N} \biggr\vert = \varepsilon_1 \frac{\bigl\vert \binom{V_n}{V_N}\bigr\vert ^2}{I(N,n)}\gg \vert V_n\vert \gg 1.
		\end{equation}
		Together with the Chernoff bound~(\ref{equation: Chernoff}), inequality (\ref{equation: bounds on p (Vn choose VN)}) implies
		\begin{align}\label{equation: prob F1 fails, osse}
		\mathbb{P}(F_1 \textrm{ does not hold})& \leq 2 \exp\left (-p\biggl\vert \binom{V_n}{V_N} \biggr\vert/16\right) = o(1).
		\end{align}
By Markov's inequality applied to the number $Y_{H'}$ of pairs of $r$-edges $(e,e')$ with $\vert e\cap e'\vert \geq 2$ (i.e. the number of overlapping pairs in $H'$),
\begin{align}\label{equation: prob F2 fails, osse}
\mathbb{P}(F_2 \textrm{ does not hold})&=\mathbb{P}\left(Y_{H'}\geq 3\mathbb{E}Y_{H'}\right) \leq \frac{1}{3}.
\end{align}

Finally, consider a set $S\subseteq V(H)$ with $e(H[S])\geq 24\varepsilon_1 e(H)$. Applying the Chernoff bound~\eqref{equation: Chernoff} with $\delta=1-1/\sqrt{2}$ and the lower bound~\eqref{equation: bounds on p (Vn choose VN)} for $pe(H)$ we get
\begin{equation}\label{eq: too few induced edges, osse}
\mathbb{P}\left(e\bigl(H'[S]\bigr)\leq\frac{1}{\sqrt {2}}\mathbb{E}e\bigl(H'[S]\bigr)\right) \leq 2e^{-(1-\frac{1}{\sqrt{2}})\frac{\mathbb{E} e(H'[S])}{4}}=e^{-\Omega( p\varepsilon_1 e(H))}= e^{-\omega(\vert V_n\vert )}.
\end{equation}
Moreover, by \eqref{equation: Chernoff} and~\eqref{equation: bounds on p (Vn choose VN)} again,
\begin{equation}\label{eq: too many induced edges, osse}
\mathbb{P}\left(e(H')\ge\sqrt{2}\mathbb{E} e(H')\right)\leq 2e^{-(\sqrt{2}-1)^2\frac{p e(H)}{4}}= e^{-\omega(\vert V_n\vert )}.
\end{equation}
Say that a nonempty set~$S\subseteq V(H)$ is \emph{bad} if $e(H[S])\geq 24\varepsilon_1 e(H)$ and $e(H'[S])\leq 12\varepsilon_1 e(H')$.   By \eqref{eq: too few induced edges, osse},~\eqref{eq: too many induced edges, osse} and the union bound, the probability that $F_3$ fails, i.e., that there exists some bad $S\subseteq V(H)$, is at most
\begin{equation}\label{equation: prob F3 fails osse}
\mathbb{P}\bigl(\exists \textrm{ bad }S\bigr) \leq \mathbb{P}\bigl(e(H')\geq \sqrt{2}\mathbb{E} e(H')\bigr)+\sum_S \mathbb{P}\biggl(e(H'[S]\leq\frac{1}{\sqrt {2}}\mathbb{E}e\bigl(H'[S]\bigr)\biggr)\leq 2^{k\vert V_n \vert} e^{-\omega(\vert V_n\vert )}=o(1).
\end{equation}
Putting(\ref{equation: prob F1 fails, osse}), \eqref{equation: prob F2 fails, osse} and \eqref{equation: prob F3 fails osse} together we have that  $F_1$, $F_2$ and~$F_3$ hold simultaneously with probability at least~$2/3-o(1)$, which is strictly positive for $n$~sufficiently large.
	\end{proof}
With this sparsification lemma, we can now finish the proof in exactly the same way as we did in Theorem~\ref{theorem: multi-colour container}.

By Lemma~\ref{lemma: sparsification lemma, osse}, for any $\varepsilon>0$, any $\varepsilon_1>0$ satisfying \eqref{equation: epsilon-good choice of epsilon1 in osse container thm} and \eqref{eq: general sparsification probability}  and all $n$ sufficiently large, there exists a sparsification~$H'$ of~$H$ for which the events $F_1$, $F_2$ and~$F_3$ from the lemma all hold.  Deleting one $r$-edge from each overlapping pair in $H'$, we obtain a linear $r$-graph~$H''$ with average degree~$d$ satisfying
\begin{equation}\label{eq: linear hypergraph average degree, ossee}
d=\frac{re(H'')}{v(H'')} = \frac{\vert V_N\vert }{k\vert V_n\vert}\left(e(H') - Y_{H'}\right)
\geq  \frac{\vert V_N\vert }{k\vert V_n\vert} \left(\frac{1}{2}-3\varepsilon_1\right)p\biggl\vert \binom{V_n}{V_N} \biggr\vert \gg 1,
\end{equation}
where in the first inequality we used the fact that $F_1$ and $F_2$ hold, and in the last two inequalities we used the bounds $\varepsilon_1<1/6$ from \eqref{equation: epsilon-good choice of epsilon1 in osse container thm} and the lower bound on $p\binom{V_n}{V_N}$ from \eqref{equation: bounds on p (Vn choose VN)}.

Apply Theorem~\ref{theorem: saxton thomason simple containers} to $H''$ with parameter $\delta=6\varepsilon_1$ and let $d_0=d_0(\delta, r)$ be the constant in Theorem~\ref{theorem: saxton thomason simple containers}. Equation \eqref{eq: linear hypergraph average degree, ossee} tells us that for $n$ sufficiently large we have $d\geq d_0$. Thus there exists a collection~$\mathcal{C}$ of subsets of~$V(H'')=V(H)$ satisfying conclusions 1.--3.\ of Theorem~\ref{theorem: saxton thomason simple containers}. For each $C\in \mathcal{C}$, we obtain a template~$t=t(C)$ for a \emph{partial} $k$-colouring of~$V_n$, assigning to  each $v\in V_n$ a (possibly empty) palette $t(v)=\{i\in[k]: \ (v,i)\in C\}$ of available colours.
Set
\[\mathcal{T}:=\{t(C): \ C\in \mathcal{C}, \ t(v)\neq \emptyset  \textrm{ for all }v\in V_n  \}\]
to be the family of templates from $\left(2^{[k]}\setminus \{\emptyset\}\right)^{V_n}$ which can be constructed in this way. We claim that the template family~$\mathcal{T}$ satisfies the conclusions (i)--(iii) of Theorem~\ref{theorem: ossee container}.

Indeed, by definition of~$H$, any template $t'$ with $\langle t'\rangle\subseteq \mathcal{P}_n$ gives rise to an independent set $I$ in the $r$-graph~$H$ and hence its subgraph~$H''$, namely $I=\{(v,i): \ i\in [k], \ v\in V_n, \ i\in t(v)\}$. Thus there exists $C\in \mathcal{C}$ with $I\subseteq C$, giving rise to a proper template~$t \in \mathcal{T}$ with $t'\leq t$. Conclusion~(i) is therefore satisfied by $\mathcal{T}$.

Further for each $C\in \mathcal{C}$, conclusion 2.\ of Theorem~\ref{theorem: saxton thomason simple containers} and the event~$F_1$ and $F_2$ together imply
\begin{equation*}
e\bigl(H'[C]\bigr)\leq e\bigl(H''[C]\bigr) + \left(e(H')-e(H'')\right)< \delta e(H'') + 6\varepsilon_1e(H')= 12\varepsilon_1e(H').
\end{equation*}
Since $F_3$ holds in $H'$, this implies $e(H[C])<24 \varepsilon_1 e(H)$, which by our choice of $\varepsilon_1$ satisfying \eqref{equation: epsilon-good choice of epsilon1 in osse container thm} and our upper 
bound \eqref{equation: bound on e(H), ossee} on $e(H)$ is at most $\varepsilon \binom{V_n}{V_N}$. In particular, by construction of~$H$, we have that for each $t=t(C)\in \mathcal{T}$ there are at most~$\varepsilon \binom{n}{N}$ pairs~$(\phi, c)$ with $\phi\in \binom{V_n}{V_N}$, $c\in \mathcal{F}$ and $c\in \langle t_{\vert \phi}\rangle$. This establishes (ii).

Finally by conclusion 3.\ of Theorem~\ref{theorem: saxton thomason simple containers} and our bound (\ref{eq: linear hypergraph average degree, ossee}) on the average degree~$d$ in $H''$, we have
\begin{equation*}
 \vert \mathcal{T}\vert \leq \vert \mathcal{C}\vert \leq 2^{\beta(d) k\vert V_n\vert }=\exp\left(O\left(\left(\vert V_n\vert I(N,n)\big/ \biggl\vert \binom{V_n}{V_N} \biggr\vert^2\right)^{1/(2\vert V_N\vert -1)} \vert V_n\vert \right) \right)=k^{o(V_n)},
\end{equation*}
which means that (iii) is satisfied. This concludes the proof of the theorem.
\end{proof}

Theorem~\ref{theorem: ossee container} gives us container theorems for hereditary properties of $k$-colourings of ssee-s defined by a finite family of forbidden colouring. To obtain the standard counting applications of containers for a given ssee $\mathbf{V}$, we need two more ingredients, namely (a) the existence of an entropy density function for $\mathbf{V}$ (i.e.\ an analogue of Proposition~\ref{proposition: entropy density}) and (b) a supersaturation theorem for $\mathbf{V}$  (i.e.\ an analogue of Lemma~\ref{lemma: supersaturation}).

These ingredients are obtained on a more ad hoc basis than the general container theorem, Theorem~\ref{theorem: ossee container} --- the proofs have to be tailored to $\mathbf{V}$ to a greater extent ---  though in many of the most interesting cases the same arguments as those we used in Section~\ref{subsection: entropy density, supersaturation} will work with only trivial modifications. In Section~\ref{subsection: hypercubes} we shall illustrate this by giving a complete treatment of the case of hypercube graphs.

Provided we can obtain (a) and (b), we have as an immediate corollary of Theorem~\ref{theorem: ossee container} the following:
\begin{theorem}\label{theorem: counting for possee}
	Let $\mathbf{V}$ be a good ssee and let $k$,~$N \in \N$. Let $\mathcal{F}$ be a nonempty collection of $k$-colourings of~$V_N$ and let\/ $\mathcal{P}=\Forb_{\mathbf{V}}(\mathcal{F})$. Suppose that the following hold:
	\begin{enumerate}[(a)]
		\item $\pi(\mathcal{P}):=\lim_{n\rightarrow \infty}\ex(V_n, \mathcal{{P}})/\vert V_n\vert$ exists;
		\item for all $\varepsilon>0$ there exist $\delta>0, n_0\in \mathbb{N}$ such that if $n\geq n_0$ then every $t\in \left(2^{[k]}\setminus \{\emptyset\}\right)^{V_n}$ with at most $\delta \bigl\vert \binom{V_n}{V_N}\bigr\vert $ pairs $(\phi, c)$ with $\phi \in \binom{V_n}{V_N}$, $c\in \mathcal{F}$ and $c\in \langle t_{\vert \phi}\rangle$ must have entropy at most~$\mathrm{Ent}(t)\leq \left (\pi(\mathcal{P})+\varepsilon\right)\vert V_n\vert$. 
		\end{enumerate}	
	Then
	\[\vert \mathcal{P}_n\vert = k^{\bigl(\pi(\mathcal{P})+o(1)\bigr)\vert V_n\vert }.\]
\end{theorem}
\begin{proof}
	This is identical to the deduction of Theorem~\ref{theorem: counting result for Forb(F), F finite hereditary families} from Theorem~\ref{theorem: multi-colour container}, Proposition~\ref{proposition: entropy density} and Lemma~\ref{lemma: supersaturation}, by using Theorem~\ref{theorem: ossee container} and assumptions (a) and (b) to replace these three ingredients.
\end{proof}
To obtain counting results for general order-hereditary properties we need a little more.
\begin{theorem}\label{theorem: counting for arbitrary possee}
	Let $\mathbf{V}$ be a good ssee and let $k\in \mathbb{N}$. Suppose that the following hold:
	\begin{enumerate}[(a)]
		\item $\ex(V_n, \mathcal{P})/\vert V_n\vert$ is nonincreasing (and in particular tends to a limit $\pi(\mathcal{P})$) for all hereditary properties $\mathcal{P}$ of $k$-colourings of $\mathbf{V}$;
		\item for all $N$ and all nonempty families $\mathcal{F}\subseteq [k]^{V_N}$ we have supersaturation for $\Forb_{\mathbf{V}} (\mathcal{F})$:\\
		$\forall \varepsilon>0, \ \exists \delta>0, \ n_0\in \mathbb{N}$ such that if $n\geq n_0$ then for every $t\in \left(2^{[k]}\setminus \{\emptyset\}\right)^{V_n}$ with at most~$\delta \vert \binom{V_n}{V_N}\vert $ pairs $(\phi, c)$ with $\phi \in \binom{V_n}{V_N}$, $c\in \mathcal{F}$ and $c\in \langle t_{\vert \phi}\rangle$ we have $\mathrm{Ent}(t)\leq \left (\pi(\mathcal{P})+\varepsilon\right)\vert V_n\vert$. 
	\end{enumerate}	
	Then for any hereditary property $\mathcal{P}$ of $k$-colourings of $\mathbf{V}$,
	\[\vert \mathcal{P}_n\vert = k^{\bigl(\pi(\mathcal{P})+o(1)\bigr)\vert V_n\vert }.\]
\end{theorem}
\begin{proof}
The monotonicity in (a) and Theorem~\ref{theorem: counting for possee} allow us to apply the proof of Theorem~\ref{theorem: approximation of general hereditary properties} and obtain an ssee-version of our approximation of arbitrary hereditary $k$-colouring properties by properties defined by finite families of forbidden colourings.  We then deduce the claimed counting result in exactly the same way as we deduced Corollary~\ref{corollary: speed of arbitrary hereditary properties} from Theorems \ref{theorem: counting result for Forb(F), F finite hereditary families} and~\ref{theorem: approximation of general hereditary properties}.
\end{proof}
\begin{remark}
While we do not pursue this here, one can also use Theorem~\ref{theorem: ossee container} to derive stability and transference results by following the proofs of Theorems \ref{theorem: strong stability and containers} and~\ref{theorem: transference}. In the latter case, the lower bound on $p$ required for transference is 
\[ p \gg  \left(\frac{\vert V_n\vert I(N,n)}{\bigl\vert\binom{V_n}{V_N}\bigr\vert}\right)^{1/(2\vert V_N\vert -1)}\]  
provided $I(N,n)=\Omega( \binom{V_n}{V_N})$, as opposed to $p\gg n^{\-1/\bigl(2\binom{N}{2}-1\bigr)}$ when we were colouring $E(K_n)$ .
\end{remark}
\subsection{Graph and hypergraph sequences}\label{subsection: hypergraph sequences}
As an ssee is quite an abstract object, we feel it is helpful to clearly state the implications of Theorem~\ref{theorem: ossee container} in the language of graphs and hypergraphs. Recall from Example~\ref{example: graph edge ssee} that we may obtain an ssee $\mathbf{V}$ from a sequence of $l$-graphs $\Graphseq$ on linearly ordered vertex sets by taking $V_n=E(G_n)$ and taking as our set of embeddings $\binom{V_n}{V_N}$ all maps $\psi: \ E(G_N)\rightarrow E(G_n)$ arising from order-preserving embeddings $\phi$ from $G_N$ into $G_n$. We now restate the main definitions and results of the previous section (Definition~\ref{definition: good ossee}, Theorem~\ref{theorem: ossee container}) in terms of hypergraph sequences:
\begin{definition}[Good $l$-graph sequence]\label{definition: good graph sequence}
Let $l\geq 2$. An $l$-graph sequence $\Graphseq$ is \emph{good} if it satisfies the following conditions:
\begin{enumerate}[(i)]
	\item $e(G_n)\rightarrow \infty$ (`the graphs in the sequence become large');
	\item for all $N \in \N$ with $e(G_N)>1$, $\binom{G_n}{G_N}\gg e(G_n)$ as $n\rightarrow \infty$ (`on average, edges of $G_n$ are contained in many embeddings of $G_N$ ');
	\item for all $N \in \N$ with $e(G_N)>1$, $e(G_n)I(N,n)\big/ \binom{G_n}{G_N}^2 \rightarrow 0$ as $n\rightarrow \infty$ (`most pairs of edges in $G_n$ are contained in relatively few embeddings of $G_N$').
\end{enumerate}	
\end{definition}
\begin{theorem}\label{theorem: general multicolour container result for hypergraph sequences}
	Let $l\geq 2$ and let $\Graphseq$ be a good $l$-graph sequence. Let $k$,~$N \in \N$. Let $\mathcal{F}$ be a nonempty collection of $k$-colourings of $G_N$ and let\/ $\mathcal{P}=\Forb_{\mathbf{G}}(\mathcal{F})$.  For any $\varepsilon>0$, there exists $n_0>0$ such that for any $n\geq n_0$ there exists a collection~$\mathcal{T}_n$ of $k$-colouring templates for $G_n$ satisfying:
	\begin{enumerate}[(i)]
		\item $\mathcal{T}_n$ is a container family for\/ $\mathcal{P}_n$;
		\item for each template $t\in\mathcal{T}_n$, there are at most $\varepsilon \bigl\vert \binom{G_n}{G_N}\bigr\vert $ pairs $(\phi, c)$ with $\phi\in \binom{G_n}{G_N}$, $c\in \mathcal{F}$ and  $c\in \langle t_{\vert \phi}\rangle $;
		\item $\vert \mathcal{T}_n\vert\leq k^{\varepsilon e(G_n)}$.
	\end{enumerate}	
\end{theorem} 
In certain applications one maybe be interested in \emph{vertex}-colourings of $l$-graph sequences, rather than the edge-colourings considered above. As described in Example~\ref{example: graph vertex ssee}, we may obtain an ssee~$\mathbf{V}$ from a sequence of $l$-graphs $\Graphseq$ by taking $V_n=V(G_n)$ and taking as our collection of embeddings $\binom{G_n}{G_N}$ all (order-preserving) embeddings from $G_N$ to $G_n$. In this case, Definition~\ref{definition: possee} and Theorem~\ref{theorem: ossee container} become:
\begin{definition}[Vertex-good $l$-graph sequences]\label{def: hypergraph vertex-good}
	Let $l\geq 2$.  An $l$-graph sequence~$\Graphseq$ is \emph{vertex-good}  if it satisfies the following conditions:
	\begin{enumerate}[(i)]
		\item $v(G_n)\rightarrow \infty$ (`the graphs in the sequence become large');
		\item for all $N \in \N$ with $\vert V_N\vert >1$, $\binom{G_n}{G_N}\gg v(G_n) $ (`on average, vertices in $G_n$ are in many embeddings of $G_N$');
		\item for all $N \in \N$ with $\vert V_N\vert>1$, $v(G_n)I(N,n)\big/ \binom{G_n}{G_N}^2 \rightarrow 0$ as $n\rightarrow \infty$ (`most pairs of embeddings of~$G_N$ share at most one vertex').
	\end{enumerate}	
\end{definition}

\begin{theorem}\label{theorem: general multicolour container result for vertex colouring of hypergraph sequences}
	Let $l\geq 2$ and let $\Graphseq$ be a vertex-good $l$-graph sequence. Let $k$,~$N \in \N$. Let $\mathcal{F}$ be a nonempty collection of $k$-colourings of~$V(G_N)$ and let\/ $\mathcal{P}=\Forb_{\Graphseq}(\mathcal{F})$. 
	For any $\varepsilon>0$, there exists $n_0>0$ such that for any $n\geq n_0$ there exists a collection\/ $\mathcal{T}_n$ of $k$-colouring templates for $V(G_n)$ satisfying:
	\begin{enumerate}[(i)]
		\item $\mathcal{T}_n$ is a container family for\/ $\mathcal{P}_n$;
		\item for each template $t\in\mathcal{T}_n$, there are at most $\varepsilon \bigl\vert \binom{G_n}{G_N}\bigr\vert $ pairs $(\phi, c)$ with $\phi\in \binom{G_n}{G_N}$, $c\in \mathcal{F}$ and  $c\in \langle t_{\vert \phi}\rangle $;
		\item $\vert \mathcal{T}_n\vert\leq k^{\varepsilon v(G_n)}$.
	\end{enumerate}	
\end{theorem} 
Furthermore, as discussed in Section~\ref{subsection: possee}, and as more specifically stated in Theorem~\ref{theorem: counting for possee}, if $\mathcal{P}=\Forb_{\Graphseq}(\mathcal{F})$ admits an entropy density $\pi(\mathcal{P})$  (condition (a)) and the supersaturation property (condition (b)) then Theorems \ref{theorem: general multicolour container result for hypergraph sequences} and~\ref{theorem: general multicolour container result for vertex colouring of hypergraph sequences} immediately give a counting result, with $\vert \mathcal{P}_ n\vert = k^{( \pi(\mathcal{P}) +o(1) )e(G_n)}$ in the $k$-edge--colouring case and 
$\vert \mathcal{P}_ n\vert = k^{( \pi(\mathcal{P}) +o(1) )v(G_n)}$ in the $k$-vertex--colouring case.

However rather than state edge- and vertex-colouring versions of the counting result for ssee, Theorem~\ref{theorem: counting for possee}, we give an illustrative example in the next subsection by providing counting results for hereditary properties of edge- and vertex-colourings of hypercubes. This will hopefully make the situation clearer than an abstract theorem, avoid repetition, and in addition show how one goes about checking in practice that a sequence of graphs is good/vertex-good, has an entropy density function and satisfies the supersaturation property.

\subsection{Colourings of hypercube graphs}\label{subsection: hypercubes}
In this subsection we show the sequence $\Hypercube$ of hypercube graphs is good, has an entropy density function and satisfies the supersaturation property with respect to both edge-- and vertex--colourings. Using our results from Section~\ref{subsection: possee}, we immediately deduce for hereditary properties of edge-- and vertex--colourings of hypercubes.

In both settings, let $\binom{Q_n}{Q_N}$ denote the collection of injections~$\phi: \ V(Q_N)\to V(Q_n)$ obtained by  selecting an $N$-set $A=\{a_1, \ldots , a_N\}\subseteq [n]$ with $a_1<a_2<\cdots <a_N$ and a vector $\mathbf{v} \in \{0,1\}^{[n]\setminus A}$, and setting
\[\phi(\mathbf{x})_i=\left\{  \begin{array}{ll}\mathbf{v}_i & \textrm{if } i\notin B\\
\mathbf{x}_j & \textrm{if } i=a_j \in B.
\end{array}
 \right. \]
In other words, we have one embedding $\phi$ for each copy of $Q_N$ in $Q_n$. An edge/vertex--colouring property of hypercubes $\mathcal{P}$ is then called hereditary if for every $c \in \mathcal{P}_n$ and $\phi \in \binom{Q_n}{Q_N}$  we have $c_{\vert \phi}\in \mathcal{P}_N$.
\subsubsection{Edge-colourings of hypercubes}
\begin{proposition}[Goodness of hypercube graphs]\label{proposition: hypercube graphs are good}
The sequence $\Hypercube$ is good.
\end{proposition}
\begin{proof}
Let $N>1$.	We have $e(Q_n)=n2^{n-1}$ and $\bigl\vert \binom{Q_n}{Q_N}\bigr\vert=\binom{n}{N}2^{n-N}=\Omega(n^N2^n)\gg e(Q_n)$, establishing parts (i) and~(ii) of Definition~\ref{definition: good graph sequence}. For part~(iii), noting that two $N$-dimensional subcubes with at least two edges in common must meet in an $i$-dimensional subcube for some $i$ such that $2\leq i\leq N$, we have
	\begin{align*}
	I(N,n)&=\frac{1}{2}\biggl\vert \dbinom{Q_n}{Q_{n_0}}\biggr\vert \sum_{2\leq i\leq N} \binom{N}{i}2^{N-i}\binom{n-N}{N-i}=O\biggl(\biggl\vert \dbinom{Q_n}{Q_{n_0}}\biggr\vert n^{N-2}\biggr).
	\end{align*}
It follows that $I(N,n)e(Q_n)/{\bigl\vert \binom{Q_n}{Q_N}\bigr\vert ^2}=O(1/n)=o(1)$ as required. 
\end{proof}
\begin{proposition}[Entropy density for edge-colourings of hypercubes]\label{proposition: entropy density edge-hypercube}
If\/ $\mathcal{P}$ is a hereditary property of $k$-edge-colourings of~$\Hypercube$, then the sequence~$\ex(Q_n, \mathcal{P})/2^{n-1}n$ is nonincreasing and tends to a limit $\pi(\mathcal{P})$ as $n \to \infty$.	
\end{proposition}
\begin{proof}
Let $t$ be an extremal entropy template for $\mathcal{P}$ in $Q_{n+1}$. By averaging over all $\phi \in \binom{Q_{n+1}}{Q_n}$, we have
\begin{equation*}
n\ex(Q_{n+1}, \mathcal{P}) =n\Ent(t)=\sum_{\phi} \Ent\bigl(t_{\vert \phi}\bigr)\leq 2(n+1) \ex(Q_n, \mathcal{P}), 
\end{equation*}
whence $\ex(Q_n, \mathcal{P})/(2^{n-1}n)$ is nonincreasing in $[0,1]$, and hence tends to a limit as $n\rightarrow \infty$.
\end{proof}
\begin{proposition}[Supersaturation for edge--colourings of hypercubes]\label{proposition: supersaturation for hypercube graphs}
Let $N \in \N$ be fixed and let $\mathcal{F}$ be a nonempty collection of $k$-colourings of~$Q_N$. Set\/ $\mathcal{P}=\Forb_{\Hypercube}(\mathcal{F})$. For every $\varepsilon$ with $0 < \varepsilon < 1$, there exist constants $n_0\in\N$ and~$C_{0}>0$ such that for all $n\geq n_0$ and every template~$t\in \left(2^{[k]}\setminus \{\emptyset\}\right)^{Q_n}$ with
\[\Ent(t)> (\pi(\mathcal{P})+\varepsilon)e(Q_n),\] 
there are at least~
\[C_{0} \varepsilon\biggl\vert \dbinom{Q_n}{Q_{n_0}}\biggr\vert\] 
pairs~$(\phi, c)$ with $\phi\in \binom{Q_n}{Q_N}$,  $c \in \mathcal{F}$ and $c\in \langle t_{\vert \phi}\rangle$.
\end{proposition}
\begin{proof}
We follow the proof of Lemma~\ref{lemma: supersaturation}, modifying it as needed to fit the hypercube setting. Given a template   $t'\in \left(2^{[k]}\setminus \{\emptyset\}\right)^{Q_m}$, for some $m\geq N$, let  $B(t')$ denote the collection of pairs~$(\phi, c)$ with $\phi \in \binom{Q_m}{Q_N}$, $c\in \mathcal{F}$ and $c\in \langle {t'}_{\vert \phi}\rangle$.

By Proposition~\ref{proposition: entropy density edge-hypercube}, there exists $n_0\geq N$ such that 
for all $t'\in\left(2^{[k]}\setminus \{\emptyset\}\right)^{Q_{n_0}}$ with $\Ent(t')>\left(\pi(\mathcal{P}) +\frac{\varepsilon}{2}\right)e(Q_{n_0}) $, we must have $\vert B(t')\vert\geq 1$. Let $t\in \left(2^{[k]}\setminus\{\emptyset\}\right)^{Q_n}$ for some $n\geq n_0$, and suppose $\Ent(t)>\left(\pi(\mathcal{P})+\varepsilon\right)e(Q_n)$.  Let $X$ denote the number of embeddings $\phi\in \binom{Q_n}{Q_{n_0}}$ such that $\Ent(t_{\vert\phi})> \left(\pi(\mathcal{P}) +\frac{\varepsilon}{2}\right)e(Q_{n_0})$. By summing $\Ent(t_{\vert\phi})$ over all $\phi \in \binom{Q_n}{Q_{n_0}}$, we have
\begin{align*}
\left(\pi(\mathcal{P})+\varepsilon\right)e(Q_n) \binom{n-1}{n_0-1}&<\Ent(t)\binom{n-1}{n_0-1}=\sum_{ \phi} \Ent(t_{\vert \phi})\\
&\leq
\binom{n}{n_0}2^{n-n_0}\left(\pi(\mathcal{P})+\frac{\varepsilon}{2}\right)e(Q_{n_0}) +X e(Q_{n_0}),
\end{align*}
implying  $X> \frac{\varepsilon}{2}\binom{Q_n}{Q_{n_0}}$.
On the other hand, summing $\vert B(t_{\vert \phi})\vert $ over all $\phi\in \binom{Q_n}{Q_{n_0}}$ yields:
\begin{align*}
\vert B(t)\vert \binom{n-N}{n_0-N} = \sum_{\phi}\vert B(t_{\vert \phi})\vert \geq X>\frac{\varepsilon}{2}\biggl\vert \dbinom{Q_n}{Q_{n_0}}\biggr\vert,
\end{align*}
	so that 
	\begin{align*}
	\vert B(t)\vert > \frac{\varepsilon}{2} \frac{\bigl\vert \binom{Q_n}{Q_{n_0}}\bigr\vert}{\binom{n-N}{n_0-N} }=  \frac{1}{2\bigl\vert \binom{Q_{n_0}}{Q_N}\bigr\vert } \varepsilon\biggl\vert \dbinom{Q_n}{Q_N}\biggr\vert.
	\end{align*}
	This proves the lemma with $C_0= \bigl(2\bigl\vert \binom{Q_{n_0}}{Q_N}\bigr\vert\bigr)^{-1}$.
\end{proof}
\begin{corollary}[Counting for hypercube graph colourings]\label{corollary: edge hypercube counting}
If\/ $\mathcal{P}$ is a hereditary property of $k$-edge--colourings of~$\Hypercube$, then $\vert \mathcal{P}_n\vert=k^{\left(\pi(\mathcal{P})+o(1)\right)2^{n-1}n}$.
\end{corollary}
\begin{proof}
Propositions \ref{proposition: hypercube graphs are good}, \ref{proposition: entropy density edge-hypercube} and~\ref{proposition: supersaturation for hypercube graphs} tell us that the hypotheses of Theorem~\ref{theorem: counting for arbitrary possee} are satisfied; applying it yields the desired counting result.
\end{proof}

\subsubsection{Vertex--colourings of hypercubes}\label{subsection: vertex colouring hypercube}

\begin{proposition}[Vertex-goodness of hypercube graphs]\label{proposition: vertex-goodness of hypercubes}
The sequence $\Hypercube$ is vertex-good	
\end{proposition}
\begin{proof} Let $N>1$. We have $v(Q_n)=2^n\gg 1$ and $\vert \binom{Q_n}{Q_N}\vert =\binom{n}{N} 2^{n-N}\gg v(Q_n)$, so parts (i) and (ii) of Definition~\ref{def: hypergraph vertex-good} are satisfied. Part~(iii) is a simple calculation: two copies of $Q_N$ in $Q_n$ sharing at least two vertices must intersect in an $i$-dimensional subcube for some $i$ such that $1\leq i \leq N$. Thus
	\begin{equation*}
	\frac{v(Q_n) I(N,n)}{\bigl\vert \binom{Q_n}{Q_N}\bigr\vert^2} =\frac{2^{n-1}}{2^{n-N}\binom{n}{N}} \sum_{i\geq 1} \binom{N}{i}\binom{n-N}{N-i}2^{N-i}=O\Bigl(\frac{1}{n}\Bigr)
	\end{equation*}
which tends to $0$ as $n\rightarrow \infty$ as required.
\end{proof}
\begin{proposition}[Entropy density for vertex-colourings of hypercubes]\label{proposition: entropy density, vertex hypercube}
	If\/ $\mathcal{P}$ is a hereditary $k$--vertex-colouring property of hypercubes, then the sequence $\ex(Q_n, \mathcal{P})/ 2^n$ is nonincreasing and tends to a limit $\pi_v(\mathcal{P})$ as $n \rightarrow \infty$.
\end{proposition}
\begin{proof}
Let $t$ be a $k$-vertex--colouring template for $Q_{n+1}$ which is entropy-extremal for $\mathcal{P}$. By averaging over the $2(n+1)$ distinct embeddings $\phi \in \binom{Q_{n+1}}{Q_n}$,
\begin{align*}
	(n+1)\ex(Q_{n+1}, \mathcal{P})=(n+1)\Ent(t)
	=\sum_{\phi} \Ent\bigl(t_{\vert \phi }\bigr)\leq 2(n+1) \ex(Q_n, \mathcal{P}), 
	\end{align*}
	whence $\ex(Q_{n}, \mathcal{P})/2^{n}$ is non-increasing in $[0,1]$ and converges to a limit as required.
\end{proof}
\begin{proposition}[Supersaturation for vertex-colourings of hypercubes]\label{proposition: vertex hypercube supersaturation}
	Let $N\in \mathbb{N}$ be fixed and let $\mathcal{F}$ be a nonempty collection of $k$-vertex--colourings of $Q_N$. Let\/ $\mathcal{P}=\Forb_{\Hypercube}(\mathcal{F})$ be the hereditary $k$-vertex--colouring property of being $\mathcal{F}$-free.

	For every $\varepsilon$ with $0<\varepsilon<1$, there exist constants $n_0\in\N$ and~$C_{0}>0$ such that for all $n\geq n_0$ and every template~$t\in \left(2^{[k]}\setminus\{ \emptyset\}\right)^{V(Q_n)}$ with
\[
\Ent(t)> \left(\pi_v(\mathcal{P})+\varepsilon\right)2^n,
\]
there are at least
\[
C_{0} \varepsilon\biggl\vert \dbinom{Q_n}{Q_N}\biggr\vert
\]
pairs~$(\phi, c)$ with $\phi\in \binom{Q_n}{Q_N}$,  $c\in \mathcal{F}$ and $c\in \langle t_{\vert \phi}\rangle$.
\end{proposition}
\begin{proof}
	We follow the proof of Proposition~\ref{proposition: supersaturation for hypercube graphs}, modifying it as needed to fit the vertex-colouring setting. Given a template   $t'\in \left(2^{[k]}\setminus \{\emptyset\}\right)^{V(Q_m)}$ for some $m\geq N$, let  $B(t')$ denote the collection of pairs~$(\phi, c)$ with $\phi\in \binom{Q_m}{Q_N}$,  $c\in \mathcal{F}$ and $c\in \langle {t'}_{\vert \phi}\rangle$.

	By Proposition~\ref{proposition: entropy density, vertex hypercube}, there exists $n_0\geq N$ such that 
	for all $t'\in\left(2^{[k]}\setminus \{\emptyset\}\right)^{V(Q_{n_0})}$ with $\Ent(t')>\left(\pi(\mathcal{P}) +\frac{\varepsilon}{2}\right)v(Q_{n_0}) $, we must have $\vert B(t')\vert\geq 1$. Let $t\in \left(2^{[k]}\setminus\{\emptyset\}\right)^{V(Q_n)}$ for some $n\geq n_0$, and suppose $\Ent(t)>\left(\pi(\mathcal{P})+\varepsilon\right)v(Q_n)$.  Let $X$ denote the number of  $\phi\in \binom{Q_n}{Q_{n_0}}$ such that $\Ent(t_{\vert\phi})> \left(\pi(\mathcal{P}) +\frac{\varepsilon}{2}\right)v(Q_{n_0})$. By summing $\Ent(t_{\vert\phi})$ over all $\phi\in \binom{Q_n}{Q_{n_0}}$ we have
	\begin{align*}
	\left(\pi(\mathcal{P})+\varepsilon\right)v(Q_n) \binom{n}{n_0}&<\Ent(t)\binom{n}{n_0}=\sum_{ \phi} \Ent(t_{\vert \phi})
\leq
	\binom{n}{n_0}2^{n-n_0}\left(\pi(\mathcal{P})+\frac{\varepsilon}{2}\right)v(Q_{n_0}) +X v(Q_{n_0}),
	\end{align*}
	implying  $X> \frac{\varepsilon}{2}\binom{Q_n}{Q_{n_0}}$.
	On the other hand, summing $\vert B(t_{\vert \phi})\vert $ over all $\phi\in \binom{Q_n}{Q_{n_0}}$ yields:
	\begin{align*}
	\vert B(t)\vert \binom{n-N}{n_0-N} = \sum_{\phi}\vert B(t_{\vert \phi})\vert \geq X>\frac{\varepsilon}{2}\biggl\vert \dbinom{Q_n}{Q_N}\biggr\vert,
	\end{align*}
	so that 
	\begin{align*}
	\vert B(t)\vert > \frac{\varepsilon}{2} \frac{\bigl\vert \binom{Q_n}{Q_{n_0}}\bigr\vert}{\binom{n-N}{n_0-N} }=  \frac{1}{2\bigl\vert \binom{Q_{n_0}}{Q_N}\bigr\vert } \varepsilon\biggl\vert \dbinom{Q_n}{Q_N}\biggr\vert.
	\end{align*}
	This proves the lemma with $C_0= \bigl(2\bigl\vert \binom{Q_{n_0}}{Q_N}\bigr\vert\bigr)^{-1}$. 
\end{proof}
From there, the counting result is immediate:
\begin{corollary}\label{corollary: vertex hypercube counting}
If\/ $\mathcal{P}$ is a hereditary property of  $k$-vertex--colourings of $Q_n$, then 
\[\vert \mathcal{P}_n\vert = k^{\left(\pi_v(\mathcal{P})+o(1)\right)2^n}.\]	
\end{corollary}
\begin{proof}
Propositions \ref{proposition: vertex-goodness of hypercubes}, \ref{proposition: entropy density, vertex hypercube} and~\ref{proposition: vertex hypercube supersaturation} tell us that the hypotheses of Theorem~\ref{theorem: counting for arbitrary possee} are satisfied; applying it yields the desired counting result.
\end{proof}




	\section{Examples and applications}\label{section: examples}
	\subsection{Order-hereditary versus hereditary}
	Here we include a quick example stressing the essential difference between hereditary and order-hereditary properties. We identify graphs with $\{0,1\}$-colourings of~$K_n$ in the usual way; as the properties we consider in this example are in fact monotone, templates will consist of pairs~$e\in E(K_n)$ with $t(e)=\{0,1\}$ and entropy~$1$ and pairs~$e$ with $t(e)=\{0\}$ and entropy~$0$. We can thus represent the templates simply as the graph of edges with entropy~$1$.

	Let $\mathcal{P}_1$ be the hereditary property of graphs on $[n]$ of having maximum degree~$2$, and let $\mathcal{P}_2$ be the order-hereditary property of graphs on $[n]$ of not having any triples of vertices~$i<j<k$ with $ij$,~$jk$ both being edges. Clearly, $\mathcal{P}_1$ is the collection of all matchings on $[n]$, with $\ex(n, \mathcal{P}_1)=\lfloor n/2\rfloor$ and maximal matchings as the  extremal entropy templates.  The speed of this property is known as the Hosoya index or $Z$ index of~$K_n$, and is equal to the $n$th telephone number, which is of order~$n^{n/2 +o(n)}$~\cite{ChowlaHersteinMoore51}.	 On the other hand, $\ex(n, \mathcal{P}_2)=\lfloor \frac{n^2}{2}\rfloor$. For the lower bound, consider the template $t$ whose entropy~$1$ edges are $\{(i,j): \ 1\leq i \leq \frac{n}{2}\leq j\leq n, \ i\neq j\}$. Clearly $\Ent(t)=\lfloor n^2/4\rfloor$ and we have no $i<j<k$ with $ij$,~$jk$ both being edges. For the upper bound, suppose $\Ent(t)>n^2/4$. By Mantel's theorem, there must exist a triangle~$ijk$ of edges with full entropy, which gives us a triple of vertices~$i<j<k$ with $ij$,~$jk$ both being edges. Applying Theorem~\ref{theorem: counting result for Forb(F), F finite hereditary families}, we have that
	  \[\bigl\lvert (\mathcal{P}_2)_n\bigr\rvert  = 2^{\frac{1}{4}n^{2}\bigl(1+o(1)\bigr)}.\]

	\subsection{Graphs}\label{subsection: proof of alekseev--bollobas--thomason}

In this subsection, we give a short proof of the Alekseev--Bollob\'as--Thomason theorem (Theorem~\ref{theorem: alekseevbollobasthomason}).  Our argument is similar to the proof in~\cite{BT}.  However, container theory allows us to avoid using the Regularity Lemma, which simplifies the argument.

We shall use the Erd\H{o}s--Stone theorem~\cite{ErdosStone46}.

\begin{theorem}[Erd\H{o}s--Stone theorem]\label{theorem: erdosstone}
Let $r \geq 2$, $m \geq 1$ and $\varepsilon > 0$.  There exists $n_0(r, m, \varepsilon)$ such that if $G$ is a graph of order~$n \geq n_0$ and
\[
e(G) \geq \biggl(1 - \dfrac{1}{r} + \varepsilon \biggr)\binom{n}{2},
\]
then $G$ contains a copy of~$K_{r+1}(m)$.
\end{theorem}

Recall the family~$\mathcal{H}(r, \vector{v})$ from Definition~\ref{definition: colouring number} and observe that $\mathcal{H}(r, \vector{v})$ is a hereditary property of graphs. Write $\mathcal{H}(r, \vector{v})_l$ for the collection of all members of $\mathcal{H}(r, \vector{v})$ in which each of the $r$ parts contains at most $l$ vertices.

\begin{lemma}\label{lemma: universal}
Let\/ $\PP$ be a hereditary property of graphs, let $r \geq 2$, let $\ell \geq 1$ and let $\varepsilon > 0$.  There exists a constant~$n_0 \in \N$ such that if $n \geq n_0$ and $t\in (2^{[k]}\setminus \{\emptyset\} )^{K_n}$ with $\langle t \rangle \subseteq \PP_n$ and
\[
\Ent(t) \geq \biggl(1 - \dfrac{1}{r} + \varepsilon \biggr)\binom{n}{2},
\]
then $\mathcal{H}(r + 1, \vector{v})_{\ell} \subseteq \PP$ for some $\vector{v} \in \{0, 1\}^{r + 1}$.
\end{lemma}

\begin{proof} Recall that we may identify graphs with colourings of $E(K_n)$ by colours from $[2]$. Let $t$ be as above. Let $G$ be the graph with vertex set~$[n]$ and $E(G) = \{e \in E(K_n) : \ t(e) = \{1, 2\}\}$. By Ramsey's theorem, for every $\ell$, there exists $m$ such that any $2$-colouring of~$E(K_m)$ contains a monochromatic copy of~$K_{\ell}$.  Our assumption on $\Ent(t)$ and the Erd\H{o}s--Stone theorem imply that if $n$ is sufficiently large, then $G$ contains a copy~$K$ of~$K_{r + 1}(m)$.

Let $t'$ denote the restriction of~$t$ to $V(K)$ and let $V_1$, \dots,~$V_{r + 1}$ denote the classes of~$V(K)$.  Now we construct a vector~$\vector{v} \in \{0, 1\}^{r + 1}$.  By our choice of~$m$, for each $i$, either $\{e \in E(K[V_i]) : 1 \in t(e)\}$ or $\{e \in E(K[V_i]) : 2 \in t(e)\}$ contains a copy of~$K_{\ell}$.  In the former case set $v_i = 0$, and otherwise set $v_i = 1$.  In either case, we let $U_i$ denote the vertex set of the monochromatic copy of~$K_{\ell}$.

Let $H \in \mathcal{H}(r + 1, \vector{v})_{\ell}$ and let $W_1$, \dots~$W_{r + 1}$ be a partition of~$V(H)$ such that for each $i$, $W_i$ is a clique if $v_i = 1$ and an independent set if $v_i = 0$.  Because $\abs{W_i} \leq  \ell$ for each $i$, we may embed $W_i$ into $U_i \subseteq V_i$ arbitrarily.  It follows that there is a realisation~$c$ of~$t'$ such that $H$ is a subgraph of~(the graph corresponding to) $c(K_n)$. Since $H$ was arbitrary and $\PP$ is hereditary, it follows that $\mathcal{H}(r + 1, \vector{v})_{\ell} \subseteq \langle t' \rangle \subseteq \PP$, as desired.
\end{proof}

\begin{proof}[Proof of Theorem~\ref{theorem: alekseevbollobasthomason}.]
First, by the definition of~$\chi_c(\PP)$, there exists $\vector{v} \in \{0, 1\}^r$ such that $\mathcal{H}(r, \vector{v}) \subseteq \PP$.  By considering the graphs in $\mathcal{H}(r, \vector{v})$ such that each clique or independent set has size $\lfloor n/r \rfloor$ or~$\lceil n/r \rceil$, we see that
\[
\lvert \PP_n \rvert \geq \lvert \mathcal{H}(r, \vector{v})_n \rvert \geq 2^{\bigl(1 - 1/r + o(1)\bigr)\binom{n}{2}}.
\]

Second, suppose for a contradiction that for some $\varepsilon > 0$, there exist infinitely many $n$ such that
\begin{equation}\label{eq: P_n too big}
\lvert \PP_n \rvert \geq 2^{(1 - 1/r + \varepsilon)\binom{n}{2}}.
\end{equation}
Corollary~\ref{corollary: speed of arbitrary hereditary properties} implies that there exists $n_0$ such that for all $n \geq n_0$ for which~\eqref{eq: P_n too big} holds, there exists a template~$t\in (2^{[k]}\setminus \{\emptyset\})^{K_n}$ such that $\langle t \rangle \subseteq \PP_n$ and
\[
\Ent(t) \geq \biggl(1 - \dfrac{1}{r} + \dfrac{\varepsilon}{2} \biggr)\binom{n}{2}.
\]

It follows from Lemma~\ref{lemma: universal} that for each $\ell \geq 1$, there exists $\vector{v} \in \{0, 1\}^{r+1}$ such that $\mathcal{H}(r + 1, \vector{v})_{\ell} \subseteq \PP$.  In particular, there is some $\vector{v} \in \{0, 1\}^{r+1}$ such that $\mathcal{H}(r + 1, \vector{v})_{\ell} \subseteq \PP$ for infinitely many $\ell$, and thus for all $\ell$.  However, this contradicts the definition of~$\chi_c(\PP)$. The theorem follows.
\end{proof}

	\subsection{Digraphs}\label{subsection: digraph examples}

	As mentioned earlier, hereditary properties for tournaments, orgraphs and digraphs have received significant attention from the extremal combinatorics research community, see~\cite{Bollobas07}. In a recent paper, K\"uhn, Osthus, Townsend, Zhao~\cite{KuhnOsthusTownsendZhao14} determined the typical structure of certain families of oriented and directed graphs.  As part of their argument, they proved a container theorem and, using it, a counting theorem for $H$-free orgraphs and $H$-free digraphs, where $H$ is a fixed orgraph with at least~two edges (Theorems~3.3 and Corollary~3.4 in~\cite{KuhnOsthusTownsendZhao14}).

	They went on to observe that their results did not extend to the case where $H$ is a digraph, giving the specific example when $H=DK_3$, the double triangle $([3], [3]\times[3])$. Their approach considered the extremal weight achievable in an $H$-free digraph where double edges receive a different weight from single edges. In the case of $DK_3$, they observed that the extremal weight did not predict the correct count of $DK_3$-free digraphs, showing that their container theorem failed to generalise in its given form to the digraph case. Giving some vindication to our entropy-based approach to containers, we use our theorems to determine the speed of the digraph property~$\Forb(DK_3)$ of not containing any $DK_3$. More generally, given a graph~$F$, let $DF$ be the digraph obtained by replacing each edge of~$F$ with a directed edge in each direction. 
\begin{theorem}\label{theorem: extremal entropy, no F}  Let\/ $\PP=\Forb(DF)$ be the digraph property of not containing any $DF$. Then
	\[\ex(n, \mathcal{P})= (1-\log_4 3) \ex(n, F) + \log_43 \binom{n}{2},\]
	where $\ex(n,F)$ is the Tur\'an number of $F$.	
	\end{theorem}
	\begin{proof}
	We use the correspondence between digraphs and $4$-colourings of~$K_n$ from Observation~\ref{observation: encoding of digraphs}. Let $t$ be an $n$-vertex $4$-colouring template for $\mathcal{P}$ with maximal entropy. The monotonicity of $\mathcal{P}$ and the maximality of $\Ent(t)$ imply that all edges $e$ of~$K_n$ have $t(e)=[4]$ or $t(e)=[3]$. As $\mathcal{P}$ is exactly the property of having no copy of $F$ in colour $4$, at most $\ex(n, F)$ edges can have full entropy (entropy $1$), with the rest having entropy $\log_4 3$. This gives the upper bound on $\ex(n, \mathcal{P})$.

	For the lower bound, consider an $F$-free graph $G$ on $[n]$ with $e(G)=\ex(n,F)$. Let $t$ be the $4$-colouring template with $t(e)=[4]$ if $e\in E(G)$ and $t(e)=[3]$ otherwise. Clearly every realisation of $t$ contains no copy of $F$ in colour $4$, and hence lies in $\mathcal{P}$. The entropy of $t$ exactly matches the upper bound we established above, concluding the proof of the theorem.
	\end{proof}
	\begin{corollary} Let $F$ be a graph with chromatic number $r$. Then there are 
		\[\vert \Forb(DF)_n\vert = 3^{\frac{1}{r-1}\binom{n}{2}}4^{(1-\frac{1}{r-1})\binom{n}{2} +o(n^2) }\] 
		digraphs on $[n]$ not containing any copy of $DF$.
	\end{corollary}
	\begin{proof}
	The Erd{\H o}s--Stone--Simonovits theorem~\cite{ErdosSimonovits65, ErdosStone46}  implies that for every graph $F$  with chromatic number $r$, $\ex(n,F)=(1-\frac{1}{r-1})\binom{n}{2} +o(n^2)$. Together with theorem~\ref{theorem: extremal entropy, no F} this implies $\pi(\Forb(DF))= (1-\frac{1}{r-1})+\frac{1}{r-1}\log_43$. The result is then immediate from Corollary~\ref{corollary: speed of arbitrary hereditary properties}.
	\end{proof}
 Furthermore, we can characterise typical graphs in $\PP=\Forb(DF)$. Suppose $F$ has chromatic number $r$. Let $\mathcal{S}_n=\mathcal{S}_n(F)$ denote the collection of $t\in (2^{[4]}\setminus \{\emptyset\} )^{K_n}$ obtained by taking a balanced $(r-1)$-partition $\bigsqcup_{i=1}^{r-1}A_i$ of $[n]$ and setting $t(e)=[4]$ for all edges~$e$ between distinct parts $A_i$,~$A_j$ with $i\neq j$, and letting $t(e)=[3]$ for all other edges. The celebrated Erd{\H o}s--Simonovits stability theorem~\cite{Simonovits68} applied to the graph $F$ immediately implies the following result:
	\begin{proposition}\label{theorem: stability for no double triangle digraph}
	Let\/ $\PP=\Forb(DF)$ and $\mathcal{S}_n$ be as above.  For every $\varepsilon>0$, there exists $\delta>0$ and $n_0$ such that if $n\geq n_0$ and $t\in (2^{[4]}\setminus \{\emptyset\} )^{K_n}$ satisfies
	\begin{enumerate}[(i)]
	\item $\Ent(t)\geq \left(\pi(\mathcal{P})-\delta\right)\binom{n}{2}$, and
	\item there are at most $\delta \binom{n}{v(F)}$ monochromatic copies of $F$ in colour $4$  which can be found in the realisations of~$t$,
	\end{enumerate} 
	then $\rho(\mathcal{S}_n, t) \leq \varepsilon \binom{n}{2}$.\qed
	\end{proposition}
	Applying Theorem~\ref{theorem: strong stability and containers} then yields:
	\begin{corollary}
	Let\/ $\PP=\Forb(DF)$ and $\mathcal{S}_n$ be as above. For every $\varepsilon >0$ there exists $n_0>0$ such that for all $n\geq n_0$, all but $\varepsilon \vert \mathcal{P}_n\vert$ colourings in\/ $\mathcal{P}$ are within edit distance~$\varepsilon \binom{n}{2}$ of a realisation from\/ $\mathcal{S}_n$.

Equivalently, for all but an $\varepsilon$-proportion of $DF$-free digraphs~$D$ on $[n]$, there exists a digraph~$H$ that is obtained by taking a balanced $(r-1)$-partition $\bigsqcup_{i=1}^{r-1}A_i$ of $[n]$, setting double edges between distinct parts $V_i$,~$V_j$ and placing quasirandom tournaments inside each of the parts, and a subdigraph~$H'$ of~$H$ such that $\rho(D, H') \leq \varepsilon\binom{n}{2}$.\qed
	\end{corollary}
	The case $F=K_3$ (which has chromatic number $r=3$) in the results above resolves the problem identified by K\"uhn, Osthus, Townsend and Zhao.
	\subsection{Multigraphs}
In this subsection we study multigraphs with bounded edge multiplicities, viewed as weightings of the edges of~$K_n$ by non-negative integers. An $n$-vertex multigraph~$G$ in which all edge multiplicities are at most~$d$ can be encoded as a $(d + 1)$-colourings of~$E(K_n)$, with each edge coloured by its multiplicity. In this way, the problem of counting such multigraphs is placed in our framework of counting $k$-colourings.

		Let $\mathcal{P}$ be the property of multigraphs that no triple of vertices supports more than $4$ edges (counting multiplicities). Clearly no edge of such a multigraph can have weight more than $4$.  We shall determine the speed of $\mathcal{P}_n$. As always, we do this by first proving an extremal result (which in this case is quite easy), with the counting result then following immediately from an application of Corollary~\ref{corollary: speed of arbitrary hereditary properties}.

	Similar extremal problems for multigraphs were previously considered by Bondy and Tuza~\cite{BondyTuza97} and F\"uredi and K\"undgen~\cite{FurediKundgen02}. However the crucial difference is that, as far as counting results are concerned, we need to determine the asymptotically extremal entropy, rather than the asymptotically extremal total number of edges that was studied in~\cite{FurediKundgen02}. Indeed, in our problem, there exist configurations which are extremal with respect to the number of edges but \emph{not} with respect to entropy --- see Examples~\ref{example: digraph1} and \ref{example: digraph2} below.

Very recently in a pair of papers, Mubayi and Terry~\cite{MubayiTerry16a,MubayiTerry16b} study our problem in much greater generality, determining the extremal entropy, number and typical structure of multigraphs in which no $s$ vertices support more than~$q$ edges for a very large class of pairs~$(s, q)$. Our work in this subsection is thus a special case of their much more general results.

	\begin{example}\label{example: digraph1}
	Consider a balanced bipartition $V_1\sqcup V_2$ of $[n]$ and let $G_1$ be the multigraph assigning weight~$2$ to every edge from $V_1$ to $V_2$ and weight~$0$ to every other edge. Let also $t_1$ be the associated template, assigning colour list~$\{0,1,2\}$ to every edge from $V_1$ to $V_2$ and colour list~$\{0\}$ to every other edge.

Clearly $G_1\in \langle t_1 \rangle\subseteq \mathcal{P}$.  The total edge weight of~$G_1$ is $\lfloor \frac{n^2}{2}\rfloor$, and the entropy of~$t_1$ is $\log_5(3) \left\lfloor \frac{n^2}{4}\right\rfloor$. 
	\end{example}
It is not hard to show that the total edge weight of~$G_1$ is extremal; this is an easy exercise on proof by induction, and follows from results of Bondy--Tuza~\cite{BondyTuza97} and F\"uredi--K\"undgen~\cite{FurediKundgen02}. 
	\begin{proposition}\label{proposition: extremal size (3,5) problem}
	If $G$ is a multigraph in\/ $\mathcal{P}_n$, for some $n\geq 3$, then $e(G)\leq \lfloor \frac{n^2}{2}\rfloor$.
	\end{proposition}
 The total edge weight of $G_1$ is thus maximal; however, the entropy of the associated template~$t_1$ is not. Indeed we can construct a different edge-extremal construction with strictly larger entropy.
	 	\begin{example}\label{example: digraph2} Let $M$ be a maximal matching in $[n]$ and let $G_2$ be the multigraph assigning weight~$2$ to every edge in $M$ and weight~$1$ to every other edge. Let also $t_2$ be the associated template, assigning colour list~$\{0,1,2\}$ to every edge of~$M$ and colour list~$\{0,1\}$ to every other edge.

	 	As before, we have $G_2\in \langle t_2 \rangle\subseteq \mathcal{P}$ and $e(G_2)=\lfloor \frac{n^2}{2}\rfloor$. However,
	 	\[\Ent(t_2)=\log_5(2) \binom{n}{2}+\log_5\left(\frac{3}{2}\right)\left\lfloor \frac{n}{2}\right\rfloor=\log_5(\sqrt{2}) n^2+o(n^2) > \log_5(3^{1/4})n^2\geq \Ent(t_1).\]
	 	\end{example}
	 	It is straightforward to show that $t_2$ is indeed an entropy-extremal template for $\mathcal{P}$; this is a special case of recent and much more general results of Mubayi and Terry~\cite{MubayiTerry16b}.
	\begin{theorem}\label{theorem: extremal entropy (3,5) problem}
For all $n\geq 3$, $\ex(n, \mathcal{P})=\log_5(2) \binom{n}{2}+\log_5\left(\frac{3}{2}\right)\left\lfloor \frac{n}{2}\right\rfloor$.
	\end{theorem} 
			\begin{proof}
By induction on $n$. The base cases $n=3$,~$4$ are again easily checked by hand. For $n\geq 4$, consider a template~$t$ for a $5$-colouring of~$E(K_{n+1})$ with $\langle t \rangle\subseteq \mathcal{P}_{n+1}$, with colours from $\{0,1,2,3,4\}$ corresponding to edge weights. Suppose $\Ent(t)\geq \log_5(2) \binom{n+1}{2}+\log_5\left(\frac{3}{2}\right)\left\lfloor \frac{n+1}{2}\right\rfloor$. We claim that we must in fact have equality. By the inductive hypothesis it is enough to show that we can find a pair of vertices~$u_1u_2$ such that the sum of the entropies of the edges incident to $u_1$ or~$u_2$ is at most~$2(n-1)\log_5(2)+ \log_5(3)$. By monotonicity of the property~$\mathcal{P}$, we may assume that for every edge~$e$ if $i<j$ and $j\in t(e)$ then $i\in t(e)$. Thus the possible entropies for a single edge are $0$ (weight~zero), $\log_5 (2)$ (weight $0$ or~$1$),  $\log_5(3)$ (weight $0$, $1$ or~$2$), and so on.

			Suppose $G$ contains an edge~$u_1u_2$ with entropy at least~$\log_5(4)$. Then $3\in t(u_1u_2)$, and thus for every other vertex~$v$, the combined weight of $u_1v$,~$u_2v$ in any realisation of~$t$ must be at most~one, so that $\log_5 \vert t(u_1v)\vert +\log_5 \vert t(u_2v)\vert \leq \log_5(2)$. Thus the total entropy of the edges incident to $u_1$ or~$u_2$ is at most~$\log_5(5)+(n-1)\log_5(2)<2(n-1)\log_5(2)$. We may therefore assume that every edge~$u_1u_2$ has entropy at most~$\log_5(3)$ in $t$, and, given the bound we are trying to prove, that there is some edge with entropy exactly~$\log_5(3)$. Then $2\in t(u_1u_2)$, and for every other vertex~$v$ the pairs $u_1v$,~$u_2v$ can have combined weight at most~$2$ in every realisation of $t$. In particular,
\[\log_5 \vert t(u_1v)\vert +\log_5 \vert t(u_2v)\vert \leq \max\bigl\{ \log_5(3)+\log_5(1), \log_5(2)+\log_5(2)\bigr\}=2\log_5(2).\]
Thus the total entropy of the edges incident to $u_1$ or~$u_2$ is at most~$2(n-1)\log_5(2)+\log_5(3)$, as required, and
\[\Ent(t) \leq \log_5(2) \binom{n+1}{2}+\log_5\left(\frac{3}{2}\right)\left\lfloor \frac{n+1}{2}\right\rfloor. \qedhere\]
			\end{proof}
We may thereby deduce a counting result for $\mathcal{P}$: 	
	\begin{corollary}
		There are $2^{\binom{n}{2}+o(n^2)}$ multigraphs on $[n]$ for which no triple of vertices supports more than~$4$ edges (counting multiplicities).
	\end{corollary}
	\begin{proof}
		Immediate from Theorem~\ref{theorem: extremal entropy (3,5) problem} and Corollary~\ref{corollary: speed of arbitrary hereditary properties}.
	\end{proof}
\begin{remark}
With a little more work, it can be shown that $t_2$ and its isomorphic copies constitute a strong stability template for $\mathcal{P}$ and that typical members of~$\mathcal{P}$ are close to realisations of~$t_2$ --- and thus far from realisations of~$t_1$, despite the fact that $t_1$ was constructed from an edge-extremal graph. This also follows from considerably more general (and more difficult) stability results for multigraphs in which no $s$-set spans more than $q$ edges, which was obtained by Mubayi and Terry~\cite{MubayiTerry16b}, and shows how different the extremal problems for the total number of edges and for the entropy are in this setting.
\end{remark}

	\subsection{\texorpdfstring{$3$}{3}-coloured graphs}
	Let $\mathcal{P}$ denote the set of $3$-coloured graphs with no rainbow triangle, where a triangle is called \emph{rainbow} if it has an edge in each of the three colours~$\{1,2,3\}$. We use our multicolour container results to count the number of graphs in $\mathcal{P}$ and to characterise typical elements of~$\mathcal{P}$. This is related to the multicolour Erd{\H o}s--Rothschild problem~\cite{Erdos74}, which has received significant attention, see e.g.\ Alon, Balogh, Keevash and Sudakov's proof of a conjecture of Erd{\H o}s and Rothschild in~\cite{AlonBaloghKeevashSudakov04}, as well as the recent work of Benevides, Hoppen and Sampaio~\cite{BenevidesHoppenSampaio16}, Pikhurko, Staden and Yilma~\cite{PikhurkoStadenYilma16} and Hoppen, Lefmann and Odermann~\cite{HoppenLefmannOdermann}.

	\begin{theorem}[Extremal entropy]~\label{theorem: extremal entropy no rainbow k3}
Let\/ $\mathcal{P}$ denote the set of $3$-coloured graphs with no rainbow triangle.  For all $n\geq 3$, 
	\[\ex(n, \mathcal{P})=(\log_3 2)\binom{n}{2}.\]
	Furthermore, the unique extremal templates~$t$ are obtained by choosing a pair of colours~$\{c_1,c_2\}$ from $\{1,2,3\}$ and setting $t(e)=\{c_1,c_2\}$ for every $e\in E(K_n)$.
	\end{theorem}
	\begin{proof}
	Our theorem shall follow from the following observation and a straightforward averaging argument.
	\begin{observation}\label{observation: rainbow K3}
	Suppose $\langle t \rangle\subseteq \mathcal{P}$ and $e=\{v_1,v_2\}$ is some edge of~$K_n$. Then rainbow $K_3$-freeness implies the following:
	\begin{enumerate}[(i)]
	\item if $\vert t(e)\vert =3$, then for all $x\in V(K_n)\setminus e$ and $i \in \{1, 2\}$, we have $\vert t(xv_i)\vert =1$;
	\item if $\vert t(e)\vert =\vert t(f)\vert= 2$ and $t(e)\neq t(f)$, then $e\cap f = \emptyset$;
	\item if $\vert t(e)\vert =2$ and $c$ is the colour missing from $t(e)$, then for every $x\in V(K_n)\setminus e$, either $t(xv_1)=t(xv_2)=\{c\}$ or $c$ is missing from both $t(xv_1)$ and $t(xv_2)$. \qed
	\end{enumerate}
	\end{observation}
	In particular, for any $3$-set $A\subseteq [n]$, we have $\Ent(t\vert_A)\leq 3\log_32$, with equality attained if and only if all~three edges of $A$ are assigned the same pair of colours~$\{c_1,c_2\}$ by $t$.

	Now, suppose $t$ is a template with $\Ent(t)\geq (\log_3 2) \binom{n}{2}$. The average entropy of $t\vert_A$ over all $3$-sets~$A\subseteq [n]$ is:
	\begin{equation}
	\frac{1}{\binom{n}{3}} \sum_A \Ent\bigl(t_{\vert_A}\bigr) =\frac{1}{\binom{n}{3}} (n-2)\Ent(t)\geq 3\log_3 2. \label{equation: average entropy}
	\end{equation}
	Our previous bound on the entropy inside triangles then tells us that we must have equality everywhere in~\eqref{equation: average entropy} and that $t$ must have entropy $3\log_32$ inside \emph{every} $3$-set $A$. In particular, all edges~$e$ must have $\vert t(e)\vert=2$. Finally by (ii) in Observation~\ref{observation: rainbow K3}, there exists a pair of colours $\{c_1,c_2\}$ such that $t(e)=\{c_1,c_2\}$ for \emph{all} edges $e\in E(K_n)$. This concludes the proof of the theorem.
	\end{proof}
	\begin{corollary}[Counting]
	For all $\varepsilon>0$, there exists $n_0\in \N$ such that for all $n\geq n_0$,
	\[3\cdot3^{(\log_3 2)\binom{n}{2}} -3 \leq \vert \mathcal{P}_n\vert \leq 3^{(\log_3 2)\binom{n}{2}+\varepsilon \binom{n}{2}}. \]
	\end{corollary}
	\begin{proof}
	The lower bound is the number of colourings of~$E(K_n)$ such that each edge receives one of a prescribed pair of colours.  For the upper bound, Theorem~\ref{theorem: extremal entropy no rainbow k3} gives $\pi(\mathcal{P})=\log_32$, and the result then follows from Corollary~\ref{corollary: speed of arbitrary hereditary properties}.
	\end{proof}
We note that the stronger bound $\abs{\PP_n} \leq 3^{(\log_3 2)\binom{n}{2}+ O(n \log n)}$ was proved in~\cite{BenevidesHoppenSampaio16}. With a bit more case analysis, we can obtain the following stability result --- see the Appendix for a proof.  
	\begin{theorem}[Stability]\label{theorem: stability no rainbow K3}
	The family of templates $\mathcal{S}=\bigcup_n \{  \{1,2\}^{K_n}, \{1,3\}^{K_n}, \{2,3\}^{K_n}\}$ is a stability family for\/ $\mathcal{P}$.  That is, for all $\varepsilon>0$, there exist $\delta=\delta(\varepsilon)>0$ and $n_0=n_0(\delta)\in \N$ such that the following holds: if $t$ is a $3$-colouring template on $n\geq n_0$ vertices satisfying 
	\begin{enumerate}[(i)]
		\item $\Ent(t)\geq (\log_3 2  -\delta) \binom{n}{2}$, and
		\item there at most $\delta \binom{n}{3}$ rainbow triangles in $K_n$ which can be realised from $t$,
	\end{enumerate}
	then there exists a pair of colours $\{c_1,c_2\}\in [3]^{(2)}$ such that $t(e)=\{c_1,c_2\}$ for all but at most $\varepsilon \binom{n}{2}$ edges of~$K_n$. 
	\end{theorem}
	\begin{corollary}[Typical colourings]
	Almost all $3$-coloured graphs with no rainbow triangles are almost $2$-coloured: for every $\varepsilon>0$ there exists $n_0$ such that for all $n\geq n_0$ at most $3^{\varepsilon \binom{n}{2}}$ (rainbow $K_3$)-free $3$-colourings of~$K_n$ have at least $\varepsilon \binom{n}{2}$ edges in each of the colours $\{1,2,3\}$.
	\end{corollary}
	\begin{proof}
	Instant from Theorems \ref{theorem: strong stability and containers} and~\ref{theorem: stability no rainbow K3}.
	\end{proof}

	We note that there are (many) examples of rainbow $K_3$-free $3$-coloured graphs in which all three colours are used.  Indeed, consider a balanced bipartition $[n]=A\sqcup B$. Colour the edges from $A$ to $B$ Red, and then arbitrarily colour the edges internal to $A$ Red or Blue and the edges internal to $B$ Red or Green.  The resulting $3$-colouring has no rainbow $K_3$, and by randomly colouring the edges inside $A$ and~$B$ we can in fact ensure that all three colours are used on at least $(1+o(1))\frac{n^2}{16}$ edges.

\subsection{Hypercubes}
Let $\mathcal{P}=(\mathcal{P}_n)_{n\in \mathbb{N}}$ be the collection of all induced subgraphs of $Q_n$, $n\in \mathbb{N}$, with no copy of the square, or $4$-cycle, $Q_2$. Clearly, this may be viewed as a hereditary property of $2$-vertex-colourings of $Q_n$.
We have $\pi_v(\mathcal{P})\geq \frac{2}{3}$, as may be seen for example by removing every third layer of~$Q_n$, i.e.\ taking as our construction the family of all $\vector{x}\in Q_n$ with $\sum_i \vector{x}_i\not\cong 0 \mod 3$, which clearly contains no $Q_2$. Kostochka~\cite{Kostochka76} and, later and independently, Johnson and Entringer~\cite{JohnsonEntringer89} showed that this lower bound is tight: 
\[\pi_v(\mathcal{P})=\frac{2}{3}.\]
By Corollary~\ref{corollary: vertex hypercube counting} this immediately implies the following counting result:
\begin{corollary}
There are $\vert \mathcal{P}_n\vert = 2^{\left(\frac{2}{3}+o(1)\right)2^n}$ $Q_2$-free induced subgraphs of~$Q_n$.
\end{corollary}

In a different direction, let $\mathcal{Q}=(\mathcal{Q}_n)_{n\in \mathbb{N}}$ be the collection of all subgraphs of $Q_n$, $n\in \mathbb{N}$, with no copy of $Q_2$. This may be viewed as a hereditary property of $2$-edge--colourings of $Q_n$.
A long-standing conjecture of Erd{\H o}s~\cite{Erdos84} states that the edge-Tur\'an density (entropy density relative to $\Hypercube$) of this property is~$\pi(\mathcal{Q})=1/2$. The lower bound is obtained by deleting all edges between layer~$2i$ and layer~$2i+1$ for $0\leq i \leq \lfloor n/2\rfloor$. The best upper bound to date is~$0.603\ldots$ from applications of flag algebras due to Baber~\cite{Baber12} and Balogh, Hu, Lidick{\'y} and Liu~\cite{BaloghHuLidickyLiu14}. By Corollary~\ref{corollary: edge hypercube counting} we have the following:
\begin{corollary}
There are at most 	
	\[\vert \mathcal{Q}_n\vert =2^{\left(0.604+o(1)\right)2^{n-1}n}\]
	$Q_2$-free subgraphs of $Q_n$. Further, if Erd{\H o}s's conjecture on $\pi(\mathcal{Q})$ is true, then there are
\[\vert \mathcal{Q}_n\vert =2^{\left(\frac{1}{2}+o(1)\right)2^{n-1}n}\]
$Q_2$-free subgraphs of $Q_n$.
\end{corollary}
\subsection{A non-example: a graph sequence with too few embeddings} \label{subsection: nonexample containers fail for paths}
Let $\Path=(P_n)_{n\in\N}$ be the sequence of paths on $[n]$ introduced in Section~\ref{subsection: possee}. An easy calculation reveals that $\Path$ fails to satisfy the `goodness' condition introduced in Definition~\ref{definition: good graph sequence}, and is therefore not covered by Theorem~\ref{theorem: general multicolour container result for hypergraph sequences}.There is a good reason for this: the conclusion Theorem~\ref{theorem: general multicolour container result for hypergraph sequences} does not hold for $\Path$ (or, more generally, for sequences of `tree-like' graphs).

Let $\mathcal{P}$ be the order-hereditary property of $3$-colourings of $\Path$ of not having two consecutive edges in the same colour. It is easy to see that $\vert \mathcal{P}_n\vert= 3\cdot 2^{n-2}=3^{n\log_3(2) -O(1)}$. On the other hand, the extremal entropy of $\mathcal{P}_n$ is only about $ n \log_3\sqrt{2}$.
\begin{theorem}
For any $n\geq 3$, $\ex(P_n, \mathcal{P})=\lceil (n-1)/2\rceil \log_32$.
\end{theorem}
\begin{proof}
If $f$ and $f'$ are consecutive edges and $t$ is a $3$-colouring template with $\langle t \rangle\subseteq \mathcal{P}$ then $t(f)\cap t(f')=\emptyset$, from which it follows that $\log_3(\vert t(f)\vert) + \log_3(\vert t(f')\vert)\leq \log_32$. Further there can be no edge~$f$ with $t(f)=[3]$, since otherwise we would have a realisation of $t$ with two consecutive edges of the same colour. Partitioning the path~$P_n$ into disjoint pairs of consecutive edges and at most~one single edge, we get $\Ent(t)\leq \lceil(n-1)/2\rceil \log_3(2)$ as desired. For the lower bound, consider the template~$t$ defined by setting $t(\{2i+1, 2i+2\})=[2]$ and $t(\{2i, 2i+1\})=\{3\}$ for $0\leq i\leq \lfloor (n-1)/2\rfloor$. This has the correct entropy and all of its realisations clearly lie in $\mathcal{P}$. 
\end{proof} 
Now, $\binom{P_n}{P_3}=n-3$, and it is easy to see that we have supersaturation of sorts for $\mathcal{P}$: if $t$ is a template with $\Ent(t)\geq n \log_3(\sqrt{2})+\varepsilon n$, there are at least~$\Omega(\varepsilon n)=\Omega(\varepsilon\binom{P_n}{P_3})$ pairs of consecutive edges which can be made monochromatic in some realisation of~$t$. In particular, templates having $o(n)$ such pairs must have entropy at most~$\log_3(\sqrt{2}) n +o(n)$. A collection of $3^{o(n)}$ such templates can thus cover at most~$2^{n/2}3^{o(n)}=o(2^n)= o(\vert \mathcal{P}_n\vert)$ colourings---in particular, it cannot form a container family for $\mathcal{P}_n$. This shows that the analogue of Theorem~\ref{theorem: general multicolour container result for hypergraph sequences} does not hold for the graph sequence $\Path$, and that some form of the `goodness' assumption in the statement of that theorem is necessary, as we claimed.

	\section{Concluding remarks}\label{section: concluding remarks}
\subsection{Entropy maximisation in the multicolour setting}
In the $2$-colour setting, the rough structure of entropy maximisers for hereditary properties is well-understood, via the choice number~$\chi_c$: given a hereditary property~$\PP$ with $\chi_c(\PP) = r$, partition the vertex sets into $r$ equal parts and define a template by giving the $r$-partite edges full entropy (i.e.\ free choice of their colour) and the other edges zero entropy (i.e.\ fix their colour). 
In particular, Theorem~\ref{theorem: alekseevbollobasthomason} implies that the set of possible entropy densities for hereditary properties is $\{0, 1/2, 2/3, 3/4, \cdots\}\cup\{1\}$.

By contrast, it is less clear what the set of possible values of entropy densities or the possible rough structure of entropy maximisers should be in the $k$-coloured setting for $k\geq 3$.  We are only aware of one partial result in this area: Alekseev and Sorochan~\cite{AlekseevSorochan:colored} showed that if $\PP$ is a hereditary property of $k$-coloured graphs, then either $\pi(\PP) = 0$ or $\pi(\PP) \geq (1/2)\log_k (2)$.  Moreover, the examples in Section~\ref{section: examples} suggest that the possible structures of entropy maximisers are much more varied than in the case $k  = 2$.  

\begin{problem}\label{problem: possible entropy density in multicolour setting}
	Let $k\in \N$ with $k\geq 3$. Determine the set of possible entropy densities of hereditary properties of $k$-colourings of~$K_n$ and the rough structure of entropy maximisers.
\end{problem}
\subsection{Containers and the entropy of graph limits}\label{subsection: discussion of containers/limits}
In a forthcoming paper~\cite{FalgasRavryStrombergUzzell17} (see also~\cite{FalgasRavryOConnellStrombergUzzell16} for a preliminary version of these results), Johanna Str\"omberg and a subset of the authors of the present paper relate the container theorems to work of Hatami--Janson--Szegedy on the entropy of graph limits~\cite{HJS}. The multicolour container theorems in the present paper are used to obtain generalisations of Hatami--Janson--Szegedy's results to the setting of decorated graph limits. In the other direction, a second proof of those generalisations is obtained by working directly in the world of decorated graphons and using tools from analysis; it is further shown that these analytic results can then be used to recover some of the main combinatorial applications of containers, namely counting and characterisation (the case of transference is more delicate) for hereditary properties of multicoloured graphs. There thus appear to be significant links --- or at least similarities --- between the applications of the rich and currently quite distinct theories of graph limits and of hypergraph containers. The general `abstract' container results obtained in this paper and those obtained by Terry~\cite{Terry16} may thus be seen as first steps towards an elucidation of those links.

\section*{Acknowledgements}
We are extremely grateful to two anonymous referees for their careful work and scholarship. Not only did they help us greatly improve the readability and exposition of this paper, but they pointed us to some crucial references which we had missed, in particular the work of Balogh--Wagner and Mubayi--Terry, and corrected our misunderstandings of some recent work in the area, in particular the K{\L}R conjecture. Their comments on the notation and organisation of the paper were particularly appreciated, and led us to formulate a more general version of Theorem~\ref{theorem: ossee container} than in our initial manuscript.

Victor Falgas-Ravry is grateful for an AMS-Simons award which allowed him to invite Andrew Uzzell to visit him and Kelly O'Connell at Vanderbilt University in November~2015, when the last stages of this research were carried out, and to the Swedish Research Council (Vetenskapsr{\aa}det) for a grant supporting his research.  The authors would also like to thank Caroline Terry for helpful remarks about~\cite{Terry16} and Daniel Toundykov for Russian language assistance.

	\bibliographystyle{plain}
	\bibliography{containersgraphonbiblio}

\begin{thebibliography}{10}

\bibitem{Alekseev93}
V.~E. Alekseev.
\newblock On the entropy values of hereditary classes of graphs.
\newblock {\em Discrete Math. Appl.}, 3(2):191--200, 1993.

\bibitem{AlekseevSorochan:colored}
V.~E. Alekseev and S.~V. Sorochan.
\newblock On the entropy of hereditary classes of colored graphs.
\newblock {\em Diskret. Mat.}, 12(2):99--102, 2000.

\bibitem{ABBM11}
N.~Alon, J.~Balogh, B.~Bollob{\'a}s, and R.~Morris.
\newblock The structure of almost all graphs in a hereditary property.
\newblock {\em J. Combin. Theory Ser. B}, 101(2):85--110, 2011.

\bibitem{AlonBaloghKeevashSudakov04}
N.~Alon, J.~Balogh, P.~Keevash, and B.~Sudakov.
\newblock The number of edge colorings with no monochromatic cliques.
\newblock {\em J. London Math. Soc.}, 70(2):273--288, 2004.

\bibitem{Austin08}
T.~Austin.
\newblock On exchangeable random variables and the statistics of large graphs
  and hypergraphs.
\newblock {\em Probab. Surv.}, 5:80--145, 2008.

\bibitem{Baber12}
R.~Baber.
\newblock Tur{\'a}n densities of hypercubes.
\newblock Preprint, \url{http://arxiv.org/abs/1201.3587}, 2012.

\bibitem{BaloghBollobasMorris06}
J.~Balogh, B.~Bollob{\'a}s, and R.~Morris.
\newblock Hereditary properties of ordered graphs.
\newblock In {\em Topics in discrete mathematics}, pages 179--213. Springer,
  2006.

\bibitem{BaloghBollobasMorris07b}
J.~Balogh, B.~Bollob{\'a}s, and R.~Morris.
\newblock Hereditary properties of combinatorial structures: posets and
  oriented graphs.
\newblock {\em J. Graph Theory}, 56(4):311--332, 2007.

\bibitem{BaloghBollobasMorris07}
J.~Balogh, B.~Bollob{\'a}s, and R.~Morris.
\newblock Hereditary properties of tournaments.
\newblock {\em Electronic J. Combin.}, 14(3):R60, 2007.

\bibitem{BBS04}
J.~Balogh, B.~Bollob{\'a}s, and M.~Simonovits.
\newblock The number of graphs without forbidden subgraphs.
\newblock {\em J. Combin. Theory Ser. B}, 91(1):1--24, 2004.

\bibitem{BBS09}
J.~Balogh, B.~Bollob{\'a}s, and M.~Simonovits.
\newblock The typical structure of graphs without given excluded subgraphs.
\newblock {\em Random Structures Algorithms}, 34(3):305--318, 2009.

\bibitem{BaloghHuLidickyLiu14}
J.~Balogh, P.~Hu, B.~Lidick{\'y}, and H.~Liu.
\newblock Upper bounds on the size of 4-and 6-cycle-free subgraphs of the
  hypercube.
\newblock {\em European J. Combin.}, 35:75--85, 2014.

\bibitem{BaloghMorrisSamotij15}
J.~Balogh, R.~Morris, and W.~Samotij.
\newblock Independent sets in hypergraphs.
\newblock {\em J. Amer. Math. Soc.}, 28(3):669--709, 2015.

\bibitem{BaloghWagner16}
J.~Balogh and A.~Zs. Wagner.
\newblock Further applications of the container method.
\newblock In A.~Beveridge, R.~J. Griggs, L.~Hogben, G.~Musiker, and P.~Tetali,
  editors, {\em Recent Trends in Combinatorics}, volume 159 of {\em The IMA
  Volumes in Mathematics and its Applications}, pages 191--213. Springer
  International Publishing, 2016.

\bibitem{BenevidesHoppenSampaio16}
F.~S. Benevides, C.~Hoppen, and R.~M. Sampaio.
\newblock Edge-colorings of graphs avoiding complete graphs with a prescribed
  coloring.
\newblock {\em Discrete Math.}, 340(9):2143--2160, 2017.

\bibitem{Bollobas07}
B.~Bollob{\'a}s.
\newblock Hereditary and monotone properties of combinatorial structures.
\newblock In A.~Hilton and J.~Talbot, editors, {\em Surveys in combinatorics
  2007}, volume 346 of {\em London Math. Soc. Lecture Note Ser.}, pages 1--39.
  Cambridge Univ. Press, Cambridge, 2007.

\bibitem{BT}
B.~Bollob{\'a}s and A.~Thomason.
\newblock Hereditary and monotone properties of graphs.
\newblock In R.~L. Graham and J.~Ne\v{s}et\v{r}il, editors, {\em The
  {M}athematics of {P}aul {E}rd{\"o}s {II}}, volume~14 of {\em Algorithms
  Combin.}, pages 70--78. Springer Berlin Heidelberg, 1997.

\bibitem{BondyTuza97}
J.~A. Bondy and Zs. Tuza.
\newblock A weighted generalization of {T}ur{\'a}n's theorem.
\newblock {\em J. Graph Theory}, 25(4):267--275, 1997.

\bibitem{ChowlaHersteinMoore51}
S.~Chowla, I.~N. Herstein, and W.~K. Moore.
\newblock On recursions connected with symmetric groups. {I}.
\newblock {\em Canadian J. Math.}, 3:328--334, 1951.

\bibitem{Conlon14}
D.~Conlon.
\newblock Combinatorial theorems relative to a random set.
\newblock {\em Proceedings of the International Congress of Mathematicians,
  Seoul 2014}, to appear.

\bibitem{ConlonGowers10}
D.~Conlon and W.~T. Gowers.
\newblock Combinatorial theorems in sparse random sets.
\newblock {\em Ann. Math.}, 184(2):367--454, 2016.

\bibitem{ConlonGowersSamotijSchacht14}
D.~Conlon, W.~T. Gowers, W.~Samotij, and M.~Schacht.
\newblock On the {K{\L}R} conjecture in random graphs.
\newblock {\em Israel J. Math.}, 203(1):535--580, 2014.

\bibitem{DotsonNagle09}
R.~Dotson and B.~Nagle.
\newblock Hereditary properties of hypergraphs.
\newblock {\em J. Combin. Theory Ser. B}, 99(2):460--473, 2009.

\bibitem{Erdos74}
P.~Erd{\H o}s.
\newblock Some new applications of probability methods to combinatorial
  analysis and graph theory.
\newblock In {\em Proceedings of the {F}ifth {S}outheastern {C}onference on
  {C}ombinatorics, {G}raph {T}heory and {C}omputing ({F}lorida {A}tlantic
  {U}niv., {B}oca {R}aton, {F}la., 1974)}, pages 39--51. Congressus
  Numerantium, No. X. Utilitas Math., Winnipeg, Man., 1974.

\bibitem{Erdos84}
P.~Erd{\H o}s.
\newblock On some problems in graph theory, combinatorial analysis and
  combinatorial number theory.
\newblock {\em Graph Theory and Combinatorics (Cambridge, 1983), Academic
  Press, London}, pages 1--17, 1984.

\bibitem{ErdosFranklRodl86}
P.~Erd{\H o}s, P.~Frankl, and V.~R{\"o}dl.
\newblock The asymptotic number of graphs not containing a fixed subgraph and a
  problem for hypergraphs having no exponent.
\newblock {\em Graphs Combin.}, 2(1):113--121, 1986.

\bibitem{ErdosKleitmanRothschild76}
P.~Erd{\H{o}}s, D.~J. Kleitman, and B.~L. Rothschild.
\newblock Asymptotic enumeration of {$K_n$}-free graphs.
\newblock In {\em Internat. Colloq. Combin.} Atti Convegni Lincei (Rome), 1976.

\bibitem{ErdosSimonovits65}
P.~Erd{\H{o}}s and M.~Simonovits.
\newblock A limit theorem in graph theory.
\newblock {\em Studia Sci. Math. Hungar}, 1:51--57, 1966.

\bibitem{ErdosStone46}
P.~Erd{\"o}s and A.~H. Stone.
\newblock On the structure of linear graphs.
\newblock {\em Bull. Amer. Math. Soc.}, 52:1087--1091, 1946.

\bibitem{FalgasRavryOConnellStrombergUzzell16}
V.~Falgas-Ravry, K.~O'Connell, J.~Str{\"o}mberg, and A.~Uzzell.
\newblock Multicolour containers and the entropy of decorated graph limits.
\newblock 2016.
\newblock Manuscript, available on ArXiv:
  \url{http://arxiv.org/abs/1607.08152}.

\bibitem{FalgasRavryStrombergUzzell17}
V.~Falgas-Ravry, J.~Str{\"o}mberg, and A.~Uzzell.
\newblock Multicolour containers and the entropy of decorated graph limits.
\newblock 2017.
\newblock Preprint.

\bibitem{FurediKundgen02}
Z.~F{\"u}redi and A.~K{\"u}ndgen.
\newblock Tur{\'a}n problems for integer-weighted graphs.
\newblock {\em J. Graph Theory}, 40(4):195--225, 2002.

\bibitem{Galvin14}
D.~Galvin.
\newblock Three tutorial lectures on entropy and counting.
\newblock Preprint, \url{arxiv.org/abs/1406.7872}, 2014.

\bibitem{GreenTao08}
B.~Green and T.~Tao.
\newblock The primes contain arbitrarily long arithmetic progressions.
\newblock {\em Ann. Math. (2)}, 167(2):481--547, 2008.

\bibitem{HJS}
H.~Hatami, S.~Janson, and B.~Szegedy.
\newblock Graph properties, graph limits, and entropy.
\newblock {\em J. Graph Theory}.
\newblock To appear.

\bibitem{HoppenLefmannOdermann}
C.~Hoppen, H.~Lefmann, and K.~Odermann.
\newblock On graphs with a large number of edge-colorings avoiding a rainbow
  triangle.
\newblock {\em European J. Combin.}
\newblock To appear.

\bibitem{Ishigami07}
Y.~Ishigami.
\newblock The number of hypergraphs and colored hypergraphs with hereditary
  properties.
\newblock Preprint, \url{arxiv.org/abs/0712.0425}, 2007.

\bibitem{JohnsonEntringer89}
K.~A. Johnson and R.~Entringer.
\newblock Largest induced subgraphs of the {$n$}-cube that contain no
  {$4$}-cycles.
\newblock {\em J. Combinatorial Theory Ser. B}, 46(3):346--355, 1989.

\bibitem{KleitmanWinston82}
D.~J. Kleitman and K.~J. Winston.
\newblock On the number of graphs without {$4$}-cycles.
\newblock {\em Discrete Math.}, 41(2):167--172, 1982.

\bibitem{KLR}
Y.~Kohayakawa, T.~{\L}uczak, and V.~R{\"o}dl.
\newblock On {$K\sp 4$}-free subgraphs of random graphs.
\newblock {\em Combinatorica}, 17(2):173--213, 1997.

\bibitem{KohayakawaNagleRodl03}
Y.~Kohayakawa, B.~Nagle, and V.~R{\"o}dl.
\newblock Hereditary properties of triple systems.
\newblock {\em Combin. Probab. Comput.}, 12(2):155--189, 2003.

\bibitem{Kostochka76}
E.~A. Kostochka.
\newblock Piercing the edges of the n-dimensional unit cube.
\newblock {\em Diskret. Analiz Vyp}, 28(223):55--64, 1976.

\bibitem{KuhnOsthusTownsendZhao14}
D.~K{\"u}hn, D.~Osthus, T.~Townsend, and Y.~Zhao.
\newblock On the structure of oriented graphs and digraphs with forbidden
  tournaments or cycles.
\newblock {\em J. Combin. Theory Ser. B}, 124:88--127, 2017.

\bibitem{LovaszBook}
L.~Lov{\'a}sz.
\newblock {\em Large Networks and Graph Limits}, volume~60 of {\em Amer. Math.
  Soc. Colloq. Publ.}
\newblock Amer. Math. Soc., Providence, RI, 2012.

\bibitem{MubayiTerry16a}
D.~Mubayi and C.~Terry.
\newblock An extremal graph problem with a transcendental solution.
\newblock Preprint, \url{http://arxiv.org/abs/1607.07742}, 2016.

\bibitem{MubayiTerry16b}
D.~Mubayi and C.~Terry.
\newblock Extremal theory of locally sparse multigraphs.
\newblock Preprint, \url{http://arxiv.org/abs/1608.08948}, 2016.

\bibitem{NagleRodlSchacht06}
B.~Nagle, V.~R{\"o}dl, and M.~Schacht.
\newblock Extremal hypergraph problems and the regularity method.
\newblock {\em Topics in discrete mathematics}, pages 247--278, 2006.

\bibitem{PikhurkoStadenYilma16}
O.~Pikhurko, K.~Staden, and Z.~Yilma.
\newblock The {E}rd{\H o}s--{R}othschild problem on edge-colourings with
  forbidden monochromatic cliques.
\newblock {\em Math. Proc. Cambridge Philos. Soc.}, 163(2):341--356, 2017.

\bibitem{PS92}
H.~J. Pr{\"o}mel and A.~Steger.
\newblock Excluding induced subgraphs. {III}. {A} general asymptotic.
\newblock {\em Random Structures Algorithms}, 3(1):19--31, 1992.

\bibitem{Razborov07}
A.~A. Razborov.
\newblock Flag algebras.
\newblock {\em J. Symbolic Logic}, 72(4):1239--1282, 2007.

\bibitem{Sapozhenko87}
A.~A. Sapozhenko.
\newblock On the number of connected subsets with given cardinality of the
  boundary in bipartite graphs.
\newblock {\em Metody Diskret. Analiz.}, (45):42--70, 96, 1987.

\bibitem{Sapozhenko01}
A.~A. Sapozhenko.
\newblock On the number of independent sets in extenders.
\newblock {\em Diskret. Mat.}, 13(1):56--62, 2001.

\bibitem{SaxtonThomason15}
D.~Saxton and A.~Thomason.
\newblock Hypergraph containers.
\newblock {\em Invent. Math.}, 201(3):925--992, 2015.

\bibitem{SaxtonThomason16}
D.~Saxton and A.~Thomason.
\newblock Simple containers for simple hypergraphs.
\newblock {\em Combin. Probab. Comput.}, 25(3):448--459, 2016.

\bibitem{Schacht2009}
M.~Schacht.
\newblock Extremal results for random discrete structures.
\newblock {\em Ann. Math.}, 184(2):333--365, 2016.

\bibitem{Shannon48}
C.~Shannon.
\newblock A mathematical theory of communication.
\newblock {\em Bell System Tech. J.}, 27:379--423, 1948.

\bibitem{Simonovits68}
M.~Simonovits.
\newblock A method for solving extremal problems in graph theory, stability
  problems.
\newblock In {\em Theory of {G}raphs ({P}roc. {C}olloq., {T}ihany, 1966)},
  pages 279--319. Academic Press, New York, 1968.

\bibitem{Terry16}
C.~Terry.
\newblock Structure and enumeration theorems for hereditary properties in
  finite relational languages.
\newblock Preprint, \url{http://arxiv.org/abs/1607.04902}, 2016.

\bibitem{Terry17}
C.~Terry.
\newblock {$VC_ {\ell}$}-dimension and the jump to the fastest speed of a
  hereditary $\mathcal{L}$-property.
\newblock Preprint, \url{http://arxiv.org/abs/1701.00470}, 2017.

\end{thebibliography}
	\appendix
	\section{Appendix}

	\begin{proof}[Proof of Theorem~\ref{theorem: stability no rainbow K3}]
		Fix $\varepsilon>0$. Let $\delta>0$ and $n_0\in \N$ be sufficiently small and sufficiently large constants respectively, to be specified later. let $n\geq n_0$ and let $t\in (2^{[3]}\setminus \{\emptyset\})^{K_n}$ satisfy conditions (i) and~(ii) in the statement of the theorem. Our proof is a (lengthy) exercise in stability analysis --- essentially, we shall prove an approximate version of Observation~\ref{observation: rainbow K3}, and then run through the proof of Theorem~\ref{theorem: extremal entropy no rainbow k3} replacing each `for all pairs' by a `for almost all pairs'.

		By Lemma~\ref{lemma: supersaturation}, there exists an absolute constant $C_0=C_0(\mathcal{P})$ such that  for all $\eta>0$ there exists $n_1(\eta, \mathcal{P})$ such that for all $n\geq n_1$, if $t\in (2^{[3]}\setminus \{\emptyset\})^{K_n}$  can realise at most $\eta\binom{n}{3}$ rainbow triangles, then $\Ent(t)\leq \left(\pi(\mathcal{P})+ C_0\eta\right)\binom{n}{2}$.

		Let $e_3'$ be the number of edges $e=\{u,v\}\in E(K_n)$ for which there are at least~$\delta n$ vertices~$x\in V(K_n)\setminus\{u,v\}$ for which $\vert t(\{x,u\})\vert+ \vert t(\{x,v\})\vert >2$. For each such edge~$e$ and each such vertex~$x$, there is at least~one rainbow triangle which can be realised inside $e\cup\{x\}$. Each such triangle is counted at most~$3$ times, so that in total we must have at least $\frac{e_3'\delta n}{3}<\delta \binom{n}{3}$ rainbow triangles, and in particular we must have $e_3'<\frac{\delta}{2}n^2$.

		Now let $e_3''$ denote the number of edges $e=\{u,v\}$ for which there are at most $\delta n$ vertices $x$ with $\vert t(\{x,u\})\vert + \vert t(\{x,v\})\vert >2$. We shall choose $\delta$~sufficiently small to ensure that (a) $(1-200C_0\delta)^2>2/3$ and (b) $\delta <\frac{2\pi(\mathcal{P}) -1/50}{200(C_0 +2)}$ (we can certainly do that since the value of the constant~$C_0$ does not depend on $\delta$).

		Suppose $n>3n_1(2\delta)$. We claim that $e_3''<200 (C_0 +1)\delta n^2$. Indeed suppose not. Then we can find a set $E_3''$ of at least $200 (C_0 +1)\delta n^2/2n = 100 (C_0 +1)\delta n:=cn$ pairwise vertex-disjoint edges $e=\{u,v\}$ with $\vert t(e)\vert =3$ and $\vert t(\{x,u\})\vert+\vert t(\{x,v\})\vert=2$ for all but at most $\delta n$ vertices $x$. Remove from $K_n$ the pairs of vertices $e=\{u,v\}$ from $E_3''$ one by one. This leaves us with a graph on $n'= n -2c n$ vertices, which by (a) and our assumption on $n$ is strictly greater than~$n_1(2\delta)$.

		Let $t'$ denote the subtemplate of $t$ induced by the remaining vertices. Clearly $t'$ can realise at most $\delta\binom{n}{3}$ rainbow triangles, which by (a) is at most $2\delta \binom{n'}{3}$. Now Lemma~\ref{lemma: supersaturation} and the fact that $n'>n_1(2\delta)$ implies that 
		\begin{align}\label{equation: upper bound on entropy of t'}
		\Ent({t'})&\leq \left(\pi(\mathcal{P})+ C_02\delta\right)\binom{n'}{2}\leq \pi(\mathcal{P})\binom{n}{2}+\frac{C_02\delta}{2}n^2 -2c(1-c)\pi(\mathcal{P})n^2.
		\end{align} 
		On the other hand, each of the edges $e$ from $E_3''$ we removed decreased the entropy by at most $\delta n $, so we have the following lower bound on $\Ent(t')$:
		\begin{align}\label{equation: lower bound on entropy of t'}
		\Ent(t')&\geq \Ent(t)-c\delta n^2\geq\pi(\mathcal{P})\binom{n}{2}-(c\delta +\delta)\frac{n^2}{2}.
		\end{align}
		Bringing the two bounds (\ref{equation: upper bound on entropy of t'}) and (\ref{equation: lower bound on entropy of t'}) together and cancelling terms as appropriate, we get
		\[ -\frac{c\delta}{2} -\frac{\delta}{2}\leq \frac{C_02\delta}{2}-2c(1-c)\pi(\mathcal{P}).\]
		Rearranging yields
		\[ c\left(2(1-c)\pi(\mathcal{P})-\frac{\delta}{2}\right)\leq \frac{\delta}{2}(1+2C_0).\]
		Since $c=100(C_0+1)$, this contradicts our assumption (b) on $\delta $. It follows that $e_3''<200(C_0+1)\delta n^2$, as claimed. Thus in total, there are at most~$e_3'+e_3''=(\delta/2 +200(C_0+1)\delta)n^2:=C_2 \delta n^2$ edges~$e$ with $\vert t(e)\vert=3$.

		We now move on to bounding the number $e_1$ of edges $e$ with $\vert t(e)\vert =1$. We have
		\[\bigl(\pi(\mathcal{P})-\delta\bigr)\binom{n}{2}\leq \Ent(t)\leq \pi(\mathcal{P})\biggl(\binom{n}{2} -e_1\biggr) + e_3' + e_3'',\]
		which together with our bound on $e_3'+e_3''$ implies that
		\[e_1< \frac{1}{\pi(\mathcal{P})} \left(\frac{1}{2}+C_2\right)\delta n^2.\]
		In particular, all but at most 
		\[\left(\frac{1}{\pi(\mathcal{P})} \left(\frac{1}{2}+C_2\right) +C_2\right)\delta n^2:=C_3\delta n^2\] edges $e$ have $\vert t(e)\vert=2$.

		Finally we turn to the edges assigned two colours by $t$. For each pair of colours~$A\in [3]^{(2)}$, let $V_A$ denote the collection of vertices incident to at least~$\delta^{1/3} n$ edges that are assigned $A$ by $t$. For any $A\neq B$, each vertex in $V_A\cap V_B$ gives rise to at least~$\delta^{2/3} n^2$ distinct rainbow triangles, whence 
		\[\frac{\vert V_A\cap V_B\vert}{3}\leq \delta \binom{n}{3},\]
		implying $\vert A\cap V_B\vert \leq  \delta^{1/3} n/2$. Suppose we had $\vert V_A\vert$ and $\vert V_B\vert$ both greater than~$(\sqrt{C_3}+3)\delta^{1/3}n$ for some colour pairs~$A\neq B$, and let $C$ denote the third colour pair from $[3]$. Then all but at most~$\delta^{1/3} n$ vertices in $A$ are incident to at most~$2\delta^{1/3} n$ edges whose $t$-colour assignment is $B$ or~$C$. In particular such vertices~$a$ must be incident to at least $\vert V_B\vert -\vert V_B\cap V_A\vert -2\delta^{1/3} n$ edges~$ab$ with $b\in V_B\setminus V_A$ and $t(a,b)\notin \{A,B,C\}$. This gives at least
		\[ \left(\vert V_A\vert - \delta^{1/3} n\right) \left(\vert V_B\vert -\vert V_B\cap V_A\vert -2\delta^{1/3} n\right)\geq (\sqrt{C_3}+1)\delta^{1/3}n \sqrt{C_3} \delta^{1/3}n>C_3\delta n^2\]
		edges~$e$ with $\vert t(e)\vert\neq 2$, a contradiction. It follows that there is at most~one colour pair, say $A$, with $\vert V_A\vert\geq (\sqrt{C_3}+3)\delta^{1/3}n$. Let $B$,~$C$ denote the two other colour pairs, and $e_B$,~$e_C$ the number of edges~$e$ with $t(e)=B$ and $t(e)=C$ respectively. By the definition of $V_B$, we have 
		\[e_B\leq \vert V_B\vert n/2 + (n-\vert V_B\vert )\delta^{1/3}n/2< \frac{(\sqrt{C_3}+4)\delta^{1/3}}{2}n^2,\]
		and similarly $e_C \leq (\sqrt{C_3}+4)\delta^{1/3}n^2/2$. We have thus shown that all but at most~$(C_3\delta +(\sqrt{C_3}+4)\delta^{1/3})n^2$ edges $e\in E(K_n)$ have $t(e)\neq A$. Picking $\delta=\delta(\varepsilon)$~sufficiently small (and $n_0\geq 3n_1(2\delta)$), this is less than $\varepsilon \binom{n}{2}$, proving the theorem.
	\end{proof}
\end{document}